\documentclass[11pt]{amsart}
\usepackage{amsfonts, bm}
\usepackage{amssymb}
\usepackage{mathrsfs}
\usepackage{amsmath}
\allowdisplaybreaks[4]

\makeatletter

\renewcommand{\tocsection}[3]{%
  \indentlabel{\@ifnotempty{#2}{\bfseries\ignorespaces#1 #2\quad}}\bfseries#3}

\usepackage{amsmath,amsthm,amstext,amsopn,graphicx,cite,xcolor,enumitem}
\newtheorem{theorem}{Theorem}[section]
\newtheorem{lemma}[theorem]{Lemma}

\newtheorem{remark}[theorem]{Remark}

\theoremstyle{definition}

\theoremstyle{remark}

\numberwithin{equation}{section}
\newcommand\bes{\begin{eqnarray}}
\newcommand\ees{\end{eqnarray}}
\newcommand\bess{\begin{eqnarray*}}
	\newcommand\eess{\end{eqnarray*}}
\newcommand{\lf}{\left}
\newcommand{\rr}{\right}
\newcommand{\R}{{\mathbb R}}
\newcommand{\dd}{\displaystyle}

\newcommand{\td}{\tilde}
\newcommand{\wtd}{\widetilde}

\newcommand\yy{\infty}

\newcommand{\ol}{\overline}
\newcommand{\rd}{{\rm d}}

\newcommand{\prt}{\partial_t}

\newcommand{\ggs}{{\succ\!\!\succ}}
\newcommand{\llp}{{\prec\!\!\prec}}

\begin{document}
 
\date{\today}
\title[Semi-wave, traveling wave and spreading speed ]{Semi-wave, traveling wave and spreading speed for monostable cooperative systems with nonlocal \\ diffusion and free boundaries$^\$$}
\author[Y. Du and W. Ni]{Yihong Du$^\dag$ and Wenjie Ni$^\dag$}

\thanks{$^\dag$School of Science and Technology, University of New England, Armidale, NSW 2351, Australia}
	\thanks{$^\$$This research was supported by the Australian Research Council.}

\begin{abstract} We consider a class of cooperative reaction-diffusion systems with free boundaries in one space dimension, where the diffusion terms
are nonlocal, given by  integral operators involving suitable kernel functions, and they are allowed not to appear in some of the  equations in the system. The problem is monostable in nature, resembling the well known Fisher-KPP equation. Such a system covers various models arising from mathematical biology, with the Fisher-KPP equation as the simplest special case, where a ``spreading-vanishing" dichotomy is known to govern the long time dynamical behaviour. The question of spreading speed is  widely open for such systems except for the scalar case. In this paper, we  develop a systematic approach to determine the spreading profile of the system, and obtain threshold conditions on the kernel functions which decide exactly when the spreading has finite speed,  or infinite speed (accelerated spreading).  This relies on a rather complete understanding of both the associated semi-waves and traveling waves. When the spreading speed is finite, we show that the speed is determined by a particular semi-wave, and obtain sharp estimates of the semi-wave profile and the spreading speed. For  kernel functions that behave like $|x|^{-\gamma}$ near infinity, we are able to obtain better estimates of the spreading speed for both the finite speed case, and the infinite speed case, which appear to be the first for this kind of free boundary problems, even for the special Fisher-KPP equation.

\bigskip


\noindent
{\bf MSC2010 subject classifications:} 35K20, 35R35, 35R09.
\end{abstract}
\maketitle
\tableofcontents
\section{Introduction}

\subsection{The equations and assumptions}
In this paper we determine the spreading speed for cooperative systems with nonlocal diffusion and free boundaries of the following form:
\begin{equation}\label{1.1}
\begin{cases}
\prt u_i=d_i \mathcal{L}_i[u_i](t,x)+f_i(u_1,u_2,\cdots,u_m), & t>0, \; x\in (g(t), h(t)),\ 1\leq i\leq m_0,\\
\prt u_i=f_i(u_1,u_2,\cdots,u_m), & t>0, \; x\in (g(t), h(t)),\ m_0<i\leq m,\\
u_i(t,g(t))= u_i(t,h(t))=0, & t >0,\; 1\leq i\leq m,\\
\dd g'(t)= -\sum_{i=1}^{m_0}\mu_i\!\! \int_{g(t)}^{h(t)}\!\!\int_{-\yy}^{g(t)}\!\!J_i(x-y)u_i(t,x)\rd y\rd x, & t>0,\\[3mm]
\dd h'(t)= \sum_{i=1}^{m_0}\mu_i \!\!\int_{g(t)}^{h(t)}\!\!\int_{h(t)}^{\yy}\!\!J_i(x-y)u_i(t,x)\rd y\rd x, & t >0,\\
u_i(0,x)=u_{i0}(x), &x\in [-h_0,h_0],\ 1\leq i\leq m,
\end{cases}
\end{equation}
where $1\leq m_0\leq m$, and for $i\in \{1, ..., m_0\}$, 
\begin{align*}
 \mathcal{L}_i[v](t,x):=\int_{g(t)}^{h(t)}J_i(x-y)v(t,y)\rd y-v(t,x),\\
  d_i>0 \mbox{ and } \mu_i\geq 0 \mbox{ are constants,  with  $\dd\sum_{i=1}^{m_0}\mu_i>0$.}
\end{align*}
  The initial functions satisfy
\begin{align}\label{1.2}
u_{i0}\in C([-h_0,h_0]),\ \ u_{i0}(-h_0)=u_{i0}(h_0)=0,\ u_{i0}(x)>0\  \mbox{ in } (-h_0,h_0),\; 1\leq i\leq m.
\end{align}
The kernel functions $J_i(x)$ $(i=1,\cdots,m_0)$ satisfy 
\begin{itemize}
	\item[{\rm $(\mathbf{J})$}:] $ J_i\in C(\R)\cap L^\infty(\R)$ is nonnegative, even, $J_i(0)>0$,   $\displaystyle\int_\R J_i(x) \rd x=1$
	for $1\leq i\leq m_0$.
\end{itemize}

\bigskip

In order to describe the assumptions on   the function  $F=(f_1,..., f_m)\in [C^1(\R_+^m)]^m$  with 
\[
\R_+^m:=\{x=(x_1,..., x_m)\in \R^m: x_i\geq 0 \mbox{ for } i=1,..., m\},
\]
 and for convenience of later discussions,  we introduce some notations about vectors in $\R^m$, and also recall the definition of irreducible matrices and some associated properties.
\bigskip

\noindent
{\bf Notations about vectors in $\R^m$:} 
\begin{itemize}
	\item[{\rm (i)}] For $x=(x_1,\cdots,x_m)\in \R^m$, we simply write $(x_1,\cdots,x_m)$ as $(x_i)$. For $x=(x_i)$, $y=(y_i)\in \R^m$, 
	\begin{align*}
	&x\succeq (\preceq)\ y\ \ \ \ \ {\rm means}\ \ \  x_i\geq (\leq )\ y_i\ {\rm for}\ 1\leq i\leq m,\\
	&x\succ (\prec)\ y\ \ \ \ \ {\rm means}\ \ \  x\succeq (\preceq)\  y \ {\rm but }\ x\not=y,\\
	& x\; \ggs (\llp)\ y\ \ \  {\rm means}\ \ \   x_i>(<) \ y_i\ {\rm for}\ 1\leq i\leq m.
	\end{align*}  
	\item[{\rm (ii)}] If $x\preceq y$, then $[x,y]:=\{z\in\R^m: x\preceq z\preceq y\}$.
	\item[{\rm (iii)}]   \underline{Hadamard product}:  For $x=(x_i)$, $y=(y_i)\in \R^m$, 
	\begin{align*}
	x\circ y=(x_iy_i)\in\R^m.
	\end{align*}
	\item[{\rm (iv)}] Any $x\in \R^m$ is  viewed as a row vector, namely a $1\times m$ matrix,  whose transpose is denoted by $x^T$.
\end{itemize}
\smallskip

\noindent
{\bf Irreducible matrix and principal eigenvalue:}

 An $m\times m$ matrix $A=(a_{ij})$,
with $m\geq 2$, is called \underline{reducible} if the index set $\{1,..., m\}$ can be split into the union of  two subsets $S$ and $S'$ with $r\geq 1$ and $m-r\geq 1$ elements, respectively,  such
that $a_{ij}=0$ for all $i\in S$ and $j\in S'$. $A$ is called \underline{irreducible} if it is not reducible. 
If $D$ is a diagonal $m\times m$ matrix, clearly $A+D$ is irreducible if and only if $A$ is irreducible. 

If $A$ is irreducible and all its off-diagonal elements are nonnegative, then for $\sigma>0$ large $A+\sigma{\bf I}_m$ is a nonnegative irreducible matrix, where ${\bf I}_m$ denotes the $m\times m$ identity matrix. By the Perron-Frobenius theorem, $A+\sigma{\bf I}_m$ has a largest eigenvalue $\td \lambda_1=\td\lambda_1(\sigma)$ which is the only eigenvalue that corresponds to a positive eigenvector $v_1\ggs {\bf 0}$: $(A+\sigma{\bf I}_m)v_1=\td\lambda_1 v_1$.
Hence $Av_1=\lambda_1 v_1$ with $\lambda_1=\td \lambda_1-\sigma$, which is the largest eigenvalue of $A$ and is independent of $\sigma$. We will call $\lambda_1$ the \underline{principal eigenvalue} of $A$. 

A $1\times 1$ matrix is irreducible if and only if its sole element is not 0.
\medskip

\noindent
{\bf Assumptions on $F$:}

\begin{itemize}
\item[$\mathbf{(f_1)}$] 
	    \begin{itemize}
		\item[{\rm (i)}]   $F(u)={\bf 0}$  has only two roots in $\R^m_+$: $\mathbf{0}=(0,0,\cdots,0)$ and $\mathbf{u}^*=(u_1^*,u_2^*,\cdots,u_m^*)\ggs \mathbf{0}$. 
         \item[{\rm (ii)}] $\partial_j f_i(u)\geq 0$ for $i\neq j$ and $u\in [\mathbf{0}, \mathbf{\hat u}]$, where either $\mathbf{\hat u}=\mathbf{\infty}$ meaning $[\mathbf{0}, \mathbf{\hat u}]=\R^m_+$, or $\mathbf{u^*}\llp \mathbf{\hat u}\in\R^m$; which implies that \eqref{1.1} is a  \underline{cooperative system} in $[{\bf 0},\mathbf{\hat u}]$. 
         \item[{\rm (iii)}] The matrix $\nabla F(\mathbf{0})$ is irreducible with principal  eigenvalue positive, where $\nabla F({\bf 0})=(a_{ij})_{m\times m}$ with $a_{ij}=\partial_jf_i({\bf 0})$.
         \item[{\rm (iv)}] If $m_0<m$ then $\partial_jf_i(u)>0$ for $1\leq j\leq m_0<i\leq m$ and $ u\in[{\bf 0}, \mathbf{u}^*]$.
	     \end{itemize}
\item[$\mathbf{(f_2)}$]  $F(ku)\geq kF( u)$ for any  $0\leq k\leq 1$ and $ u\in [{\bf 0},\mathbf{\hat u}]$. 
		
\item[$\mathbf{(f_3)}$]  The matrix $\nabla F(\mathbf{u}^*)$ is invertible,    $ [\mathbf{u}^*]\nabla F(\mathbf{u}^*) \preceq \mathbf{0}$ and for each
$i\in\{1,..., m\}$, either 
\begin{itemize}
\item[{\rm (i)}] $\dd \sum_{j=1}^m \partial_jf_i(\mathbf{u^*})u^*_j<0$, or
\item[{\rm (ii)}] $\dd \sum_{j=1}^m \partial_jf_i(\mathbf{u^*})u^*_j=0$ and $f_i(u)$ is linear in $[\mathbf{u^*}-\epsilon_0\mathbf{1},\mathbf{u^*}]$
for some small $\epsilon_0>0$, where $\mathbf{1}=(1,..., 1)\in\R^m$.
\end{itemize}
\item[$\mathbf{(f_4)}$] The set $[\mathbf{0}, \mathbf{\hat u}]$ is invariant for 
\begin{align}\label{2.3}
U_t=D\circ\int_\R\mathbf{J} (x-y)\circ U(t,y) {\rm d}y-D\circ U+F(U) \mbox{ for } t>0,\; x\in \R, 
\end{align}
and the equilibrium ${\bf u^*}$ attracts all the nontrivial solutions 
in $[\mathbf{0}, \mathbf{\hat u}]$;
namely,  $U(t,x)\in [\mathbf{0}, \mathbf{\hat u}]$ for all $t>0, x\in\R$ if $U(0,x)\in [\mathbf{0}, \mathbf{\hat u}]$ for all $x\in\R$, and  $\lim_{t\to\infty} U(t,\cdot)={\bf u^*}$ in $L^\infty_{loc}(\R)$ if additionally $U(0,x)\not\equiv \mathbf{0}$.
\end{itemize}
\smallskip

In \eqref{2.3} we have used the convention that $d_i=0$ and $J_i\equiv 0$ for $m_0<i\leq m$, and
\[
D=(d_i),\; {\mathbf J}(x)=(J_i(x)).
\]
This convention will be used throughout the paper. 
\smallskip

The above assumptions on $F$ indicate that the system is cooperative in $[\mathbf{0},\mathbf{\hat u}]$, and of monostable type, with ${\bf u^*}$ the unique stable equilibrium of \eqref{2.3}, which is the global attractor of all the nontrivial nonnegative solutions of \eqref{2.3} in $[\mathbf{0},\mathbf{\hat u}]$. 
\smallskip

Problems \eqref{1.1} and \eqref{2.3} arise frequently in population and epidemic models. For example, if $m_0=m=1$, then \eqref{1.1} reduces to the Fisher-KPP model with free boundary considered in \cite{cdJFA, dlz}. With $m_0=m=2$, \eqref{1.1} contains the West Nile virus model in \cite{dn} as a special case, and with $(m_0,m)=(1,2)$, it covers the epidemic model in \cite{ZZLD}. In these special cases, it is known that the long-time dynamical behaviour of the 
solution to \eqref{1.1} exhibits a spreading-vanishing dichotomy. When spreading happens, one important question is to determine  the spreading speed. For the Fisher-KPP model in \cite{cdJFA}, it was shown recently in \cite{dlz} that the spreading speed may be finite or infinite, depending on whether a threshold condition is satisfied by the kernel function. For the epidemic models in \cite{dn} and \cite{ZZLD}, the spreading speed has not been determined. 

Problem \eqref{2.3} and its various special cases have been extensively studied in the literature; see, for example, \cite{AC, A-V, BZ, BCV, BGHP, cc2004, CDM, FT, G, HKLL, HMMV, LSW, SZ, WLS, XLL, XLR, yagi} for a very small sample. 

The purpose of this paper is to provide a unified treatment of  free boundary problems of the form \eqref{1.1}. Similar to the special cases mentioned in the last paragraph, it can be shown that \eqref{1.1} with initial data satisfying \eqref{1.2} and $U(0,x)\in [\mathbf{0},\mathbf{\hat u}]$ has a unique positive solution $(U(t,x), g(t), h(t))$ defined for all $t>0$. We say \underline{spreading} happens if, as $t\to\infty$,
\[
\mbox{ $(g(t), h(t))\to (-\infty,\infty)$ and $U(t,\cdot)\to {\bf u^*}$ component-wise in $L^\infty_{loc}(\R)$,}
\]
and we say \underline{vanishing} happens if
\[
\mbox{ $(g(t), h(t))\to (g_\infty, h_\infty)$ is a finite interval, and $\max_{x\in [g(t), h(t)]}|U(t,x)|\to 0$.}
\]

\subsection{Main results}
We aim to determine the spreading speed  when spreading happens for \eqref{1.1}. We will show that the spreading speed is finite \underline{if and only if} the following additional condition is satisfied by the kernel functions:
\begin{itemize}
	\item[{\rm $(\mathbf{J_1})$}:] \ \ \ \  $\dd \int^\infty_{0}\!\! x J_i(x) \rd x<\yy$ for every $i\in\{1,..., m_0\} $ such that $\mu_i>0$.
\end{itemize}
If $\mathbf{(J_1)}$ is not satisfied, then the spreading speed is infinite, namely \underline{accelerated spreading} happens. Let us note that if for some $i\in \{1,..., m_0\}$, $\mu_i=0$, then no restriction on $J_i$ is imposed by $\mathbf{(J_1)}$.

The proof of these conclusions rely on a complete understanding of the associated semi-wave problem to \eqref{1.1}, which consists of the following two equations \eqref{2.1a} and \eqref{2.2a} with unknowns $(c, \Phi(x))$:

\begin{equation}\label{2.1a}
\begin{cases}
\dd D\circ\int_{-\yy}^{0}\mathbf{J} (x-y)\circ \Phi(y) {\rm d}y-D\circ\Phi+c\Phi'(x)+F(\Phi(x))=0 \mbox{ for } -\yy<x<0,\\
\Phi(-\yy)=\mathbf{u}^*,\ \ \Phi(0)=\mathbf{0},
\end{cases}
\end{equation}
and 
\begin{align}\label{2.2a}
c= \sum_{i=1}^{m_0}\mu_i\int_{-\yy}^{0}\int_{0}^{\yy}J_i(x-y)\phi_i(x) {\rm d}y{\rm d}x,
\end{align}
where $D=(d_i)$, $\mathbf{J}=(J_i)$, $\Phi=(\phi_i)$  and ``$\circ$'' is the Hadamard product.
 
If $(c,\Phi)$ solves \eqref{2.1a}, we say that $\Phi$ is a \underline{semi-wave solution} to \eqref{2.3} with speed $c$. This is not to be confused with the semi-wave to \eqref{1.1}, for which the extra equation \eqref{2.2a} should be satisfied, yielding a semi-wave solution of \eqref{2.3} with a desired speed $c=c_0$, which determines the spreading speed of \eqref{1.1}.

To fully understand the semi-wave solutions to \eqref{2.3}, we also need to examine  the associated traveling wave problem for \eqref{2.3}.
A differentiable function $\Psi$ is called a \underline{traveling wave solution} of \eqref{2.3} with speed $c$ if $\Psi$  satisfies
\begin{equation}\label{TW}
\begin{cases}
\displaystyle D\circ\int_{-\yy}^{\yy}\mathbf{J} (x-y)\circ \Psi(y) {\rm d}y-D\circ\Psi+c\Psi'(x)+F(\Psi(x))={\bf 0} \mbox{ for } 
x\in\R,\\
\Psi(-\yy)=\mathbf{u}^*,\ \ \Psi(\yy)=\mathbf{0}.
\end{cases}
\end{equation}

We are interested in semi-waves and traveling waves which are \underline{monotone and with positive speed}.
It turns out that for any fixed speed $c>0$, either a semi-wave or traveling wave exists, but not both.
 More precisely, the following dichotomy holds:

\begin{theorem}\label{prop2.1}
	Suppose $\mathbf{(J)}$ and  $\mathbf{(f_1)-(f_4)}$  hold. Then	
	 there exists $C_*\in (0,\yy]$ such that 
		  \eqref{2.1a} has   a monotone solution  if and only if  $0<c<C_*$,
			and \eqref{TW} has a monotone solution if and only if  $c\geq C_*$.
			\end{theorem}
	Therefore a monotone traveling wave with some positive speed $c$ exists if and only if $C_*<\infty$. We will show that $C_*<\infty$ if and only if  the following condition is satisfied by the kernel functions:
	
	\begin{itemize}
	\item[{\rm $(\mathbf{J_2})$}:] \ \ \ \  $\dd 
	\int_0^\infty \!\! e^{\lambda x}J_i(x)dx <\yy\ \ \ {\rm for\ some }\ \lambda>0$ and every $i\in\{1,..., m_0\}$.
	\end{itemize}
	\smallskip
\noindent	

We have the following refinements of the conclusions in Theorem \ref{prop2.1}.

\begin{theorem}\label{prop2.3}
	Under  the conditions of Theorem  \ref{prop2.1}, the following hold:
	\begin{itemize}
		\item[{\rm (i)}] For $0<c<C_*$, \eqref{2.1a} has a unique monotone solution $\Phi^c=(\phi_i^c)$, and
		\begin{align*}
		\lim_{c\,\nearrow\, C_*} \Phi^c(x)={\bf 0} \ {\rm locally\ uniformly\ in}\ (-\yy,0].
		\end{align*}
		  \item[{\rm (ii)}] $C_*\not=\infty$ if and only if $\mathbf{(J_2)}$ holds.
		\item[{\rm (iii)}] The system \eqref{2.1a}-\eqref{2.2a} has a solution pair $(c,\Phi)$  with  $\Phi(x)$ monotone
		if and only if $\mathbf{(J_1)}$ holds. 
		And when $\mathbf{(J_1)}$ holds,  there exists a unique $c_0\in (0, C_*)$ such that $(c,\Phi)=(c_0,\Phi^{c_0})$ solves
		 \eqref{2.1a} and \eqref{2.2a}. 
	\end{itemize}
\end{theorem}

It is easily checked that $\mathbf{(J_2)}$ implies $\mathbf{(J_1)}$, but if $J_i(x)=\xi_i (1+|x|)^{-\eta_i}$ with $\xi_i>0$, $\eta_i>2$ for $1\leq i\leq m_0$, then $\mathbf{(J_1)}$ holds but $\mathbf{(J_2)}$ does not.
\smallskip

We are now able to state our first result on the spreading speed of \eqref{1.1}.

\begin{theorem}\label{theorem1.3}
	Suppose  the conditions in Theorem  \ref{prop2.1} are satisfied,
	   $(U,g,h)$ is a solution of \eqref{1.1} with $U(0,x)\in [\mathbf{0}, \mathbf{\hat u}]$, and spreading happens. Then the following conclusions hold for the spreading speed:
	  \begin{itemize}
	  	\item[{\rm (i)}] If  $\mathbf{(J_1)}$ is satisfied, then the spreading speed is finite, and is determined by
	  	 \begin{align*}
	  		-\lim_{t\to\yy}\frac{g(t)}{t}=\lim_{t\to\yy} \frac{h(t)}{t}=c_0 \mbox{ with $c_0$ given in Theorem \ref{prop2.3} (iii).}
	  		\end{align*}
	\item[{\rm (ii)}]  If $\mathbf{(J_1)}$ is not satisfied, then accelerated spreading happens, namely	
		  		  	\begin{align*}
	  		  	-\lim_{t\to\yy}\frac{g(t)}{t}=\lim_{t\to\yy} \frac{h(t)}{t}=\yy.
	  		  	\end{align*}
	  \end{itemize}
	 
\end{theorem}

Under further conditions on $F$ and the kernel functions, the conclusions in Theorem \ref{theorem1.3} can be sharpened. For $\alpha>0$, we introduce the condition

\begin{itemize}
	\item[{\rm $(\mathbf{J^\alpha})$}:] \ \ \ \  $\dd 
	\int_0^\infty \!\! x^\alpha J_i(x)dx <\yy$\ \ \ {\rm for\  every }  $i\in\{1,..., m_0\}$.
	\end{itemize}
Let us note that $\mathbf{(J^1)}$ implies $\mathbf{(J_1)}$, but unless $\mu_i>0$ for every $i\in\{1,..., m_0\}$, $\mathbf{(J_1)}$ does not imply $\mathbf{(J^1)}$. On the other hand, if $\mathbf{(J_2)}$ holds, then $\mathbf{(J^\alpha)}$ is satisfied for all $\alpha>0$.

\begin{theorem}\label{theorem1.4}
In Theorem \ref{theorem1.3}, suppose additionally $\mathbf{(J^\alpha)}$ holds for some $\alpha\geq 2$, $F$ is $C^2$ and   
  $\mathbf{u}^*[\nabla F(\mathbf{u}^*)]^T  \llp \mathbf{0}$.  Then there exist positive constants $\theta$, $C$ and $t_0$ such that, for all $t>t_0$
  and  $ x\in[g(t), h(t)]$,
\[
| h(t)-c_0t|+| g(t)+c_0t|\leq C,
\]
\[\begin{cases}
U(t,x)\succeq [1-\epsilon(t)]\big[\Phi^{c_0}(x-c_0t+C)+\Phi^{c_0}(-x-c_0t+C)-\mathbf{u}^*\big] ,\\
U(t,x)\preceq [1+\epsilon(t)]\min\big\{\Phi^{c_0}(x-c_0t-C),\; \Phi^{c_0}(-x-c_0t-C)\big\},
\end{cases} 
\]
where $\epsilon(t):=(t+\theta)^{-\alpha}$, and $(c_0, \Phi^{c_0})$ is the unique pair solving 
		 \eqref{2.1a} and \eqref{2.2a} obtained in Theorem \ref{prop2.3} (iii), with $\Phi^{c_0}(x)$ extended by {\bf 0} for $x>0$.
\end{theorem}

\medskip

{Further estimates on $g(t)$ and $h(t)$ can be obtained if we narrow down more on the class of kernel functions $\{J_i: i=1,..., m_0\}$. We will write
\begin{align*}
\eta(t)\approx \xi(t) \ {\rm\  if\ }\ C_1\xi(t)\leq \eta(t)\leq C_2\xi(t)
\end{align*}
for some positive constants $C_1\leq C_2$ and all $t$ is the concerned range.

Our next two theorems are about kernel functions satisfying, for some $\gamma>0$,
\smallskip
\begin{itemize}
	\item[{\rm $(\mathbf{\hat J^\gamma})$}:] \ \ \ \  $\dd 
	 J_i(x) \approx |x|^{-\gamma} $\ \ \ {\rm for\   $|x|\gg 1$ and  $i\in\{1,..., m_0\}$.}
	\end{itemize}\smallskip
	
Note that for kernel functions satisfying $(\mathbf{\hat J^\gamma})$, condition $(\mathbf{J})$ is satisfied if and only if $\gamma>1$, and
$\mathbf{(J_1)}$ is satisfied if and only if $\gamma>2$.
The next result determines the orders of accelerated spreading when $\gamma\in (1,2]$.

\begin{theorem}\label{th1.5a}
	In Theorem  \ref{theorem1.3}, if additionally the kernel functions satisfy $(\mathbf{\hat J^\gamma})$ for some $\gamma\in (1, 2]$, 	
 then for $t\gg 1$,
	\[\begin{array}{rll}
	-g(t),\; h(t)&\approx\  t\ln t\   &{\rm if}\ {\gamma}=2,\\
	-g(t),\; h(t)&\approx \ t^{1/({\gamma}-1)}\ & {\rm if}\ {\gamma}\in (1,2).
	\end{array}
	\]
\end{theorem}

 For kernel functions satisfying $(\mathbf{\hat J^\gamma})$, clearly $(\mathbf{J^\alpha})$ holds if and only if $\gamma>1+\alpha$. Therefore the case  $\gamma>3$ is already covered by Theorem \ref{theorem1.4}. The following theorem is concerned with the remaining case 
 $\gamma \in (2, 3]$, which  indicates that the result in Theorem \ref{theorem1.4} is sharp.

\begin{theorem}\label{th1.5b}
	In Theorem  \ref{theorem1.3}, suppose additionally the kernel functions satisfy $(\mathbf{\hat J^\gamma})$ for some $\gamma\in (2, 3]$, $F$ is $C^2$ and
	\begin{align}\label{1.7a}
F( v)-v[\nabla F( v)]^T \ggs \mathbf{0}\mbox{ \  for }\  \mathbf{0}\ \llp\; v\preceq \mathbf{u^*}.
	\end{align}
Then for $t\gg 1$,
	\[\begin{array}{rll}
	c_0t+g(t),\ c_0t-h(t)&\approx\  \ln t\ \ & {\rm if}\ {\gamma}=3,\\
	c_0t+g(t),\ c_0t-h(t)&\approx\ t^{3-{\gamma}}\ \ & {\rm if}\ {\gamma}\in (2,3).
	\end{array}
	\]
\end{theorem}

 Note that 
 $\mathbf{(f_2)}$ implies
 \[F( v)-v [\nabla F( v)]^T\succeq \mathbf{0} \ \mbox{ for } v\in [\mathbf{0}, \mathbf{u^*}].
 \]
 Therefore \eqref{1.7a} is a strengthened version of
 $\mathbf{(f_2)}$. If we take $v=\mathbf{u^*}$ in \eqref{1.7a}, then it yields  $\mathbf{u}^*[\nabla F(\mathbf{u}^*)]^T  \llp \mathbf{0}$.
 When $m=1$, \eqref{1.7a} reduces to $F(v)> F'(v)v$ for $0<v\leq \hat u$, which is satisfied, for example, by $F(v)=av-bv^p$ with $a, b >0$ and $p>1$.} 

\medskip

The proofs of Theorems \ref{theorem1.4} and \ref{th1.5b} rely on some of the following estimates on the semi-wave solutions of \eqref{2.3}, which are of independent interests.

\begin{theorem}\label{theorem1.6} Suppose that $F$ satisfies $\mathbf{(f_1)-(f_4)}$ and the kernel functions satisfy $\mathbf{(J)}$,
and $\Phi(x)=(\phi_i(x))$  is a monotone solution of \eqref{2.1a} for some $c>0$. Then the following conclusions hold:
\begin{itemize}
\item[{\rm (i)}] If $\mathbf{(J^\alpha)}$ holds for some $\alpha>0$, then for every $i\in\{1,..., m\}$,
\[
\int_{-\infty}^{-1}\big[u_i^*-\phi_i(x)\big]|x|^{\alpha-1}dx<\infty,
\]
which implies, by the monotonicity of $\phi_i(x)$,
 \[
 0<u_i^*-\phi_i(x)\leq C|x|^{-\alpha}  \mbox{ for some $C>0$ and all $x<0$}.
 \]
\item[{\rm (ii)}] If  $\mathbf{(J^\alpha)}$ does not hold for some $\alpha>0$, then
\[
\sum_{i=1}^m\int_{-\infty}^{-1}\big[u_i^*-\phi_i(x)\big]|x|^{\alpha-1}dx=\infty.
\]
 \item[{\rm (iii)}] If $\mathbf{(J_2)}$ holds, then there exist positive constants $C$ and $\beta$ such that 
		\begin{align*}
		0<u_i^*-\phi_i(x)\leq Ce^{\beta x} \mbox{ for all } x< 0,\ i\in\{1,..., m\}.
		\end{align*}
 \end{itemize}
\end{theorem}

\medskip

\subsection{Examples} 
It is clear that all the results here can be applied to the Fisher-KPP nonlocal diffusion model in \cite{cdJFA, dlz}. When $m=m_0=1$, our Theorem
\ref{theorem1.3} here recovers the main result in \cite{dlz}, but Theorems \ref{theorem1.4}, \ref{th1.5a}, \ref{th1.5b} and \ref{theorem1.6} are new even in this special case.
\smallskip

 Let us now examine  the models in \cite{dn} and \cite{ZZLD}.
\smallskip

\underline{Example 1}. The West Nile virus model in \cite{dn} is given by
\begin{equation}\label{1}
\begin{cases}
\dd H_t=d_1 \mathcal{L}_1[H](t,x)+a_1(e_1-H)V-b_1 H, & x\in (g(t), h(t)),\ t>0,\\
\dd V_t=d_2 \mathcal{L}_2[V](t,x)+a_1(e_2-V)H-b_2 V , & x\in (g(t), h(t)),\ t>0,\\
H(t,x)= V(t,x)=0, & t >0, \; x\in\{g(t), h(t)\},\\
\dd g'(t)= -\mu \int_{g(t)}^{h(t)}\int_{-\yy}^{g(t)}J_1(x-y)V(t,x)\rd y\rd x, & t >0,\\[3mm]
\dd h'(t)= \mu \int_{g(t)}^{h(t)}\int_{h(t)}^{\yy}J_1(x-y)V(t,x)\rd y\rd x, & t >0,\\
-g(0)=h(0)=h_0,\ H(0,x)=u_1^{0}(x),\ V(0,x)=u_2^{0}(x), \ &x\in [-h_0,h_0].
\end{cases}
\end{equation}
where $a_i$, $e_i$ and $b_i$ ($i=1,2$) are positive constants satisfying $a_1a_2e_1e_2>b_1b_2$ (which is necessary for spreading to happen).
We thus have
\begin{equation*}\label{2}
F(u)=F_1(u):=\Big(a_1(e_1-u_1)u_2-b_1u_1, a_2(e_2-u_2)u_1-b_2u_2\Big),
\end{equation*}
\[ 
\mathbf{u^*}=\left(\frac{a_1a_2-e_1e_2-b_1b_2}{a_1a_2e_2+a_2b_1}, \frac{a_1a_2-e_1e_2-b_1b_2}{a_1a_2e_1+a_1b_2} \right).
\]
It is straightforward to check that  conditions $\mathbf{(f_1)-(f_3)}$ are satisfied by $F_1$ with $\mathbf{\hat u}=(e_1, e_2)$. Condition $\mathbf{(f_4)}$ was shown to hold in \cite{dn}. It is also easy to see that $F_1$ is $C^2$ and 
\[
F_1(u)-u[\nabla F_1(u)]^T=(a_1u_1u_2, a_2u_1u_2).
\]
 Therefore \eqref{1.7a} holds as well. Thus all our results apply to \eqref{1}.

\smallskip

\underline{Example 2}. The epidemic model in \cite{ZZLD} is given by
\begin{equation}\label{FB}
\begin{cases}
u_{t}=d\displaystyle\mathcal{L}_1[u]
-au+cv,&t>0,\ x\in(g(t),h(t)),\\
v_{t}=-bv+G(u),&t>0,\ x\in(g(t),h(t)),\\
u(t,x)=v(t,x)=0,&t>0,\ x=g(t) \text{\ or\ } x=h(t),\\
g'(t)=-\mu\displaystyle\int_{g(t)}^{h(t)}\displaystyle\int_{-\infty}^{g(t)}
J_1(x-y)u(t,x)dydx,&t>0,\\
h'(t)=\mu\displaystyle\int_{g(t)}^{h(t)}\displaystyle\int_{h(t)}^{+\infty}
J_1(x-y)u(t,x)dydx,&t>0,\\
-g(0)=h(0)=h_{0},\ u(0,x)=u_{0}(x),\ v(0,x)=v_{0}(x),&x\in[-h_{0},h_{0}],
\end{cases}
\end{equation}
where  $a,\ b,\ c,\ d,\ \mu$ and $h_0$ are positive constants, and the function $G$ is assumed to satisfy
\begin{enumerate}[leftmargin=2.8em]
\item[(i)]$G\in C^{1}([0,\infty)),\ G(0)=0,\ G'(z)>0$ for $ z\geq 0$;
\item[(ii)] $\left[\frac{G(z)}{z}\right]'<0$ for $z>0$ and $\lim\limits_{z\rightarrow +\infty}
    \frac{G(z)}{z}<\frac{ab}{c}$;
    \item[(iii)] $G'(0)>\frac{ab}c$ (necessary for spreading to happen).
\end{enumerate}
In this example,
\[
F(u)=F_2(u):=(-au_1+cu_2, G(u_1)-bu_2),\;\; \mathbf{u^*}=(K_1, K_2)
\]
where $(K_1, K_2)\ggs\mathbf{0}$ are uniquely determined by
\[
\frac{G(K_1)}{K_1}=\frac{ab}{c}, \;\; K_2=\frac{G(K_1)}{b}.
\]
One easily checks that $F_2$ satisfies $\mathbf{(f_1)-(f_3)}$ with $\mathbf{\hat u}=\infty$. In \cite{ZZLD}, it was proved that $\mathbf{(f_4)}$ also holds.  Clearly $F_2$ is $C^2$. However, $\mathbf{u}^* [\nabla F_2(\mathbf{u}^*)]^T\llp{\bf 0}$ does not hold.
Therefore all our results apply to \eqref{FB} except Theorems \ref{theorem1.4} and \ref{th1.5b}.

\subsection{Related problems and comments}

The random (local) diffusion version of various special cases or variations of \eqref{1.1}  have been studied extensively in recent years, starting from the work \cite{DL2010}. In these random diffusion free boundary problems  the spreading speed is always finite, and is determined by the associated semi-waves; see, for example, \cite{ABL, DL2010, DLou, DMZ, DWZ, KMY, LZ, WND, ZLN} for an incomplete sample of such results. In many situations, especially in population and epidemic models, it has been recognised that random diffusion is often not the best approximation of the spatial dispersal behaviour of the concerned species, and  in order to include factors such as long-distance dispersal, in
a large amount of literature, the random diffusion terms in the models are
replaced by nonlocal diffusion operators as in \eqref{2.3}. A striking difference of this type of nonlocal diffusion models to their random diffusion counterparts is that accelerated spreading may occur. For the scalar case of \eqref{2.3}, namely for the Fisher-KPP  equation with nonlocal diffusion, it follows from the theory in \cite{W} that accelerated spreading  occurs exactly when the kernel function does not satisfy $\mathbf{(J_2)}$ described above. On the other hand, when $\mathbf{(J_2)}$ is satisfied by the kernel function (\underline{thin-tailed} kernel)
then there exists some $c_*>0$ such that the associated traveling wave problem has a monotone traveling wave with speed $c$ if and only if $c\geq c_*$, and $c_*$ is the asymptotic spreading speed determined by the scalar nonlocal Fisher-KPP equation \eqref{2.3}; see, for example, \cite{LZh,T1, T2, TZ, W, yagi}. Related works on accelerated spreading can be found in \cite{AC, BGHP, CR, FF, FT, G, HR, ST, XLL, XLR } and the references therein. It is well known that the random diffusion version of \eqref{2.3} with compactly supported initial functions can only spread with finite speed, which is determined by the associated traveling waves \cite{LZh, W,  WSL, ZW}. 

The works \cite{cdJFA, CQW} appear to be the first to consider a nonlocal diffusion version of the  free boundary problem proposed in \cite{DL2010}, where the free boundary conditions of the form in \eqref{1.1} were first introduced (independently). In \cite{CQW} the authors considered the case that the reaction term is identically 0 and therefore completely different long-time dynamical behaviour was shown. 
Our results here (as well as in \cite{dlz}) indicate that $\mathbf{(J_1)}$ is the threshold condition on the kernel functions which decides whether accelerated spreading happens for \eqref{1.1}; in contrast, as mentioned above, for \eqref{2.3}, $\mathbf{(J_2)}$ is the corresponding threshold condition
(at least in the scalar Fisher-KPP case).

There are two fundamental differences between the free boundary model \eqref{1.1} and the corresponding model \eqref{2.3} where no free boundary appears. Firstly \eqref{1.1} provides the exact location of the spreading front, which is the free boundary, while the location of the front is not prescribed in \eqref{2.3}, and one usually uses  suitable level sets of the solution to describe the front behaviour. Secondly,
the long time dynamical behaviour of \eqref{1.1} is often governed by a spreading-vanishing dichotomy \cite{cdJFA, dn, ZZLD}, but \eqref{2.3} predicts successful spreading all the time. Let us also note that since $\mathbf{(J_2)}$ implies $\mathbf{(J_1)}$ but not the other way round, \eqref{2.3}  is more readily than \eqref{1.1} to give rise to accelerated spreading.

\subsection{Organisation of the paper}
The rest of the paper is organised as follows. In Section 2, we prove Theorems \ref{prop2.1} and  \ref{prop2.3}. Much of the arguments there are based on the perturbed semi-wave problem \eqref{2.1}, which in the limit yields either a semi-wave solution or a traveling wave solution. This limiting process is also used in several other places of the paper; for example, it plays a crucial role in our proof for accelerated spreading in Theorem \ref{theorem1.3}.

Section 3 is devoted to the proof of Theorem \ref{theorem1.3}, which relies on careful constructions of upper and lower solutions building from the semi-wave solution with the desired speed $c_0$, and on a limiting argument when such a semi-wave does not exist, leading to accelerated spreading.

Section 4 gives the proof of Theorems \ref{theorem1.4} and \ref{theorem1.6}. The proof of the former is built on the proof and conclusions of the latter, where subtle analysis is used to find out the relationship between the behaviour of the semi-wave solution and that of the kernel functions. The constructions of upper and lower solutions in the proof of Theorem \ref{theorem1.4} are much more subtle than that in Section 3.

 Sections 5 and 6  are devoted to the proof of Theorems \ref{th1.5a} and \ref{th1.5b} for kernel functions behaving like $|x|^{-\gamma}$ near infinity.
 In Section 5, we completely determine the growth orders of $c_0t-h(t)$ for $\gamma$ in the range $(2, 3]$, while in Section 6, we completely determine the accelerated spreading orders of $h(t)$ when $\gamma$ falls into the range $ (1, 2]$.  Note that when $\gamma>3$, the spreading behaviour is already covered by the more general results in Section 4.  
 
 Although Sections 4, 5 and 6 are based on completely new ideas and techniques,
 part of the strategy in the approach of Sections 2 and 3 is borrowed from \cite{dlz}, albeit the situation here is much more general and complex,  requiring solutions to new problems and use of noval methods.  For example, (a) while in the scalar case the traveling wave solutions of \eqref{2.3} are completely understood in \cite{yagi}, here the corresponding result requires a rather nontrivial proof, (b) the perturbation problem \eqref{2.1} is treated here by a much simpler method than  the one suggested by \cite{dlz}, which appears difficult to apply in the more general setting here, and (c) our proof of accelerated spreading in Subsection 3.3 uses a completely new approach.

\section{Semi-waves and traveling waves of \eqref{2.3}}

\subsection{Some preparatory results}

Since $F=(f_i)$ is $C^1$ over $\R^m_+$,
 for any vector $K=(k_i)\ggs \mathbf{0}$ there is a constant $L(K)>0$ such that   for $u,v\in [\mathbf{0}, K]$ with $u=(u_i)$ and $v=(v_i)$, 
\begin{align}\label{1.3a}
|f_i(u)-f_i(v)|\leq L(K) \sum_{i=1}^{m}|u_i-v_i|.
\end{align}
\begin{lemma}\label{lemma2.1a}
If $\mathbf{(f_1)}$ holds, then
 there exist $\lambda_1>0$, small $\epsilon>0$,  and  vectors  $\Theta=(\theta_i)\ggs \mathbf{0},\ \tilde\Theta=(\tilde\theta_i)\ggs \mathbf{0}$  such that   
	\begin{align}\label{2.2c}
	 \Theta\nabla F(\mathbf{0})^T=\lambda_1 \Theta,\ \  \tilde \Theta\nabla F(\mathbf{0})=\lambda_1\tilde \Theta,
	\end{align} 
	and 
	\begin{align}\label{2.3c}
	F(\epsilon\Theta)\ggs\mathbf{0},\; \
	\sum_{i=1}^m \td\theta_i f_i(X)\geq \sum_{i=1}^{m} b_i x_i \mbox{ for } X=(x_i)\in [\mathbf{0}, \epsilon\mathbf{1}],
	\end{align}
	where $\mathbf{1}=(1,\cdots,1)\in \R^m$ and
	 $b_i:=\frac{\lambda_1\td\theta_i}{2}>0$.
\end{lemma}
\begin{proof}
Let $\lambda_1$ be the principal eigenvalue of $\nabla F({\bf 0})$.  By the Perron-Frobenius theorem, there exist  positive eigenvectors $\Theta$ and $\tilde \Theta$ such that the identities in \eqref{2.2c} hold.

Moreover, in view of 
 $F\in [C^1(\R_+^m)]^m$, for small $\epsilon>0$ and $X=(x_i)\in [\mathbf{0}, \epsilon\mathbf{1}]$,
\begin{align*}
&F(\epsilon\Theta)=\epsilon\Theta \big[\nabla F(\mathbf{0})^T+o(1){\bf I}_m\big]=\epsilon[\lambda_1+o(1)]\Theta,\\
&\sum_{i=1}^m \tilde\theta_if_i(X)=\td\Theta\big[\nabla F(\mathbf{0})+o_{\epsilon}(X)\big]X^T=\td\Theta[\lambda_1{\bf I}_m+o_\epsilon(X)]X^T, 
\end{align*} 
with $|o_\epsilon(X)|\to 0$ as $\epsilon\to 0$ uniformly in $X\in [\mathbf{0}, \epsilon\mathbf{1}]$.
Hence 
 \eqref{2.3c} holds provided that $\epsilon>0$ is small enough. 
\end{proof}

\begin{lemma}\label{lemma2.7a}
	Assume $\mathbf{(J)}$  holds, and for every $i\in\{1,..., m\}$, $v_i\in C(\R)\cap C^1(\R\backslash\{0\})$ satisfies
	\begin{equation*}
	\begin{cases}
	\dd d_i \int_{\R}J_i(x-y)  v_i(y) \rd y-d_i v_i(x)+p_i(x) v_i'(x)+\sum_{j=1}^mq_{ij}(x) v_j(x)\leq  {0}, &x<0,\\
	v_i(x)\geq  0,& x\geq 0,
	\end{cases}
	\end{equation*}
	where  $d_i\in [0,\infty)$, $p_{i}$, $q_{ij}\in L_{loc}^\yy(\R)$ with $q_{ij}\geq 0$ for $i\neq j$. If  $v_i(x)\geq  0$ for all $x\in \R$ and $i\in\{1,..., m\}$,  with $d_{i_0}>0$ and $v_{i_0}(x)\geq, \not\equiv 0$ for some $i_0\in\{1, ... , m\}$,  then $v_{i_0}(x)> 0$ for $x<0$.
\end{lemma}
\begin{proof}
	Assume by contradiction there  is $x_0< 0$ such that $v_{i_0}(x_0)=0$.  Since $v_{i_0}(x)\not\equiv 0$, we may also require that there exists  a sequence $\{x_n\}_{n=1}^\yy$ with $x_n\to x_0$ such that $v_{i_0}(x_n)>0$.  Then $v_{i_0}'(x_0)=0$,  and  from the inequality satisfied by $v_{i_0}$ we deduce that 
	\begin{align*}
	d_{i_0} \int_{\R}J_{i_0}(x_0-y)  v_{i_0}(y) \rd y\leq 0,
	\end{align*}
	which implies that $v_{i_0}(x)=0$  for all $x$ near $x_0$ due to  $J_{i_0}(0)>0$. However, this contradicts to $v_{i_0}(x_n)>0$.
\end{proof}

\begin{lemma}\label{lemma2.3aa} Assume that   $(\mathbf{J})$  holds, $T\in (0, \infty)$, and for every $i\in \{1,..., m\}$,
$v_i\in C([0,T]\times \R) \cup L^\yy([0,T]\times \R)$, $\partial_tv_i \in C([0,T]\times \R)$, and
	\begin{equation*}
	\begin{cases}
	\dd \partial_tv_i(t,x)\geq d_i(t,x)	\mathcal{L}_i[v_i](t,x)+\sum_{j=1}^mq_{ij}(t,x) v_j(t,x),	& t\in (0,T],\; x\in R, \\
	v_i(0,x)\geq 0, &x\in \R,
	\end{cases}
	\end{equation*}
	where 
		\begin{align*}
	\mathcal{L}_i[v_i](t,x):=\int_\R J_i(x-y)v_i(t,y)\rd y-v_i(t,x),
	\end{align*}
and the functions $d_i$, $q_{ij}\in L^\yy([0,T]\times \R)$ satisfy $d_i\geq 0$ and   $q_{ij}\geq 0$ for $i\neq j$. Then 
	\begin{align}\label{2.2b}
	v_i(t,x)\geq  0 \mbox{ for }  \ (t,x)\in [0,T]\times\R,\; 1\leq i\leq m.
	\end{align}
\end{lemma}
\begin{proof}
 For any given  $\epsilon>0$, define
\begin{align*}
w_i(t,x):=v_i(t,x)+\epsilon e^{At}\ \mbox{ for } \ (t,x)\in [0,T]\times\R,
\end{align*}
where $\dd A:=\sum_{1\leq i,j\leq m} \|q_{ij}\|_{L^\infty([0,T]\times\R)}+1$. Then 
\begin{equation}\label{2.3b}
\begin{aligned}
\dd \partial_tw_i-d_i\mathcal{L}_i[w_i] -\sum_{j=1}^mq_{ij}w_j&=\dd \partial_tv_i- d_i\mathcal{L}_i[v_i] -\sum_{j=1}^m q_{ij}v_j
+\epsilon Ae^{At}-\epsilon e^{At}\sum_{j=1}^nq_{ij}\\
&\geq \Big(A-\sum_{j=1}^n q_{ij}\Big)\epsilon e^{At}\geq \epsilon e^{At}\ \   {\rm for}\  (t,x)\in (0,T]\times\R.
\end{aligned}
\end{equation}

We claim that 
\begin{align}\label{2.4a}
w_i(t,x)>0 \ \mbox{ for } \ (t,x)\in [0,T]\times \R,\; 1\leq i\leq m.
\end{align}
Define 
\begin{align*}
T_0=\sup\{s\in (0, T]:  w_i(t,x)>0\ {\rm for\ all}\ (t,x)\in [0,s]\times \R,\; 1\leq i\leq m\}.
\end{align*}
Clearly, $T_0>0$ is well defined  since  $w_i(0,x)\geq \epsilon$ for $x\in \R$ and $\partial_t w_i$ has a finite lower bound due to \eqref{2.3b}.  Moreover,  the definition of $T_0$ implies that either \eqref{2.4a} holds or 
\begin{equation}\label{2.7a}
\begin{cases}
w_i(t,x)>0\ {\rm on}\ [0,T_0)\times\R \ \mbox{ for }1\leq i\leq m,\\
A:=\min_{1\leq i\leq m}\inf_{x\in \R} w_i(T_0,x)=0,
\end{cases}
\end{equation}
since if $A>0$ in \eqref{2.7a}, then in view of the assumption that \eqref{2.4a} does not hold, we must have $T_0<T$, but then
 we could apply the fact that  $\partial_t w_i$  is bounded below to deduce that for some small $\td \epsilon>0$,
\begin{align*}
w_i(t,x)>A/2>0\ \ \ \ \ {\rm for}\ (t,x)\in [T_0,T_0+\td\epsilon)\times\R,\ 1\leq i\leq m,
\end{align*}
which contradicts the definition of $T_0$.

 From $A=0$ we can find $1\leq i_0\leq m$ such that $\inf_{x\in\R} w_{i_0}(T_0,x)=0$. Then for $0<\epsilon_n\ll 1$ with $\epsilon_n\to 0$ as $n\to\yy$, we can find $x_n \in \R$  such that 
\begin{align*}
w_{i_0}(T_0,x_n)<\epsilon_n. 
\end{align*} 
Clearly $\dd \inf_{(t,x)\in [0,T_0]\times \R}w_{i_0}(t,x)=0$.  Making use of Ekeland's variational principle \cite{Ek}, for $\lambda=\min\{T_0/2,1\}$ and each $n\geq 1$, there is $(\td t_n,\td x_n)\in [0,T_0]\times \R$  such that 
\[
\begin{cases}
w_{i_0}(\td t_n,\td x_n)\leq w_{i_0}(T_0,x_n)<\epsilon_n, \ \ \  |T_0-\td t_n|+|x_n-\td x_n|<\lambda,\\
w_{i_0}(\td t_n,\td x_n)-w_{i_0}(t,\td x_n)<\frac{|\td t_n-t|}{\lambda}\epsilon_n \ \mbox{ for any $t\in (0,T_0)$},
\end{cases}
\]
It follows that (note that $\td t_n>T_0-\lambda\geq T_0/2>0$) 
\begin{equation}\label{2.5}
\partial_tw_{i_0}(\td t_n,\td x_n)\leq \frac{\epsilon_n}{\lambda} \mbox{ for all } n\geq 1.
\end{equation} 
 
On the other hand,  by \eqref{2.3b} and $w_i\geq 0$ on $[0,T_0]\times \R$ we deuce 
\begin{align*}
 &\partial_tw_{i_0}(\td t_n,\td x_n)\\
\geq &\  d_{i_0}\mathcal{L}_{i_0}[w_{i_0}](\td t_n,\td x_n)+\sum_{j=1}^mq_{i_0j}w_j(\td t_n,\td x_n)+\epsilon e^{At}\\
\geq &\ -d_{i_0} w_{i_0}(\td t_n,\td x_n)+q_{i_0i_0} w_{i_0}(\td t_n,\td x_n)+\epsilon e^{At}\\
\geq &\  -(\|d_{i_0}\|_\infty+\|q_{i_0i_0}\|_\infty)\epsilon_n+\epsilon e^{At}\\
\geq&\  \frac{1}{2} \epsilon e^{At}\ \ \ \ \ {\rm for\ all\ large}\ n,
\end{align*}
which is a contradiction to \eqref{2.5} since $\epsilon_n\to 0$ as $n\to\yy$.  

Hence \eqref{2.4a} always  holds true. Letting $\epsilon\to 0$, we immediately get \eqref{2.2b}.
\end{proof}

\begin{remark}\label{cp-rmk}
We note that the above proof  indicates that the assumption $\partial_tv_i \in C([0,T]\times \R)$ in Lemma \ref{lemma2.3aa} can be relaxed.
If for each $(t,x)$, both the one-sided partial derivatives $\partial_tv_i(t+0,x)$ and $\partial_tv_i(t-0, x)$ exist, and the differential inequalities
are satisfied when $\partial_tv_i(t,x)$ is replaced by both  one-sided partial derivatives, then  the conclusion of Lemma \ref{lemma2.3aa} remains valid. This also applies to the comparison lemmas in the rest of this paper, though we will not  remark again after them.
\end{remark}

Making use of  $(\mathbf{f_1})$ and Lemma \ref{lemma2.3aa}, we have the following result.

\begin{lemma}\label{lemma2.3a} Assume that  $(\mathbf{J})$ and $(\mathbf{f_1})$ hold, $T>0$ and $d_i\geq 0$ are constants, and $U=(u_i)$ with $u_i\in C([0,T]\times \R)$, $\partial_t u_i\in C([0,T]\times \R)$ $(i=1,..., m)$ solves \eqref{2.3} for $t\in [0, T]$.  If  for every $i\in\{1,..., m\}$, $v_i\in C([0,T]\times \R) \cup L^\yy([0,T]\times \R)$ and $\partial_tv_i \in C([0,T]\times \R)$ satisfy 
	\begin{equation*}
	\begin{cases}
	\dd \partial_tv_i\geq d_i\mathcal{L}_i[v_i]+f(v_1,v_2\cdots,v_m),\ \  & (t,x)\in  (0,T]\times R, \\
	v_i(0,x)\geq u_i(0,x), &x\in \R,
	\end{cases}
	\end{equation*}
	where
	\begin{align*}
	\mathcal{L}_i[v_i](t,x):=\int_\R J_i(x-y)v_i(t,y)\rd y-v_i(t,x),
	\end{align*}
	then 
	\begin{align*}
	v_i(t,x)\geq  u_i(t,x) \ \mbox{ for } \  \ (t,x)\in  [0,T]\times R.
	\end{align*}
\end{lemma}

\subsection{A perturbed semi-wave problem}

For  $\bm{\delta}\ggs \mathbf{0}$, we consider the auxiliary problem 
\begin{equation}\label{2.1}
\begin{cases}
\dd D\circ\int_{-\yy}^{\yy}\mathbf{J} (x-y)\circ \Phi(y) {\rm d}y-D\circ\Phi+c\Phi'(x)+F(\Phi(x))=0,&
-\yy<x<0,\\
\Phi(-\yy)=\mathbf{u}^*,\ \ \Phi(x)=\bm{\delta},&0\leq x<\yy.
\end{cases}
\end{equation}
If ${\bm \delta}=\mathbf{0}$ then \eqref{2.1} is reduced to the semi-wave problem \eqref{2.1a}; therefore \eqref{2.1} can be viewed as a perturbed semi-wave problem. As we will see below, the semi-wave solutions and traveling wave solutions of \eqref{2.3} can be obtained as the limit of the solution of \eqref{2.1} when $\bm{\delta}\to\mathbf{0}$, subject to suitable translations in $x$.

Define 
\begin{align}\label{2.7b}
\td\sigma(v):=F(v)+cM v-D\circ v=(f_i(v)+(cM-d_i)v_i)\ \ \ \ {\rm for}\ v=(v_i)\in \R_+^m,
\end{align}
where $M>0$ is a constant.   Then the first equation in \eqref{2.1} is equivalent to
\begin{align}\label{2.2}
-c(e^{-Mx}\Phi)'=e^{Mx}\lf[D\circ\int_{-\yy}^{\yy} \mathbf{J}(x-y)\circ\Phi(y) {\rm d}y+\td\sigma(\Phi(x)) \rr].
\end{align}
Since $F$ is $C^1$,  we could choose  $M$ large enough such that $\td\sigma(v)$ is increasing for $v\in [0,\mathbf{u}^*+\mathbf{1}]$, namely 
\[
\td\sigma(v)\succeq \td\sigma(u) \mbox{ if } u, v\in [0,\mathbf{u}^*+\mathbf{1}] \mbox{ and } v\succeq u.
\]

\begin{lemma}\label{lemma2.2}
Suppose $\mathbf{(J)}$ and $\mathbf{(f_1)}$ hold. Let  $\bm{\delta}=\epsilon \Theta$ for small $\epsilon>0$, where $\Theta$ is given by Lemma \ref{lemma2.1a}. Then   the problem \eqref{2.1} has a solution $\Phi(x)=(\phi_{i}(x))$ which is nonincreasing in $x$, and can be obtained by an iteration process to be specified in the proof. 
\end{lemma}
\begin{proof}
Let 
\begin{align*}
\Omega:=\{\Gamma\in [C(\R)]^m: 0\preceq\Gamma\preceq \mathbf{u}^*\}. 
\end{align*}
 Define an operator $P=(P_i): \Omega\to  [C(\R)]^m$ by 
\begin{equation*}
P[\Gamma](x)=
\begin{cases}
\dd e^{Mx}\bm{\delta}+\frac{e^{Mx}}{c}\int_{x}^{0}e^{-M\xi}\lf[D\circ\int_{-\yy}^{\yy} \mathbf{J}(\xi-y)\circ\Gamma(y) {\rm d}y+\td\sigma(\Gamma(\xi)) \rr]{\rm d}\xi ,&x<0,\\
\bm{\delta},&x\geq  0.
\end{cases}
\end{equation*}
Using \eqref{2.2} we easily see that \eqref{2.1} is equivalent to 
\begin{equation}\label{fpt}
\begin{cases}
\Phi(x)=P[\Phi](x) \mbox{ for } x\in \R,\\
\Phi(-\yy)=\mathbf{u}^*.
\end{cases}
\end{equation}
We next solve \eqref{fpt} in three steps.

{\bf Step 1} We  show that $P$ has a fixed point in $\Omega$. 

Firstly we prove that $P[\bm{\delta}](x)\succeq \bm{\delta}$ with  $\bm{\delta}$ regarded as a constant function.
By the definition of $P$, we have $P[\bm{\delta}](x)=\bm{\delta}$ for $x\geq 0$. For $x<0$,   
\begin{align*}
P[\bm{\delta}](x)=\ &e^{Mx}\bm{\delta}+\frac{e^{Mx}}{c}\int_{x}^{0}e^{-M\xi}\lf[D\circ\bm{\delta}+\td\sigma(\bm{\delta}) \rr]{\rm d}\xi\\
=\ &e^{Mx}\bm{\delta}+\frac{e^{Mx}}{c}\int_{x}^{0}e^{-M\xi}\lf[cM\bm{\delta}+ F(\bm{\delta}) \rr]{\rm d}\xi\\
\ggs\  & e^{Mx}\bm{\delta}+\frac{e^{Mx}}{c}\int_{x}^{0}e^{-M\xi}\lf[cM\bm{\delta} \rr]{\rm d}\xi\\
=\ & e^{Mx}\bm{\delta}- e^{Mx}\bm{\delta}+\bm{\delta}=\bm{\delta} 
\end{align*}
since $F(\bm{\delta})\ggs \mathbf{0}$  by Lemma \ref{lemma2.1a}.

Secondly we show $P[\mathbf{u}^*](x)\llp \mathbf{u}^*$.
Since $\epsilon>0$ is small,  $P[\mathbf{u}^*](x)=\bm{\delta}=\epsilon\Theta \llp \mathbf{u}^*$ for $x\geq 0$.  For $x<0$, we have
\begin{align*}
P[\mathbf{u}^*](x)=\ &e^{Mx}\bm{\delta}+\frac{e^{Mx}}{c}\int_{x}^{0}e^{-M\xi}\lf[D\circ\mathbf{u}^*+\td\sigma(\mathbf{u}^*) \rr]{\rm d}\xi\\
=\ &e^{Mx}\bm{\delta}+\frac{e^{Mx}}{c}\int_{x}^{0}e^{-M\xi}\lf[cM\mathbf{u}^*+ F(\mathbf{u}^*) \rr]{\rm d}\xi\\
=\ & e^{Mx}\bm{\delta}+\frac{e^{Mx}}{c}\int_{x}^{0}e^{-M\xi}\lf[cM\mathbf{u}^* \rr]{\rm d}\xi\\
\llp\ & e^{Mx}\mathbf{u}^*+\frac{e^{Mx}}{c}\int_{x}^{0}e^{-M\xi}\lf[cM\mathbf{u}^* \rr]{\rm d}\xi\\
=\ & e^{Mx}\mathbf{u}^*- e^{Mx}\mathbf{u}^*+\mathbf{u}^*=\mathbf{u}^*. 
\end{align*}

Next we define inductively 
\begin{align*}
\Gamma_0(x):=\bm{\delta},\ \Gamma_{n+1}(x):=P[\Gamma_n](x)=P^n[\Gamma_0](x)\ \mbox{ for} \ \ n=0,1,2,\cdots, x\in \R.
\end{align*}
Then 
\begin{align*}
\Gamma_0\preceq\Gamma_n\preceq\Gamma_{n+1}\llp \mathbf{u}^*
\end{align*}
due to the monotonicity of $P$ which is a simple consequence of the fact that  $\td F(v)$ is increasing in $v\in [0,\mathbf{u}^*]$.

Define 
\begin{align*}
\widehat \Gamma(x):=\lim_{n\to\yy }\Gamma_n(x)\in [0,\mathbf{u}^*].
\end{align*}
It is clear that  $\widehat \Gamma(x)=\bm{\delta}$ for $x\succeq 0$. Making use of the Lebesgue dominated convergence theorem and $\Gamma_{n+1}(x)=P[\Gamma_n](x)$, for $x<0$ we deduce 
\begin{align*}
\widehat \Gamma(x)=P[\widehat \Gamma](x),
\end{align*} 
which also implies that  $\widehat \Gamma'(x)$ exists and is continuous for $x<0$. Hence  $\widehat \Gamma$  is a fixed point of $P$ in $\Omega$.

{\bf Step 2}.  We show that $\widehat \Gamma'(x)\preceq \mathbf{0}$ for $x< 0$.

It suffices to prove that $\Gamma_n'(x)\preceq \mathbf{0}$ for $x<0$ and each $n=0,1,2,\cdots$, since this would imply each $\Gamma_n$ is nonincreasing and hence $\widehat\Gamma(x)$ is nonincreasing for $x<0$.  

It is clear that $\Gamma_0(x)=\bm{\delta}$ is nonincreasing.  Assume $\Gamma_n'(x)\preceq \mathbf{0}$ for $x< 0$. We show that  $\Gamma_{n+1}'(x)\preceq\mathbf{0}$ for $x<0$. 

  By the definition,  for $x<0$,
 \begin{align*}
 \Gamma_{n+1}(x)=e^{Mx}\bm{\delta}+\frac{e^{Mx}}{c}\int_{x}^{0}e^{-M\xi}g_n(\xi){\rm d}\xi, 
 \end{align*}
 where
 \begin{align*}
g_n(\xi)= g(\xi; \Gamma_{n}):=&D\circ\int_{-\yy}^{\yy} \mathbf{J}(\xi-y)\circ\Gamma_{n}(y) {\rm d}y+\td\sigma(\Gamma_{n}(\xi))\\
 =& D\circ\int_{-\yy}^{\yy} \mathbf{J}(y)\circ\Gamma_{n}(y+\xi) {\rm d}y+\td\sigma(\Gamma_{n}(\xi))
 \end{align*}
Let us note that $\Gamma_{n}'(z)\preceq \mathbf{0}$ for $z\neq 0$,  and  every element of the matrix function $\nabla \td\sigma(z)$ is nonnegative for  $z\in [\R^+]^m$. It follows that  
$g_n(\xi)$ is differentiable for all $\xi\in \R$, and  $g_n'(\xi)\preceq \mathbf{0}$ for $\xi\in \R$. Moreover, 
\begin{align*}
 g_n(0)=g(0;\Gamma_{n})
\succeq  g(0;\Gamma_{0})=D\circ \bm{\delta}+\td\sigma(\bm{\delta})=cM\bm{\delta}+F(\bm{\delta})\succeq cM\bm{\delta},
\end{align*}
since $\Gamma_{n}\succeq \Gamma_{0}=\bm{\delta}$, $F(\bm{\delta})\ggs \mathbf{0}$, and  $g(0;\Gamma_{n})$ is nondecreasing with respect to $\Gamma_{n}$. Therefore, for $x<0$,
\begin{align*}
 (\Gamma_{n+1})'(x)=&\bm{\delta} Me^{Mx}+M\frac{e^{Mx}}{c}\int_{x}^{0}e^{-M\xi}g_n(\xi){\rm d}\xi-\frac{1}{c}g_n(x)\\
 =& \bm{\delta} Me^{Mx}+M\frac{e^{Mx}}{c}\lf[\frac{-e^{-M\xi}}{M}g_n(\xi)\big |_{x}^{0}+ \int_{x}^{0}e^{-M\xi}g_n'(\xi){\rm d}\xi\rr]-\frac{1}{c}g_n(x)\\
 \preceq&\bm{\delta} Me^{Mx} +M\frac{e^{Mx}}{c}\lf[\frac{-g_n(0)}{M}+\frac{e^{-Mx}}{M}g_n(x)\rr]-\frac{1}{c}g_n(x)\\
 =&\bm{\delta} Me^{Mx} -\frac{g_n(0)e^{Mx}}{c}\preceq\bm{\delta} Me^{Mx} -\bm{\delta} Me^{Mx}=0.
\end{align*} 
By the principle of mathematical induction, we have  $\Gamma_n'(x)\preceq \mathbf{0}$ for $x< 0$ and all $n\geq 1$.

{\bf Step 3}.  We verify $\widehat \Gamma(-\yy)=\mathbf{u}^*$. 

By step 2,  $\lim_{x\to-\yy}\widehat \Gamma(x)=K$ exists, and  $\mathbf{0}\llp K\preceq \mathbf{u}^*$.   We claim that 
\begin{align}\label{2.4}
\lim_{x\to-\yy}\int_{-\yy}^{\yy}\mathbf{J}(x-y)\circ \widehat\Gamma(y){\rm d} y=K.
\end{align}
Indeed, since $\widehat \Gamma$ is nonincreasing and $\lim_{x\to-\yy}\widehat \Gamma(x)=K$, we have
\begin{align*}
&\int_{-\yy}^{\yy}\mathbf{J}(x-y)\circ\widehat\Gamma(y){\rm d} y=\int_{-\yy}^{\yy}\mathbf{J}(y)\circ\widehat\Gamma(y+x){\rm d} y\\
\succeq & \int_{-\yy}^{-x/2}\mathbf{J}(y)\circ\widehat\Gamma(y+x){\rm d} y\succeq \widehat\Gamma(x/2)\circ\int_{-\yy}^{-x/2}\mathbf{J}(y){\rm d} y \to K
\end{align*}
as $x\to -\yy$,  and on the other hand
\begin{align*}
\int_{-\yy}^{\yy}\mathbf{J}(x-y)\circ\widehat\Gamma(y){\rm d} y=&\int_{-\yy}^{\yy}\mathbf{J}(y)\circ\widehat\Gamma(y+x){\rm d} y\\
\preceq & \int_{-\yy}^{\yy}\mathbf{J}(y)\circ K{\rm d} y= K,
\end{align*}
which gives \eqref{2.4}. 

If $K\not=\mathbf{u^*}$, then
by ${\bf({f_1})}$, we have  $F(K)=(f_i(K))\neq \mathbf{0}$.
Note that  $\widehat\Gamma$ satisfies
\begin{align*}
D\circ\int_{-\yy}^{\yy}\mathbf{J}(x-y)\circ \widehat\Gamma(y){\rm d} y-D\circ\widehat\Gamma+c\widehat\Gamma'(x)+F(\widehat\Gamma(x))={\bf 0},\ \ \ \ 
-\yy<x<0.
\end{align*}
Letting $x\to -\yy$ and making use of \eqref{2.4}, we deduce 
\begin{align*}
\lim_{t\to-\yy}\widehat\Gamma'(x)=\lim_{t\to-\yy} F(\widehat\Gamma(x))=F(K)\neq {\bf 0},
\end{align*}
which contradicts  the fact that $\hat \Gamma$ is nonincreasing and  bounded. Thus, $\widehat \Gamma(-\yy)=\mathbf{u}^*$. 

Combining Steps 1-3, we see that \eqref{2.1} admits a nonincreasing solution $\widehat \Gamma$, which is the limit of $\Gamma_n$ obtained from an iteration process.
\end{proof}

The following result describes the monotonic dependence on $c$ and $\bm{\delta}=\epsilon\Theta$ of the solution $\Phi$ to \eqref{2.1} obtained in the above lemma. To stress these dependences, we will write $\Phi=\Phi_\epsilon^c$.

\begin{lemma}\label{lemma2.3}
Suppose $\mathbf{(J)}$ and $\mathbf{(f_1)}$ hold.  Let $\Phi_\epsilon^c$ be the solution of \eqref{2.1}  obtained through the iteration process  in Lemma \ref{lemma2.2}, with $c>0$ and $\bm{\delta}=\epsilon \Theta$. Then  
\begin{equation}\label{2.7}
\begin{cases}
\Phi_{\epsilon_1}^c\preceq\Phi_{\epsilon_2}^c\ \  &{\rm if\ } 0< \epsilon_1\leq \epsilon_2\ll 1,\\
\Phi_{\epsilon}^{c_1}\succeq \Phi_{\epsilon}^{ c_2}\ \  &{\rm if\ } 0<c_1\leq  c_2.
\end{cases}
\end{equation}
\end{lemma}
\begin{proof}
To verify the first inequality in \eqref{2.7} for fixed $c>0$, we adopt the definition of  $P$ and $\Phi_n$  in Lemma \ref{lemma2.2}, but
in order to distinguish them between  $\bm{\delta}=\epsilon_1\Theta$ and $\bm{\delta}=\epsilon_2\Theta$,   we write  $P=P_{i}$
and  $\Phi_{n}=\Phi_{i,n}$ for $\bm{\delta}=\epsilon_i\Theta$, $i=1,2$. Thus we have
\begin{align*}
\Phi_{\epsilon_i}^c(x)=\lim_{n\to \yy} \Phi_{i, n}(x).
\end{align*}
Since  $P[\Phi](x)$ is nondecreasing with respect to $\bm{\delta}$ and $\Phi$, respectively, we have  
\begin{align*}
\Phi_{1, n+1}(x)=P_{1}[\Phi_{1,n}](x)\preceq P_{1}[\Phi_{2, n}](x)\preceq P_{2}[\Phi_{2, n}](x)=\Phi_{2, n+1}(x)
\end{align*}
provided that 
\begin{align*}
\Phi_{1, n}(x)\preceq\Phi_{2, n}(x).
\end{align*}
Since $\Phi_{1, 0}(x)\equiv \epsilon_1\Theta \preceq \epsilon_2\Theta \equiv \Phi_{2,0}(x)$, the above conclusion combined with the induction method gives $\Phi_{1, n}(x)\preceq\Phi_{2, n}(x)$ for all $n=0,1,2,\cdots$, which implies $\Phi_{\epsilon_1}^c(x)\preceq \Phi_{\epsilon_2}^c(x)$, as desired.

We now show the second inequality in \eqref{2.7} for fixed $\bm{\delta}=\epsilon\Theta$.  To stress the reliance on $c_i$, we use the notions $P^{i}$ and $\Phi_{n}^i$, respectively, for $P$ and $\Phi$ when $c=c_i$, $i=1,2$.   
From Lemma \ref{lemma2.2},  we have for $i=1,2$,
\begin{align*}
\Phi_\epsilon^{c_i}(x)=\lim_{n\to\yy} \Phi_{n}^i(x)=\lim_{n\to\yy}  P^{i}[\Phi_{n}^i](x).
\end{align*}
Due to $c_1\leq c_2$ and \eqref{2.1}, we have
\begin{align*}
&D\circ\int_{-\yy}^{\yy}\mathbf {J}(x-y)\circ \Phi_\epsilon^{c_1}(y)dy-D\circ\Phi_{c_1}+c_2(\Phi_\epsilon^{c_1})'(x)+F(\Phi_\epsilon^{c_1}(x))\\
\preceq  &D\circ\int_{-\yy}^{\yy}\mathbf {J}(x-y)\circ \Phi_\epsilon^{ c_1}(y)dy-D\circ\Phi_\epsilon^{ c_1}+c_1(\Phi_\epsilon^{ c_1})'(x)+F(\Phi_\epsilon^{ c_1}(x))=0,
\end{align*}
which implies that 
\begin{align*}
\Phi_\epsilon^{c_1}(x)\succeq P^{2}[\Phi_\epsilon^{c_1}](x).
\end{align*}
Since $P[\Phi](x)$ is increasing with respect to $\Phi$, it follows that
\begin{align*}
\Phi_\epsilon^{c_1}(x)\succeq P^{2}[\Phi_\epsilon^{c_1}](x)\succeq P^{2}[\Phi_{n}^2](x)=\Phi_{n+1}^2(x)
\end{align*} 
provided that
\begin{align*}
\Phi_\epsilon^{c_1}(x)\succeq \Phi_{n}^2(x).
\end{align*}
Recall that $\Phi_\epsilon^{c_1}(x)\succeq  \bm{\delta}\equiv \Phi_{0}^2(x)$. By induction, we obtain that $\Phi_\epsilon^{c_1}(x)\succeq \Phi_{n}^2(x)$ for all $n=0,1,2,\cdots$, and so $\Phi_\epsilon^{ c_1}(x)\succeq  \Phi_\epsilon^{c_2}(x)$.
\end{proof}

In later analysis of the paper, we will need the following variant of \eqref{2.1}:
\begin{equation}\label{2.9a}
		\begin{cases}
	\dd	D\circ\int_{-\yy}^{\yy}\mathbf{K} (x-y)\circ \Psi(y) {\rm d}y-D\circ\Psi+c\Psi'(x)+F(\Psi(x))=0,&
		-\yy<x<0,\\
		\Psi(-\yy)=\mathbf{\hat u}^*,\ \ \Psi(x)=\bm{\delta},&0\leq x<\yy,
		\end{cases}
		\end{equation}
		where $\bm{\delta}=\epsilon\Theta$ as in \eqref{2.1}, $K(x)=(K_i(x))$ satisfies {\bf (J)} except that $\int_{\R}K_i(x)dx=1$ $(i=1,..., m)$ is now replaced by
		\[
		1-\epsilon\leq \displaystyle\int_\R K_i(x) \rd x\leq 1,\;i.e., \|K_i\|_{L^\infty(\R)}\in [1-\epsilon, 1],
		\]
and $\mathbf{\hat u}^*$ is the positive constant equilibrium  of the first equation in \eqref{2.9a}.

\begin{lemma}\label{lemma2.6a}
	Suppose that  $\mathbf{(f_1)}$ and $\mathbf{(f_3)}$  hold. Then for all small $\epsilon>0$, problem \eqref{2.9a} 
			has a solution $\Psi(x)=(\psi_{i}(x))$ which is nonincreasing in $x$, and can be obtained by an interaction process.
		Moreover,	if $\mathbf{K}(x)\preceq \mathbf{J}(x)$ for all $x\in \R$, then 
	\begin{align*}
	\mathbf{\hat u}^*\preceq \mathbf{u}^*\ \mbox{ and }\ \ \Psi(x)\preceq \Phi(x) \ \mbox{ for }\ \ x\in\R,
	\end{align*}
	where $\Phi$ is the solution of \eqref{2.1} obtained in Lemma {\rm \ref{lemma2.2}}. 
\end{lemma}
\begin{proof} Define 
\[
\begin{cases}
\mbox{$\widehat D=(\hat d_i)$ with $\hat d_i:=d_i||K_i||_{L^1(\R)}$,}\\
\mbox{ $\mathbf{\widehat K}=(\widehat K_i)$ with 
$\widehat K_i:=\frac{1}{\|K_i\|_{L^1(\R)}}K_i$,}\\
\mbox{$\widehat F(u):=F(u)-(\widehat D-D)\circ u$ for $u\in \R^m$.}
\end{cases}
\]
 Then clearly $\mathbf{ \widehat K}$ satisfies all the conditions in {\bf (J)}, and
 \eqref{2.9a} can be rewritten as 
\begin{equation}\label{hatK}
\begin{cases}
\dd \widehat D\circ\int_{-\yy}^{\yy}\mathbf{ \widehat K} (x-y)\circ \Psi(y) {\rm d}y-\widehat D\circ\Psi+c\Psi'(x)+\widehat F(\Psi(x))=0,&
-\yy<x<0,\\
\Psi(-\yy)=\mathbf{\hat u}^*,\ \ \Psi(x)=\bm{\delta},&0\leq x<\yy.
\end{cases}
\end{equation}

  From $\mathbf{(f_1)}$ and $\mathbf{(f_3)}$ we easily see  that for all sufficiently small $\epsilon>0$, $\mathbf{\hat u}^*$ is the unique solution of $\widehat F(u)=\mathbf{0}$ in a small neighbourhood of $\mathbf{u^*}$. 
  Hence  $\widehat F$ satisfies  $\mathbf{(f_1)}$ with $\mathbf{\hat u}^*$ in place of $\mathbf{u}^*$.  The existence of $\Psi$ now follows from the proof of Lemma {\rm \ref{lemma2.2}} applied to \eqref{hatK}. More precisely,
  let 
\begin{align*}
\widehat \Omega:=\{\Gamma\in [C(\R)]^m: 0\preceq\Gamma\preceq \mathbf{\hat u}^*\}, 
\end{align*}
and
define  $\widehat P=(\hat P_i): \widehat\Omega\to  [C(\R)]^m$ by 
\begin{equation*}
\widehat P[\Gamma](x)=
\begin{cases}\dd
e^{Mx}\bm{\delta}+\frac{e^{Mx}}{c}\int_{x}^{0}e^{-M\xi}\lf[D\circ\int_{-\yy}^{\yy} \mathbf{K}(\xi-y)\circ\Gamma(y) {\rm d}y+\td\sigma(\Gamma(\xi)) \rr]{\rm d}\xi ,&x<0,\\
\bm{\delta},&x\geq  0,
\end{cases}
\end{equation*}
or equivalently,
\begin{equation*}
\widehat P[\Gamma](x)=
\begin{cases}\dd
e^{Mx}\bm{\delta}+\frac{e^{Mx}}{c}\int_{x}^{0}e^{-M\xi}\lf[\widehat D\circ\int_{-\yy}^{\yy} \mathbf{\widehat K}(\xi-y)\circ\Gamma(y) {\rm d}y+\td\sigma(\Gamma(\xi)) \rr]{\rm d}\xi ,&x<0,\\
\bm{\delta},&x\geq  0,
\end{cases}
\end{equation*}
where $\td\sigma$ is given by 
\begin{align*}
\td\sigma(v):=F(v)+cM v-D\circ v=\widehat F(v)+cM v-\widehat D\circ v\ \ \ {\rm for}\ v\in \R^m,
\end{align*}
with $M>0$ large enough such that $\td\sigma(v)$ is increasing in $v\in [0, \mathbf{\hat u^*}+\mathbf{1}]$.
Then
\[
\Psi(x)=\lim_{n\to\infty} \widehat P^n[\bm{\delta}](x).
\]

Now suppose  additionally $\mathbf{K}(x)\leq \mathbf{J}(x)$ for $x\in\R$.  Then
\begin{align*}
 \widehat P(u)\preceq P(u)\ \ \ {\rm for\ all}\ u\succeq \mathbf{0},
\end{align*} 
which together with the monotonicity of $\widehat P(u)$ and $P(u)$ in $u$ yields
\begin{align*}
 \widehat P^n[\bm{\delta}]\preceq P^n[\bm{\delta}], \ \ \ n=0,1,2,\cdots.
\end{align*}
Thus, $\Psi\preceq \Phi$, which implies $\mathbf{\hat u^*}\preceq\mathbf{u^*}$. 
\end{proof}

\subsection{A dichotomy between semi-waves and traveling waves}

\begin{theorem}\label{lemma2.4}
Suppose $\mathbf{(J)}$, $\mathbf{(f_1)}$, $\mathbf{(f_2)}$ and $\mathbf{(f_4)}$   hold. Then for each $c>0$, \eqref{2.3}  has either a  monotone  semi-wave solution with speed $c$  or a monotone traveling wave solution with speed $c$, but not both. Moreover, one of the following holds:
\begin{itemize}
	\item[{\rm (i)}]  For  every $c>0$, \eqref{2.3} has  a  monotone semi-wave  solution with speed $c$. 
		\item[{\rm (ii)}] There exists $C_*\in (0,\yy)$ such that  \eqref{2.3} has   a  monotone semi-wave  solution with speed $c$ for  every $c\in(0, C_*)$, and has a monotone traveling wave solution with speed $c$ for every $c\geq  C_*$. 
\end{itemize}
\end{theorem}
We prove Theorem \ref{lemma2.4} by the following two lemmas.

\begin{lemma}\label{lem2.9}
Suppose $\mathbf{(J)}$, $\mathbf{(f_1)}$ and $\mathbf{(f_2)}$  hold. Then for each $c>0$, \eqref{2.3}  has either a  monotone  semi-wave solution with speed $c$  or a monotone traveling wave solution with speed $c$, but not both.
\end{lemma}

\begin{proof}
	Let $\Phi_{n}^c=(\phi_{n,i}^c)$ be the solution of \eqref{2.1} defined in Lemma \ref{lemma2.2} with $\bm{\delta}=\bm{\delta}_n:=\epsilon_n\Theta$,
	$\epsilon_n\searrow 0$ as $n\to\infty$.  Then 
	\begin{align*}
	 x_{n}^c:=\max \lf\{ x:\phi_{n,1}^ c(x)=u_1^*/2\rr\}
	\end{align*}
	is well defined, and 
	\begin{align*}
	\phi_{n,1}^c(x_{n}^c)=u_1^*/2,\ \ \ \phi_{n,1}^c(x)<u_1^*/2\ \ \ \ {\rm for}\ x>x_{n}^c.
	\end{align*}
	Moreover, making use of Lemma \ref{lemma2.3}, we have
	\begin{equation}\label{2.8}
	\begin{cases}
	 0>x_{n}^{c}\geq   x_{m}^{c}\ \ & {\rm if}\ n\leq m,\\
	 0>x_{n}^{c_1}\geq  x_{n}^{c_2}\ \ & {\rm if}\ c_1\leq c_2.
	 \end{cases}
	\end{equation}
	Define
	\begin{align*}
	\wtd \Phi_{n}^{ c}(x):= \Phi_{n}^ {c}(x+x_{n}^{c}), \ \ \ \ x\in R.
	\end{align*}
	Then $\wtd \Phi_{n}^{ c}$ satisfies, for $x<-x_{n}^{c}$,
	\begin{align}\label{2.9}
		D\circ \int_{-\yy}^{\yy}\mathbf{J} (x-y)\circ \wtd \Phi_{n}^{ c}(y) {\rm d}y-D\circ\wtd \Phi_{n}^{ c}(x)+c(\wtd \Phi_{n}^{ c})'(x)+F(\wtd \Phi_{n}^{ c}(x))=0,
	\end{align}
	and for $x\geq -x_{n}^{c}$,
	$
	\wtd \Phi_{n}^{ c}(x)=\bm{\delta}_n$. Moreover,
	\[
	\td \phi_{n,1}^{ c}(0)={u_1^*}/2,
	\]
	 where $\td \phi_{n,1}^{ c}$ is the first element of the vector function $\wtd \Phi_{n}^{ c}$.  Since $x_{n}^{c}$ is nonincreasing in $n$, 
	 \[
	 x^c:=-\lim_{n\to\infty} x_n^c\in (0,\infty]
	 \]
	  always exists, and there are  two possible cases
	\begin{itemize}
		\item Case 1. $x^c=\infty$
		\item Case 2.  $x^c\in (0,\yy)$.
	\end{itemize}
	
	Clearly, for fixed $c>0$, $\wtd \Phi_{n}^{ c}$ and by the equation subsequently $(\wtd \Phi_{n}^{ c})'$ (for $x\neq -x_{n}^{c}$) are uniformly bounded in $n$. Then by the Arzela-Ascoli theorem and a standard argument involving a diagonal process of choosing subsequences, we see that
	$\{\wtd \Phi_{n}^{ c}\}_{n\geq 1}$
has	a subsequence, still denoted by itself for simplicity of notation, which converges to 
	some $\wtd \Phi^{c}=(\td \phi_i^c)\in C(\R)$  locally uniformly in $\R$. Moreover, $\wtd \Phi^{c}(x)$ is nonincreasing in $x$ with $\td \phi_1^{c}(0)=u_1^*/2$ . 
	
	If Case 1 happens, we easily see that  $\wtd \Phi^{c}$ satisfies 
	\begin{align}\label{2.10}
	D\circ \int_{-\yy}^{\yy}\mathbf{J} (x-y)\circ \wtd \Phi^{c}(y) {\rm d}y-D\circ\wtd \Phi^{c}(x)+c(\wtd \Phi^{ c})'(x)+F(\wtd \Phi^{c}(x))=0 \mbox{ for } x\in \R.
	\end{align}
In fact, from \eqref{2.9}, for $x\in\R$ and all large $n$ satisfying $x<-x_n^c$, we have
\begin{align*}
c\wtd \Phi_{n}^{ c}(x)-c\wtd \Phi_{n}^{ c}(0)=&-D\circ \int_{0}^{x} \lf[\int_{-\yy}^{\yy}\mathbf{J} (\xi -y)\circ \wtd \Phi_{n}^{ c}(y) {\rm d}y-D\circ\wtd \Phi_{n}^{ c}(\xi)+F(\wtd \Phi_{n}^{ c}(\xi))\rr] {\rm d}\xi.
\end{align*}
It then follows from the dominated convergence theorem that, for $x\in\R$, 
\begin{align*}
c\wtd \Phi^{c}(x)-c\wtd \Phi^{c}(0)=&-D\circ \int_{0}^{x} \lf[\int_{-\yy}^{\yy}\mathbf{J} (\xi -y)\circ \wtd \Phi^{ c}(y) {\rm d}y-D\circ\wtd \Phi^{c}(\xi)+F(\wtd \Phi^{c}(\xi))\rr] {\rm d}\xi,
\end{align*}
and  \eqref{2.10} thus follows by differentiating this equation.  Due to the monotonicity and boundedness of $\wtd \Phi^{c}(x)$,  the  arguments in step 3 of the proof of Lemma \ref{lemma2.2} can be repeated to give 
\begin{align*}
	\lim_{x\to -\yy}\lf[D\circ \int_{-\yy}^{\yy}\mathbf{J} (x-y)\circ \wtd \Phi^{c}(y) {\rm d}y-D\circ\wtd \Phi^{c}(x)\rr]=0,
\end{align*}
and so 
\begin{align*}
	\lim_{x\to -\yy} [c (\wtd \Phi_{ c})'(x)+F(\wtd \Phi^{c}(x))]=0.
\end{align*}
Denote $K:=\lim_{x\to-\yy}\wtd \Phi^{c}(x)\in \R_+^m$. Then we must have 
\begin{align*}
F(K)=\lim_{x\to -\yy} F(\wtd \Phi^{c}(x))=-\lim_{x\to -\yy} c( \wtd \Phi_{ c})'(x).
\end{align*}
This is possible only if $F(K)=0$. By  $\mathbf{(f_1)}$ either $K=\mathbf{0}$ or $K=\mathbf{u^*}$.  Since $\wtd \Phi^{c}(x)$ is nonincreasing in $x$ with $\td \phi_1^{c}(0)=u_1^*/2>0$, we have $K\succ \mathbf{0}$ and hence we must have $K=\mathbf{u^*}$.  An analogous analysis can be applied to show  $\lim_{x\to\yy}\wtd \Phi^{c}(x)=\mathbf{0}$. Therefore, $\wtd \Phi^{c}(x)$ is a monotone traveling  wave of \eqref{2.3} with speed $c$.

	If Case 2 happens, analogously for fixed  $x<x^c$,
\begin{align*}
c\wtd \Phi^{c}(x)-c\wtd \Phi^{c}(0)=&-D\circ \int_{0}^{x} \lf[\int_{-\yy}^{\yy}\mathbf{J} (\xi -y)\circ \wtd \Phi^{ c}(y) {\rm d}y-D\circ\wtd \Phi^{c}(\xi)+F(\wtd \Phi^{c}(\xi))\rr] {\rm d}\xi,
\end{align*}
and $\wtd \Phi^{c}(x)=0$ for $x\geq x^c$,  which yields 
\begin{equation*}
\begin{cases}
\dd	D\circ \int_{-\yy}^{x^c}\mathbf{J} (x-y)\circ \wtd \Phi^{c}(y) {\rm d}y-D\circ\wtd \Phi^{c}(x)+c(\wtd \Phi^{ c})'(x)+F(\wtd \Phi^{c}(x))={\bf 0} \mbox{ for} \ \ x<x^c,\\
	\wtd \Phi^{c}(x^c)={\bf 0}.
\end{cases}
\end{equation*}
Let $\Phi^{c}(x):=\wtd\Phi^{c}(x+x^c)$ for $x\leq 0$, then $\Phi^{c}(x)$ satisfies
\begin{equation*}
\begin{cases}
\dd D\circ \int_{-\yy}^{0}\mathbf{J} (x-y)\circ  \Phi^{c}(y) {\rm d}y-D\circ \Phi^{c}(x)+c(\Phi^{ c})'(x)+F(\Phi^{c}(x))={\bf 0} \mbox{ for } \ x<0,\\
\wtd \Phi^{c}(0)=\mathbf{0}.
\end{cases}
\end{equation*}
Moreover, as in Case 1, we can show $\lim_{x\to-\yy}\Phi^{c}(x)=\mathbf{u}^*$. Therefore, $ \Phi^{c}(x)$ is a monotone semi  wave solution of \eqref{2.3} with speed $c$.

We have thus proved that for any $c>0$,  \eqref{2.3} has either a monotone traveling wave solution with speed $c$ or a monotone semi-wave solution with speed $c$. 
We show next that for any given $c>0$, \eqref{2.3} cannot have both.

  Suppose, on the contrary, there is $c_0>0$ such that  \eqref{2.3} admits a monotone traveling wave solution $\Psi=(\psi_i)$ with speed $c_0$ and also a monotone semi-wave solution $\Phi=(\phi_i)$ with speed $c_0$. We are going to drive a contradiction. 
  
  Let $\wtd \Phi(x):=k \Phi(x)$ for some fixed $k\in (0,1)$. Then by $\mathbf{(f_2)}$, $\wtd \Phi=(\tilde\phi_i)$ satisfies
	\begin{equation*}
	\begin{cases}
\dd	D\circ \int_{-\yy}^{\yy}\mathbf{J} (x-y)\circ \wtd \Phi(y) {\rm d}y-D\circ \wtd \Phi(x)+c \wtd \Phi(x)+F(k \Phi(x))\succeq {\bf 0},&x<0,\\
	\wtd \Phi(-\yy)=k\mathbf{u}^*,\ 	\wtd \Phi(x)={\bf 0},& x\geq 0.
	\end{cases}
	\end{equation*}
For $\beta\in \R$, define 
\begin{align*}
\Psi^\beta(x):=\Psi (x+\beta),\ \ \ W^\beta(x)=(w_i^\beta(x)):=\Psi^\beta(x)-\wtd \Phi(x),  \ \ \ x\in \R.
\end{align*}	
For fixed $x\leq 0$ and $i\in\{1,..., m\}$,
\[
w_i^\beta(x)\geq \psi_i(\beta)-k \phi_i(x)\geq \psi_i(\beta)-ku^*_i\to (1-k)u^*_i>0 \mbox{ as } \beta\to-\infty.
\]
Therefore there exists $\bar\beta_i\ll 0$ independent of $x$ such that
\[
w_i^\beta(x)>0 \mbox{ for } x\leq 0,\;\beta\leq\bar\beta_i.
\]
On the other hand, 
\[
w_i^\beta(-1)=\psi_i(\beta-1)-k\phi_i(-1)\to-k\phi_i(-1)<0 \mbox{ as } \beta\to\infty.
\]
Therefore we can find $\beta_i^*\in \R$ such that
\[
h_i(\beta):=\inf_{x\leq 0}w_i^\beta(x)> 0 \mbox{ for }  \beta< \beta_i^*,\; h_i(\beta_i^*)=0.
\]
Clearly $w_i^{\beta_i^*}(-\infty)=(1-k)u^*_i>0$ and $w_i^{\beta_i^*}(0)=\psi^{\beta_i^*}_i(0)>0$. Therefore due to the continuity of $w_i^\beta(x)$ there exists $x_i^0\in (-\infty, 0)$ such that
$w_i^{\beta_i^*}(x_i^0)=0$. We can thus conclude that
\[
w_i^\beta(x)\geq 0 \mbox{ for } x\leq 0, \beta\leq \beta_i^*,\; \mbox{ and } w_i^{\beta_i^*}(x_i^0)=0.
\]
Let $\beta^*=\min\{\beta_i^*, ..., \beta_m^*\}$. Then there exists $i_0\in\{1,..., m\}$ such that $\beta^*=\beta_{i_0}^*$, and hence we have
\begin{equation}\label{2.11}
W^{\beta^*}(x)\succeq \mathbf{0} \mbox{ for } x\leq 0,\; \ \ w_{i_0}^{\beta^*}(x_{i_0}^0)=0.
\end{equation}
 
 By the definition of $\Psi$ and $\Phi$, we see that $W^{\beta^*}$ satisfies
	\begin{equation*}
\begin{cases}
\dd D\circ \int_{-\yy}^{\yy}\mathbf{J} (x-y)\circ W^{\beta^*}(y) {\rm d}y-D\circ W^{\beta^*}(x)
+c {W^{\beta^*}}'(x)\\
\ \ \ \ \ \ \ \ \ \ \ \ \ \ \ \ \ \ \ \ \ \ \ \ \ \ \ \ \ \ \ \ \ \ \ \ \ \ \ \ +F( \Psi^{\beta^*}(x))-F(k \Phi(x))\preceq \mathbf{0},&x<{0},\\
W^{\beta^*}(-\infty)=(1-k)\mathbf{u^*}\ggs \mathbf{0},\; W^{\beta^*}(x)\succeq \mathbf{0},& x\in \R.
\end{cases}
\end{equation*}
We have
\begin{align*}
F( \Psi^{\beta^*}(x))-F(k \Phi(x))= [W^{\beta^*}(x)]E(x)
\end{align*}
with $E(x)=(e_{ij}(x))$ an $m\times m$ matrix given by
\[
e_{ij}(x)=\int_0^1\partial_if_j(k\Phi(x)+tW^{\beta^*}(x))dt.
\]
By $\mathbf{(f_1)}$ we have $e_{ij}\geq 0$ for $i\not=j$.
 This allows us to use Lemma \ref{lemma2.7a} to conclude that $W^{\beta^*}(x)\ggs \mathbf{0}$ for $x<0$, which contradicts the second part of \eqref{2.11}. 
This completes the proof.
\end{proof}

\begin{lemma}\label{lem2.10}
In Lemma \ref{lem2.9}, if further $\mathbf{(f_4)}$ is satisfied,  then  one of the following holds:
\begin{itemize}
	\item[{\rm (i)}]  For  every $c>0$, \eqref{2.3} has  a  monotone semi-wave  solution with speed $c$. 
		\item[{\rm (ii)}] There exists $C_*\in (0,\yy)$ such that  \eqref{2.3} has   a  monotone semi-wave  solution with speed $c$ for  every $c\in(0, C_*)$, and has a monotone traveling wave solution with speed $c$ for every $c\geq  C_*$. 
\end{itemize}
\end{lemma}
\begin{proof}
From \eqref{2.8} we see that $x^c$ is nonincreasing in $c$ and hence there are three possible cases:
\begin{itemize}
	\item[{\rm (1)}]  For any $c>0$, $x^c<\yy$.
	\item[{\rm (2)}] There is $C_*>0$ such that  $x^c<\yy$  for any $c\in (0, C_*)$, and $x^c=\yy$  for any $c>C_*$.
	\item[{\rm (3)}] For any $c>0$, $x^c=\yy$.
\end{itemize}
 
 From the proof of Lemma \ref{lem2.9}, we know that in case (1), \eqref{2.3} has a monotone semi-wave with speed $c$ for any $c>0$, and in case (2), it has a monotone semi-wave with speed $c$ for each $c\in (0, C_*)$, has a monotone traveling wave with speed $c$ for each $c>C_*$. Therefore to complete the proof  it suffices to show that case (3) does not happen, and in case (2),  \eqref{2.3} has a monotone traveling wave solution with speed $c=C_*$.

We prove the latter first.
Let $\Psi^c=(\psi^c_i)$ be a monotone traveling wave solution  of \eqref{2.3} with speed $c>C_*$. By a suitable translation we may assume $\psi_1^c(0)=u_1^*/2$. Since $\Psi^c$ is uniformly bounded,  by the equation satisfied by $\Psi^c$ we see that $(\Psi^c)'$ is also uniformly bounded in $c$ for $c>C_*$.  Then by the Arzela-Ascoli theorem and a standard argument involving a diagonal process of choosing subsequences, 
for any sequence $c_n\searrow C_*$, $\{\Psi^{c_n}\}_{n=1}^{\yy}$ has a subsequence, still denoted by itself, which converges to some
 $\Psi=(\psi_i)\in [C(\R)]^m$  locally uniformly in $\R$ as $n\to\infty$. Similar to the proof of Lemma \ref{lem2.9}, we can check at once that $\Psi$ satisfies
\begin{equation*}
\begin{cases}
\dd D\circ \int_{-\yy}^{\yy}\mathbf{J} (x-y)\circ \Psi(y) {\rm d}y-D\circ \Psi(x)+C_*\Psi'(x)+F(\Psi(x))={\bf 0},&
x\in \R,\\
\psi_1(0)=u_1^*/2.
\end{cases}
\end{equation*}
Making use of  the monotonicity of $\Psi(x)$ inherited from $\Psi_n^c(x)$, we can use the method in Step 3 of the proof of Lemma \ref{lemma2.2} to show that 
\begin{align*}
\Psi(-\yy)=\mathbf{u}^*,\ \ \Psi(\yy)={\bf 0},
\end{align*}
which implies that $\Psi$ is a monotone traveling wave solution of \eqref{2.3} with speed $c=C_*$.

We finally show that case (3) does not happen. 
Again we argue by contradiction. Suppose $x^c=\infty$ for all $c>0$. Then from our earlier argument \eqref{2.3} has a monotone traveling wave solution $\Psi^c=(\psi^c_i)$ with speed $c$ for any $c>0$.   By a suitable translation, we may assume $\psi_1^c(0)=u_1^*/2$. 
We choose a sequence $\{c_n\}$ such that $c_n\searrow 0$ as $n\to\infty$ and consider the sequence of monotone functions $\{\Psi^{c_n}\}$.
By Helly's selection theorem and a standard diagonal process of choosing subsequences, $\{\Psi^{c_n}\}$ has a subsequence, still denoted by itself, which converges to  a nonincreasing function $\Psi=(\psi_i)$  pointwisely in $\R$.   Thus $\psi_1(0)=u_1^*/2$ and
\begin{align*}
x(\psi_1(x)- u_1^*/2)\leq 0  \mbox{ for}\ x\in\R.
\end{align*}
Since $\psi_1$ is nonincreasing, it is continuous in $\R$ except at some countable set (possibly empty) where  jumping discontinuities may occur. Therefore we can find an interval $[a,b]$ with $a<b<0$ such that $\psi_1$ is continuous on $[a,b]$. Then from  \cite[Lemma 7]{yagi} we see that $\psi_1^{c_n}(x)$ converges to $\psi^1(x)$ uniformly for $x\in [a,b]$.  

We now choose  a function $\gamma_1\in C(\R)$ satisfying $0\leq \gamma_1(x)\leq u_1^*/4$ in $\R$ and whose  support is exactly  $[a,b]$.
 Then by the above properties of $\psi_1$ we have
\begin{align*}
\gamma_1(x)\leq \frac{1}{2} \psi_1(x) \ \mbox{ for } \ x\in \R,
\end{align*}
and hence, due to the uniform convergence of $\psi_1^{c_n}(x)$ to $\psi_1(x)$ over $[a,b]$, there is large $N>0$ such that 
\begin{align*}
\gamma_1(x)\leq \psi_1^{c_n}(x) \ \mbox{ for }\ \ x\in \R,\; n>N.
\end{align*}

Define  $\wtd U_0=(v_i)\in [C(\R)]^m$ with $v_1(x)=\gamma_1(x)$ and $v_i(x)\equiv 0$ for $i\neq 1$. Let $U(t,x)=(u_i(t,x))$ be the solution of \eqref{2.3} with  $U(0,x)=\wtd U_0(x)$. Then we  have 
\begin{align}\label{2.10a}
U(0,x)=\wtd U_0(x)\preceq \Psi^{c_n}(x) \ \mbox{ for }\ \ x\in \R,\; n>N,
\end{align}
and by $\mathbf{(f_4)}$,
\begin{align*}
\lim_{t\to\yy} U(t,x)=\mathbf{u}^* \mbox{ for }\ x\in\R.
\end{align*}  
It follows  that 
\begin{align}\label{2.11a}
 u_1(t,x_1)\geq \frac{2u_1^*}{3} \mbox{ for all large $t$, say }   t\geq T_1>0,
\end{align}
 where $x_1=\frac{a+b}2\in [a,b]$.

On the other hand, clearly $U^n(t,x):=\Psi^{c_n}(x-c_nt)$ is the solution of \eqref{2.3} with initial function $\Psi^{c_n}(x)$. Since $U(0,x)\leq U^n(0,x)$ for $n>N$ by \eqref{2.10a}, by Lemma \ref{lemma2.3a} we have 
\begin{align*}
U(t,x)\preceq  U^n(t,x)=\Psi^{c_n}(x-c_nt) \ \mbox{ for } \  t\geq 0,\ x\in \R,\ n>N.
\end{align*} 
In particularly, 
\begin{align*}
u_1(T_1,x_1)\leq   \psi_1^{c_n}(x_1-c_nT_1) \mbox{ for } n>N.
\end{align*}
Letting $n\to \yy$, we obtain 
\begin{align*}
u_1(T_1,x_1)\leq  \psi_1(x_1)\leq\psi_1(0)= u_1^*/2,
\end{align*}
which contradicts \eqref{2.11a}.  The proof  is now complete.
\end{proof}

\subsection{Sharp criteria for the dichotomy in Theorem \ref{lemma2.4}} From Theorem \ref{lemma2.4} we see that \eqref{2.3} can have a monotone traveling wave solution with some positive speed $c$ if and only if alternative (ii) happens,  and in that case, \eqref{2.3} has a monotone traveling wave solution with speed $c$ if and only if $c\geq C_*$. 
In this subsection, we prove the following sharp criteria for alternative (ii) to happen. 
\begin{theorem}\label{TW-J2}
In Theorem \ref{lemma2.4}, suppose additionally $\mathbf{(f_3)}$ holds. Then alternative {\rm (ii)} happens if and only if the kernel functions satisfy  $\mathbf{(J_2)}$.
\end{theorem}

We note that under the assumptions of Theorem \ref{lemma2.4}, it is already known in the literature that alternative (ii) happens when condition $\mathbf{(J_2)}$ holds; see, for example, Theorems 3.6 and 3.7 of \cite{HKLL}. Therefore, we only need to show that if $\mathbf{(J_2)}$ does not hold, then 
\eqref{2.3} does not have a monotone traveling wave solution with some positive speed $c$.
To this end, 
we first apply the method in \cite{cc2004} to obtain an estimate for any monotone traveling wave solution of \eqref{2.3}
when the kernel functions satisfy the following additional condition:
\begin{equation}\label{J3}
	\dd \int_{\R} |x|J_i(x) \rd x<\yy \mbox{  for every $i\in\{1,..., n\}$.}
\end{equation}
\begin{lemma}\label{lemma2.13}
	Suppose  $\mathbf{(J)}$, \eqref{J3} and $\mathbf{(f_1)}$ are satisfied. If \eqref{2.3} admits a monotone traveling wave solution $\Psi=(\psi_i)$ with some positive speed $c$, then there is $\alpha>0$ such that $\psi_i(x)=O(e^{-\alpha x})$ as $x\to \yy$, for every $i\in\{1,..., m\}$.
\end{lemma}
\begin{proof}	We first show that $\psi_i(x)$ is integrable on $[0,\yy)$ for every $i\in\{1,..., m\}$.  
	
	By \eqref{2.3c}, 
	\[
	\tilde\psi(x):=\sum_{i=1}^m \td\theta_i\psi_i(x)
	\]
	satisfies, for  $y>x\gg 1$,
	\begin{align*}
	c(\wtd \psi(x)-\wtd \psi(y))=&\sum_{i=1}^m \int_{x}^{y}\lf[\td\theta_id_i  J_i\ast \psi_i(z)-\td\theta_i d_i   \psi_i(z) +\td\theta_i  f_i(\Psi(z))\rr]{\rm d}z\\
	\geq& \sum_{i=1}^m \int_{x}^{y}\lf[\td\theta_i d_i  J_i\ast \psi_i(z)-\td\theta_i d_i \psi_i(z) +b_i \psi_i(z)\rr]{\rm d}z,
	\end{align*} 
	By \eqref{J3} and  a direct calculation, we have
	\begin{align*}
	& \lf|\int_{x}^{y} \lf(\int_{\R} J_i(z-w)\psi_i(w){\rm d}w-\psi_i(z)\rr){\rm d}z\rr|= \lf|\int_{x}^{y} \int_{\R} J_i(w)(\psi_i(z+w)-\psi(z)) {\rm d}w{\rm d}z\rr|\\
	=& \lf|\int_{x}^{y} \int_{\R} J_i(w) \int_{0}^1 w\psi'_i(z+sw) {\rm d}s{\rm d}w{\rm d}z\rr|=\lf|\int_{\R} wJ_i(w) \int_{0}^1 [\psi_i(y+sw)-\psi_i(x+sw)] {\rm d}s{\rm d}w\rr|\\
	\leq &  u^*_i \int_{\R} |y|J_i(y) {\rm d}y=:M_i<\yy \mbox{ for } i\in\{1,..., n\}.
	\end{align*}
	Thus, for all $y>x\gg 1$,
	\begin{align*}
	\sum_{i=1}^m \int_{x}^{y}  b_i \psi_i(z) {\rm d}z\leq c(\wtd \psi(x)-\wtd \psi(y))+\sum_{i=1}^n (\td\theta_id_iM_i)\leq  c\wtd \psi(x)+\sum_{i=1}^n (\td\theta_id_iM_i),
	\end{align*}
	which indicates that $\dd\int_{x}^{\yy}  \psi_i(z) {\rm d}z<\yy$  and hence $\dd\int_{0}^{\yy}  \psi_i(z) {\rm d}z<\yy$ for $i\in\{1,..., m\}$. This completes the proof of the desired conclusion.
	
	Moreover, by letting $y\to \yy$, $\wtd \psi$ satisfies
	\begin{align*}
	c\wtd \psi(x)
	\geq \sum_{i=1}^m \int_{x}^{\yy}\lf[\td\theta_i  d_i J_i\ast \psi_i(z)-\td\theta_i  d_i  \psi_i(z) +b_i \psi_i(z)\rr]{\rm d}z,
	\end{align*} 
	which, together with the integrability of $\psi_i$  proved above,  implies  that $J_i\ast \psi_i(z)$ is integrable on $(x,\yy)$ and hence on $[0,\infty)$. 
	
	We also have
	\begin{align*}
	&\int_{x}^{\yy} [J_i\ast \psi_i(z)- \psi_i(z)] {\rm d}z=\int_{x}^{\yy} \int_{\R} J_i(w) \int_{0}^1 w\psi'_i(z+sw) {\rm d}s{\rm d}w{\rm d}z\\
	=&-\int_{\R} wJ_i(w) \int_{0}^1 \psi_i(x+sw){\rm d}s{\rm d}w=-\int_{\R} J_i(w) \int_{0}^w \psi_i(x+s) {\rm d}s{\rm d}w\\
	=& \int_0^\infty wJ_i(w)\int_0^w\Big[ \psi_i(x-s)-\psi_i(x+s)\Big]dsdw\geq 0
	\end{align*}
	 since $\psi_i(x)$ is nonincreasing in $x$. It follows that, for  $b_*:=\min_{1\leq i\leq m}b_i/\theta_i>0$,  $ x\gg1$ and $ r>0$,
	\begin{align*}
	c\wtd  \psi(x)  \geq  \sum_{i=1}^m \int_{x}^{\yy} b_i   \psi_i(z) {\rm d}z\geq b_* \int_{x}^{\yy} \wtd  \psi(z) {\rm d}z\geq b_*\int_{x}^{x+r} \wtd \psi(z) {\rm d}z\geq b_*r  \wtd \psi(x+r).
	\end{align*}
	Let $\phi(x):=\wtd  \psi(x) e^{\alpha x}$ with $\alpha=\frac{1}{r}\ln(\frac{b_*r}{c})$ for some large  $r$ such that $\alpha>0$. Then for all large $x$, say $x\geq X\gg 1$,
	\begin{align*}
	\phi(x+r)=\wtd  \psi(x+r) e^{\alpha (x+r)}\leq \frac{c}{b_*r} \wtd\psi(x) e^{\alpha x}e^{\alpha r}=\wtd\psi(x) e^{\alpha x}=\phi(x).
	\end{align*}
	Hence
	\[
	M^*:=\sup_{x\geq X}\phi(x)=\max_{x\in[X, X+r]}\phi(x)<\infty.
	\]
	 Thus  $\wtd\psi(x)\leq M^* e^{-\alpha x}$ for $x\geq X$ which implies $\psi_i(x)=O(e^{-\alpha x})$ as 
	 $x\to \yy$ for every $i\in\{1,..., m\}$.
\end{proof}

{\bf Proof of Theorem \ref{TW-J2}:} As mentioned earlier, we only need to show that if $\mathbf{(J_2)}$ does not hold, then 
\eqref{2.3} does not have a monotone traveling wave solution with some positive speed $c$.

We argue indirectly. Suppose that $\mathbf{(J_2)}$ does not hold, but
\eqref{2.3} has a monotone traveling wave solution $\Phi=(\phi_i)$ with some positive speed $c$. We aim to derive a contradiction.

For fixed small $\epsilon>0$, let $M$ be a large constant such that 
\begin{align*}
\int_{|x|\leq M} J_i(x)dx\geq 1-\epsilon \mbox{ for every } i\in\{1,...,n\}.
\end{align*}
 Define 
	\begin{equation*}
K_i(x):=
	\begin{cases}
J_i(x) &\mbox{ for } |x|\leq M,\\
	\frac{M}{|x|}J_i(x)& \mbox{ for } |x|\geq M,\end{cases} \;\; \ \ \ \mbox{ and } \   \wtd K_i(x):=\frac{1}{\|K_i\|_{L^1(\R)}}K_i(x).
	\end{equation*}
	 Then 
	\begin{align}\label{2.20}
	K_i(x)\leq J_i(x)\ \mbox{ for } \ \  x\in \R,\; 1\leq i\leq n,
	\end{align}
	and  $\wtd K_i$ $(i=1,..., n)$ satisfy $\mathbf{(J)}$ and \eqref{J3}. 
	
	Since $\mathbf{(J_2)}$ is not satisfied, for  any $\lambda>0$,
	\begin{equation}\label{2.21}
	\begin{aligned}
	\sum_{i=1}^n\int_{\R} e^{\lambda x} K_{i}(x)  d x&\ \geq \sum_{i=1}^n\int_M^\infty J_i(x)\frac{e^{\lambda x}}{x}dx\\
	&\geq \sum_{i=1}^n\int_{\wtd M}^\infty J_i(x)e^{\frac{\lambda}{2} x}dx=\yy,
	\end{aligned}
	\end{equation}
	provided that $\wtd M>M$ is so large that $x<e^{\frac\lambda 2 x}$ for $x\geq \wtd M$.
	
We now show that there is a monotone  traveling wave  solution to \eqref{2.3} with $(J_i)$ replaced by $(K_i)$. To this end,
we first consider the  perturbed problem \eqref{2.1} with kernel functions $J=(J_i)$ and $\bm{\delta}=\epsilon_n\Theta$ as in the proof of   Lemma \ref{lem2.9}. From that proof and our assumption that \eqref{2.3} has a monotone traveling wave solution with speed $c$,
we obtain two sequences $\{\Phi_{n}^{c}\}$ and $\{x_{n}^{c} \}$, with the properties 
	\begin{align*}
	-x_{n}^{c}\to \yy \mbox{ and } \Phi_n^c(x)\to\Phi(x) \mbox{ locally uniformly  for } x\in\R
	\end{align*}
	as $n\to \yy$.
	 where $x_{n}^{c}$ is  determined by  the kernel functions $J_i$ and  a decreasing sequence $\{\bm{\delta}_n\}_{n=1}^\yy$.  
	 
	 Next we consider the  perturbed problem \eqref{2.1} with $(J_i)$ replaced by $(K_i)$ and with $\bm{\delta}=\epsilon_n\Theta$.
	  By \eqref{2.20}, Lemmas \ref{lemma2.6a} and \ref{lem2.9}  we can similarly obtain two sequences $\wtd \Phi_{n}^{c}$ and $\{\td x_{n}^{c}\}_{n=1}^\yy$, and we have $\wtd \Phi_{n}^{c}(x)\preceq \Phi_{n}^{c}(x)$. Then by the definitions of $x_{n}^{c}$ and $\td x_{n}^{c}$  we deduce
	$
	\td x_{n}^{c}\leq x_{n}^{c}$, 
	and so $-\td x_{n}^{c}\to \yy$ as $n\to\infty$. We may now repeat the argument in the proof of Lemma \ref{lemma2.4} to
	see that, by passing to a subsequence, $\wtd \Phi_{n}^{c}(x)$ converges locally uniformly to some $\wtd \Phi(x)=(\td\phi_i)$ in $\R$, and $\wtd \Phi$ is a monotone traveling wave solution of \eqref{2.3} with speed $c$ but with  the kernel functions $(K_i)$ in place of $(J_i)$, as we wanted.
	
	We are now ready to derive a contradiction by making use of 
	 Lemma \ref{lemma2.13}, which implies that  $\tilde \phi_i(x)=O(e^{-\alpha x})$ as $x\to \yy$ for every $i\in\{1,..., m\}$ and some $\alpha>0$.
By \eqref{2.21}, there exists $j\in\{1,..., n\}$ such that
	\begin{align}\label{2.12}
	\int_{\R} e^{\frac{\alpha}{2} x} K_j(x) {\rm d}y=\yy.
	\end{align}
	From \eqref{1.3a} and the exponential decay estimates for $\tilde\phi_i(x)$ $(i=1,..., m)$ near $x=+\infty$ we obtain, for $\lambda\in(0, \alpha)$,
	\begin{align*}
	\lf|\int_{\R} e^{\lambda x}f_j(\wtd\Phi(x)){\rm d}x\rr|\leq L\int_{\R} e^{\lambda x}\sum_{i=1}^{m}\td\phi_i(x){\rm d}x<\yy,
	\end{align*}
	and hence by the equation satisfied by $\td\phi_j$, 
	\begin{align*}
	d_j\int_{\R} e^{\lambda x}  K_j\ast \td\phi_j(x) {\rm d} x=&\int_{\R}  e^{\lambda x} [d_j \td\phi_j(x)-c\td\phi_j'(x)-f_j(\wtd\Phi(x))]{\rm d}x\\
	=\ &(d_j+c\lambda) \int_{\R} e^{\lambda x} \td\phi_j(x){\rm d}x- \int_{\R} e^{\lambda x}f_j(\wtd\Phi(x)){\rm d}x\\
	<\ &\yy.
	\end{align*}
	By the  Fubini-Tonelli theorem, for any finite numbers $L_1<L_2$ we have
	\begin{align*}
	\int_{\R} e^{\lambda x}  K_j\ast \td\phi_j(x) {\rm d} x&\ =\int_{\R} e^{\lambda x}  \int_{\R} K_j(y) \td\phi_j(x-y){\rm d} y {\rm d} x\\
	&\ \geq  \int_{\R} e^{\lambda x}  \int_{L_1}^{L_2} K_j(y) \td\phi_j(x-y){\rm d} y {\rm d} x\\
	& =\int_{\R}   \int_{L_1}^{L_2} [K_j(y)e^{\lambda y}] [\td\phi_j(x-y)e^{\lambda (x-y)}] {\rm d} x {\rm d} y\\
	& =\int_{L_1}^{L_2}  K_j(y)e^{\lambda y}{\rm d} y   \int_{\R} \td\phi_j(x)e^{\lambda x}{\rm d} x.
	\end{align*}
	It follows that
	\[
	\int_{L_1}^{L_2}  K_j(y)e^{\lambda y}{\rm d} y\leq M_0:=\int_{\R} e^{\lambda x}  K_j\ast \td\phi_j(x) {\rm d} x\Big/  \int_{\R} \td\phi_j(x)e^{\lambda x}{\rm d} x<\infty,
	\]
	which 
	contradicts \eqref{2.12} since $L_1<L_2$ are arbitrary and we can take $\lambda=\alpha/2$.
\hfill $\Box$

\subsection{Uniqueness and strict monotonicity of semi-wave solutions to \eqref{2.3}}
\begin{theorem}\label{lemma2.12}
Suppose that $\mathbf{(J)}$, $\mathbf{(f_1)}$ and $\mathbf{(f_2)}$ hold. Then for any $c>0$, \eqref{2.3}  has  at most one monotone semi-wave solution $\Phi=\Phi_c$ with speed $c$, and when exists, $\Phi_c(x)$ is strictly decreasing in $x$ for $x\in(-\infty, 0]$. Moreover, if $\Phi_{c_1}$ and $\Phi_{c_2}$ both exist and $0<c_1<c_2$, then $\Phi_{c_1}(x)\ggs \Phi_{c_2}(x)$ for fixed $x<0$.
\end{theorem}
\begin{proof}  Assume that $\Phi^{(1)}=(\phi_i^{(1)})$ and $\Phi^{(2)}=(\phi_i^{(2)})$ are  monotone semi-wave solutions of  \eqref{2.3}
with speed $c>0$. We want to show that $\Phi^{1}\equiv \Phi^{2}$. 

{\bf Claim 1}.   $\phi_i^{(k)}(x)>{0}$ for $x<0$, $k=1,2$ and $i\in\{1,..., m_0\}$.
 
 By $\mathbf{(f_1)}$,  we can write, for $k=1,2$, 
 \begin{align*}
 F(\Phi^{(k)}(x))=\Phi^{(k)}(x)E^{(k)}(x)
 \end{align*}
 where  $E^{(k)}(x)=(e^{(k)(x)}_{ij})$ is a matrix function with $e_{ij}\geq 0$ for $i\neq j$. Due to $\mathbf{0}\preceq \Phi^{(k)}(x)\preceq \mathbf{u}^*$ and $ \Phi^{(k)}(-\yy)=\mathbf{u}^*$, we can apply Lemma \ref{lemma2.7a} to conclude that $\phi_i^{(k)}(x)<{0}$ for $x<0$
 and $i\in\{1,..., m_0\}$.

{\bf Claim 2}. $(\phi_i^{(k)})'(0^-)<0$ for $k=1,2$ and $i\in\{1,..., m_0\}$;\ \ $(\phi_i^{(k)})'(0^-)=0<(\phi_i^{(k)})''(0^-)$ for $k=1,2$ 
and $i\in\{m_0+1,..., m\}$.

From the equation satisfied by $\Phi^{(k)}$, we deduce, for $k=1,2$, 
\begin{equation}\label{2.18}
\begin{aligned}
&{\Phi^{(k)}}'(0^-)=\lim_{x\to 0^-}\frac{{\Phi^{(k)}}(x)}{x} \\
=&\lim_{x\to 0^-}\frac{1}{cx}\int_{0}^{x}\lf[-D\circ\int_{-\yy}^{0}\mathbf{J} (z-y)\circ \Phi^{(k)}(y) {\rm d}y+D\circ\Phi^{(k)}(z)-F(\Phi^{(k)}(z))\rr]\rd z\\
=&-\frac{1}{c}D\circ\int_{-\yy}^{0}\mathbf{J} (y)\circ \Phi^{(k)}(y) {\rm d}y.
\end{aligned}
\end{equation}
Hence $(\phi_i^{(k)})'(0^-)<0$ for $k=1,2$ and $i\in\{1,..., m_0\}$;\ \ $(\phi_i^{(k)})'(0^-)=0$ for $k=1,2$ 
and $i\in\{m_0+1,..., m\}$. Moreover, for $k=1,2$ 
and $i\in\{m_0+1,..., m\}$, from
\[
c(\phi_i^{(k)})'(x)=-f_i(\Phi^{(k)}(x))
\]
we deduce
\[
c(\phi_i^{(k)})''(0^-)=-\sum_{j=1}^m \partial_jf_i(\mathbf{0})(\phi_j^{(k)})'(0^-)>0.
\]

Now with the help of Claim 1 and Claim 2, we are ready to define 
\begin{align*}
p:= \inf\{\rho\geq 1: \rho\Phi^{(1)}(x)\succeq \Phi^{(2)}(x)\ {\rm for}\ x\leq 0\}.
\end{align*}
Since  $\Phi^{(k)}(-\yy)=\mathbf{u}^*\ggs \mathbf{0}$ for $k=1,2$, and for each $i\in\{1,..., m\}$, $\frac{\phi_i^{(1)}(x)}{\phi_i^{(2)}(x)}$ is uniformly  bounded for  $x$ in a small left neighbourhood of 0 by Claim 2,  we see that
$p\in [1,\yy)$ is well-defined, and $p\Phi^{(1)}(x)\succeq \Phi^{(2)}(x)\ {\rm for}\ x\leq 0$. 

{\bf Claim 3:} $p=1$. 

Otherwise $p>1$
  and from the definition of $p$ we can find some $j\in\{1,..., m\}$ and a sequence $x_n\in (-\infty, 0)$ such that
  \begin{equation}\label{=p}
  \lim_{n\to\infty}\frac{\phi_{j}^{(2)}(x_n)}{\phi_{j}^{(1)}(x_n)}=p>1.
  \end{equation}
  From $\Phi^{(k)}(-\yy)=\mathbf{u}^*\ggs \mathbf{0}$ for $k=1,2$ we see that $\{x_n\}$ must be a bounded sequence, and hence by passing to a subsequence, we may assume that $x_n\to x_*\in (-\infty, 0]$ as $n\to\infty$. Define 
\begin{align*}
V(x)=(v_i(x)):=p\Phi^{(1)}(x)-\Phi^{(2)}(x).
\end{align*}
Clearly $V(x)\succeq \mathbf{0}$ for $x\leq 0$. Our discussion below is organised according to the following two possibilities:
\begin{itemize}
	\item Case 1. $V(x)\ggs \mathbf{0}$ for all $x< 0$.
	\item Case 2.  There exist $i_0\in\{1,..., m\}$ and $x_0<0$ such that $v_{i_0}(x_0)=0$.
\end{itemize}
In Case 1, 
 from \eqref{2.18} we obtain
\[
V'(0^-)=-\frac{1}{c}D\circ\int_{-\yy}^{0}\mathbf{J} (0-y)\circ V(y) {\rm d}y,
\]
which implies $v_i'(0^-)<0$ for $i\in\{1,..., m_0\}$ and $v_i'(0^-)=0$ for $i\in\{m_0+1,..., m\}$.
Moreover, for $i\in \{m_0+1,..., m\}$,
\[
cv_i''(0^-)=-\sum_{j=1}^m \partial_jf_i(\mathbf{0})v_j'(0^-)>0.
\]
Let us examine the sequence $\{x_n\}$ in \eqref{=p}. We have $x_n\to x_*\in (-\infty, 0]$. If $x_*<0$ then we deduce
$v_j(x_*)=0$ which is a contradiction to $V(x)=(v_i(x))\ggs \mathbf{0}$ for $x<0$. Therefore we must have $x_*=0$ and so $x_n\to 0$ as $n\to\infty$. It then follows that
\[
 \lim_{n\to\infty}\frac{\phi_{j}^{(2)}(x_n)}{\phi_{j}^{(1)}(x_n)}=\frac{(\phi_j^2)'(0^-)}{(\phi_j^1)'(0^-)}<p \mbox{ if } j\in\{1,..., m_0\}
\]
due to $v_j'(0^-)<0$ and  $(\phi_j^{(k)})'(0^-)<0$ for $k=1,2$, and
\[
 \lim_{n\to\infty}\frac{\phi_{j}^{(2)}(x_n)}{\phi_{j}^{(1)}(x_n)}=\frac{(\phi_j^2)''(0^-)}{(\phi_j^1)''(0^-)}<p \mbox{ if } j\in\{m_0+1,..., m\}
\]
due to $v_j''(0^-)>0$ and  $(\phi_j^{(k)})'(0^-)=0<(\phi_j^{(k)})''(0^-)$ for $k=1,2$. Thus we always arrive at a contradiction to \eqref{=p} in Case 1.

In Case 2, from the assumptions $\mathbf{(f_1)}$ and $\mathbf{(f_2)}$, we see that
$
W(x):=\Phi^{(1)}(x)-p^{-1}\Phi^{(2)}(x)
$ satisfies, for $x\leq 0$,
\begin{align*}
\mathbf{0}=&D\circ\int_{-\yy}^{0}\mathbf{J} (x-y)\circ W(y) {\rm d}y-D\circ W(x)+cW'(x)+F(\Phi^{(1)}(x))-p^{-1}F(\Phi^{(2)}(x))\\
\succeq  &D\circ\int_{-\yy}^{0}\mathbf{J} (x-y)\circ W(y) {\rm d}y-D\circ W(x)+cW'(x)+F(\Phi^{(1)}(x))-F(p^{-1}\Phi^{(2)}(x))\\
= &D\circ\int_{-\yy}^{0}\mathbf{J} (x-y)\circ W(y) {\rm d}y-D\circ W(x)+cW'(x)+ W(x)E(x),
\end{align*}
where $E=(e_{ij})$ is a matrix function with $e_{ij}\geq 0$ for $i\neq j$. In view of   $W(x)\succeq \mathbf{0}$ for $x\leq 0$, and $W(-\infty)\ggs\mathbf{0}$,  we can apply Lemma \ref{lemma2.7a} to  conclude that
\[
w_i(x)>0 \mbox{ for } x<0,\; i=1,..., m_0.
\]
This is already a contradiction if $i_0\in\{1,..., m_0\}$. If $i_0\in\{m_0+1,..., m\}$, then
\[
cw_{i_0}'(x)=\sum_{j=1}^m e_{ji_0}(x)w_j(x).
\]
By $\mathbf{(f_1)}$ (iv) we see that $e_{ji_0}(x)>0$ for $j\in\{1,..., m_0\}$, and hence
\[
 \sum_{j=1}^m e_{ji_0}(x)w_j(x)>e_{i_0i_0}(x)w_{i_0}(x)\geq -M w_{i_0}(x) \mbox{ for } x<0 \mbox{ and some constant } M.
 \]
 It follows that
 \[
 c w_{i_0}'(x)>-M w_{i_0}(x) \mbox{ for } x<0.
 \]
 This combined with $w_{i_0}(0)=0$ yields $w_{i_0}(x)>0$ for $x<0$, so we  arrive at a contradiction to $w_{i_0}(x_0)=0$.
 
 We have thus proved $p=1$, and so
  $\Phi^{(1)}(x)\succeq \Phi^{(2)}(x)\ {\rm for}\ x\leq 0$. By swapping $\Phi^{(1)}(x)$ with $ \Phi^{(2)}(x)$ we also have $\Phi^{(2)}(x)\succeq \Phi^{(1)}(x)\ {\rm for}\ x\leq 0$. This completes our proof for uniqueness of the semi-wave solution.

\medskip

Next we prove the strict monotonicity properties stated in the theorem. Let $\Phi^c$ be a monotone semi-wave solution of \eqref{2.3} with speed $c>0$. To show the strict monotonicity of $\Phi_c(x)$ with respect to $x\leq 0$, it suffices to verify that $V(x):=\Phi^c(x-{\delta})-\Phi^c(x)\ggs 0$ for $\delta >0$ and $x\leq 0$.   From $\mathbf{(f_1)}$(ii), 
\begin{align*}
F(\Phi^c(x-\delta))-F(\Phi^c(x))=V(x)E(x)
\end{align*}
where  $E(x)=(e_{ij}(x))$ is a matrix function with $e_{ij}\geq 0$ for $i\neq j$. Note that   $V(x)\succeq \mathbf{0}$ by Lemma \ref{lem2.9},  and $V(0)=\Phi^c(-\delta)\ggs \mathbf{0}$ by the conclusions in Claim 2. Making use of Lemma \ref{lemma2.7a} we deduce $v_i(x)>0$ for $x\leq 0$ and $i\in\{1,..., m_0\}$. Using this and $\mathbf{(f_1)}$(iv), we can further show, as in the argument near the end of the proof of Claim 3, that $v_i(x)>0$ for $i\in\{m_0+1,..., m\}$ and $x\leq 0$.

Now we show that for fixed $x<0$,  $\Phi^{c}(x)$ is strictly decreasing with respect to $c>0$, namely, $\Phi^{c_1}(x)=(\phi_i^{c_1}(x))\ggs\Phi^{c_2}(x)=(\phi_i^{c_2}(x))$ for $c_2>c_1>0$. Denote $W(x):=\Phi^{c_1}(x)-\Phi^{c_2}(x)$. By Lemma \ref{lemma2.3} and the proof of Lemma \ref{lem2.9} without shifting $\Phi_{n}^{c}$,  we see that  $W(x)\succeq \mathbf{0}$ for $x\leq 0$. By $\mathbf{(f_1)}$(ii),  
\begin{align*}
F(\Phi^{c_1}(x))-F(\Phi^{c_2}(x))=W(x)E(x)
\end{align*}
where  $E(x)=(e_{ij}(x))$ is a matrix function with $e_{ij}\geq 0$ for $i\neq j$. This, combined with $c_1(\Phi^{c_1})'(x)-c_2(\Phi^{c_2})'(x)\succ c_1 W'(x)$, allows us to apply Lemma \ref{lemma2.7a} to conclude that $w_i(x)>0$ for $x<0$ and $i\in\{1,..., m_0\}$.
We may then use this and $\mathbf{(f_1)}$(iv) to deduce, as before, $w_i(x)>0$ for $x<0$ and $i\in\{m_0+1,..., m\}$.
\end{proof}

\subsection{Semi-wave solution with the desired speed}

\begin{theorem}\label{thm2.12}
Suppose that $\mathbf{(J),\; (f_1)-(f_4)}$ hold, and $\Phi^c(x)$ is the unique monotone semi-wave solution of \eqref{2.3} with speed $c\in (0, C_*)$, where $C_*\in (0, \infty]$ is given by Theorem \ref{lemma2.4}. Then
\begin{equation}\label{c-C*}
\lim_{c\nearrow C_*}\Phi^c(x)=0 \mbox{ locally uniformly in } (-\infty, 0].
\end{equation}
Moreover,  \eqref{2.1a} and \eqref{2.2a} have a solution pair $(c,\Phi)$  with  $\Phi(x)$ monotone
		if and only if $\mathbf{(J_1)}$ holds. 
		And when $\mathbf{(J_1)}$ holds,  there exists a unique $c_0\in (0, C_*)$ such that $(c,\Phi)=(c_0,\Phi^{c_0})$ solves
		 \eqref{2.1a} and \eqref{2.2a}. 
\end{theorem}

\begin{proof}  We first prove \eqref{c-C*}.
Since $\Phi^c=(\phi_i^c)$ is decreasing  with respect to $c$, $\Phi(x):=\lim_{c\nearrow C_*}\Phi^c(x)$ is well-defined, and $\Phi(x)\in[0,\mathbf{u^*}]$ for $x\leq 0$. Moreover, by the uniform boundedness of $(\Phi^c)'(x)$ obtained from the equation it satisfies, the convergence of $\Phi^c(x)$ to $\Phi(x)$ is locally uniform in $(-\infty, 0]$. If $C_*=\infty$, then from
\[
\Phi^c(x)=\frac{1}{c}\int_{0}^{x}\lf[-D\circ\int_{-\yy}^{0}\mathbf{J} (z-y)\circ \Phi^c(y) {\rm d}y+D\circ\Phi^c(z)-F(\Phi^c(z))\rr]\rd z
\]
we immediately obtain $\Phi(x)\equiv \mathbf{0}$. If $C_*<\infty$ then
$\Phi$ satisfies
	\begin{equation*}
	\begin{cases}
	\dd D\circ \int_{-\yy}^{\yy}\mathbf{J} (x-y)\circ \Phi(y) {\rm d}y-D\circ \Phi(x)+C_*\Phi'(x)+F(\Phi(x))={\bf 0},&
	x<  0,\\
	\Phi(0)=\mathbf{0}.
	\end{cases}
	\end{equation*}
	 Note that $\Phi(x)$ is nonincreasing since $\Phi^c(x)$ is.   As in  Step 3 of the proof of Lemma \ref{lemma2.2},  we can show that $\Phi(-\yy)=\mathbf{u}^*$ or $\mathbf{0}$. By Theorem \ref{lemma2.4},  \eqref{2.3} admits no monotone semi-wave solution for $c=C_*$, and hence necessarily $\Phi(-\yy)=\mathbf{0}$. Thus we also have $\Phi\equiv \mathbf{0}$, and \eqref{c-C*} is proved.
	 
	 Next we show that if $\mathbf{(J_1)}$ holds, then \eqref{2.1a}-\eqref{2.2a} have a unique solution pair $(c_0,\Phi^{c_0})$.  It suffices to prove that 
	 \begin{align*}
	 P(c):=c-M(c), \mbox{ with } M(c):= \sum_{i=1}^n\mu_i\int_{-\yy}^{0}\int_{0}^{\yy}J_i(x-y)\phi_i^{c}(x) {\rm d}y{\rm d}x,
	 \end{align*}
	 has a unique root in $(0, C_*)$.
 Let us observe that when $\mathbf{(J_1)}$ holds, $M(c)$ is well-defined and strictly decreasing in $c$ by  Theorem \ref{lemma2.12}. Indeed,
 an elementary calculation yields
 \[
 \int_{-\yy}^{0}\int_{0}^{\yy}J_i(x-y){\rm d}y{\rm d}x=\int_0^\infty\int_0^\infty J_i(x+y)dydx=\int_0^\infty J_i(y)ydy,
 \]
 which implies that $M(c)$ is well-defined.
 
 Using the uniqueness of $\Phi^c$, we can apply a similar convergence argument as used above to prove \eqref{c-C*} to show that $\Phi^{c_n}\to \Phi^{c}$ as $c_n\to c\in(0, C_*)$, which yields the continuity of $\Phi^c(x)$ in $c\in (0,C_*)$ uniformly for $x$ over any bounded interval of $(-\infty, 0]$. Note that we can easily see that $\Phi(x):=\lim_{c_n\to c}\Phi^{c_n}(x)$ satisfies $\Phi(-\infty)=\mathbf{u^*}$ by comparing $\Phi^{c_n}$ to some $\Phi^{\hat c}$ with $\hat c\in (c, C_*)$ and using the monotonicity of $\Phi^c$ in $c$.
 
Hence $P(c)$ is increasing and continuous in $c$. For  $c\in(0,C_*/2)$ close to 0, we  have $P(c)\leq  c-M(C_*/2)<0$, and for all $c$ close to $C_*$, $M(c)$ is small and hence $P(c)>0$. Thus there is a unique $c_0\in (0,C_*)$ such that $P(c_0)=0$.  

Finally we verify that $\mathbf{(J_1)}$ holds if   \eqref{2.1a}-\eqref{2.2a} have a solution pair $(c_0, \Phi^{c_0})$.  Since
\begin{align*}
c_0=\sum_{i=1}^{m_0}\mu_i\int_{-\yy}^{0}\int_{0}^{\yy}J_i(x-y)\phi_i^{c_0}(x) {\rm d}y{\rm d}x,
\end{align*}
for every $i_0\in\{1,..., m_0\}$ such that $\mu_{i_0}>0$ we have 
\begin{align*}
\int_{-\yy}^{0}\int_{0}^{\yy}J_{i_0}(x-y)\phi_{i_0}^{c_0}(x) {\rm d}y{\rm d}x<\yy.
\end{align*}
By Theorem \ref{lemma2.12}, $\Phi^{c_0}(x)$ is decreasing in $x$. Hence, 
\begin{align*}
\int_{-\yy}^{0}\int_{0}^{\yy}J_{i_0}(x-y)\phi_{i_0}^{c_0}(x) {\rm d}y{\rm d}x\geq \phi_{i_0}^{c_0}(-1) \int_{-\yy}^{-1}\int_{0}^{\yy}J_{i_0}(x-y) {\rm d}y{\rm d}x.
\end{align*}
and so
\begin{align*}
\int_{-\yy}^{0}\int_{0}^{\yy}J_{i_0}(x-y) {\rm d}y{\rm d}x=&\int_{-1}^{0}\int_{0}^{\yy}J_{i_0}(x-y) {\rm d}y{\rm d}x+\int_{-\yy}^{-1}\int_{0}^{\yy}J_{i_0}(x-y) {\rm d}y{\rm d}x\\
\leq & 1+\int_{-\yy}^{-1}\int_{0}^{\yy}J_{i_0}(x-y) {\rm d}y{\rm d}x<\yy.
\end{align*}
Therefore,  $\mathbf{(J_1)}$ holds.
\end{proof}

\section{Spreading speed}
\subsection{Further comparison results}
The following maximum principle is a direct consequence of \cite[Lemma 3.1]{dn}, where a more general case  is considered. 
\begin{lemma}\label{lemma3.1} Suppose that $T$ and $d_i$ for $i\in\{1,..., m_0\}$ are positive constants,    $g, h\in C([0,T])$   satisfy  $g (0)<0<h(0)$, and both $-g(t)$ and $h(t)$ are   nondecreasing for $t\in [0,T]$, $\Omega_T:=\{(t,x): t\in (0, T],\; x\in (g(t),h(t))\}$, $m\geq m_0\geq 1$,  and  for $i,j\in \{1,2, ..., m\}$,  $\phi_i$, $\partial_t\phi_i\in C(\ol \Omega_T)$, $c_{ij}\in L^\yy(\Omega_T)$ and
	\begin{equation*}
	\begin{cases}
	\dd \partial_t\phi_i\geq d_i\mathcal{L}_i[\phi_i]+\sum_{j=1}^m c_{ij}\phi_j, & (t,x)\in \Omega_T,\; 1\leq i\leq m_0,\\
	\dd \partial_t\phi_i\geq \sum_{j=1}^m c_{ij}\phi_j, & (t,x)\in \Omega_T,\; m_0+1\leq i\leq m,\\
	\phi_i(t,g(t))\geq 0,\;
	\phi_i(t,h(t))\geq 0,  & t \in (0,T), \ 1\leq i\leq m,\\
	\phi_i(0,x)\geq 0, &x\in [g(0),h(0)],\; \ 1\leq i\leq m,
	\end{cases}
	\end{equation*}
	where
	\begin{align*}
	\mathcal{L}_i[v](t,x):=\int_{g(t)}^{h(t)}J_i(x-y)v(t,y)\rd y-v(t,x),
	\end{align*}
with every $J_i$ $(i=1,..., m_0)$ satisfying {\bf (J)}. Then the following conclusions hold:
	\begin{itemize}
		\item[\rm (i)] If $c_{ij}\geq 0$ on $ \Omega_T$ for $i, j\in \{1,..., m\}$ and $i\not=j$, then $\phi_i\geq 0$ on $\ol \Omega_T$ for $i\in\{1,..., m\}$.
		\item[\rm (ii)] If in addition $\phi_{i_0}(0,x)\not\equiv0$ in $[-h_0,h_0]$ for some $i_0\in\{1,..., m_0\}$, then $\phi_{i_0}> 0$ in $\Omega_T$.
	\end{itemize}
\end{lemma}

\begin{lemma}\label{lemma3.2}    Assume   $(\mathbf{J})$, $(\mathbf{f}_1)$  hold,  and $(U, g,h)$ with  $U=(u_i)$ is the solution of \eqref{1.1}  with initial function $U(0,x)$ satisfying \eqref{1.2} and $U(t,x)\in[\mathbf{0}, \mathbf{\hat u}]$ for $t\geq 0$, $x\in[g(t), h(t)]$.   Suppose $\td g,\, \td h\in C([0,T])$,  $\tilde U$ is continuous in $\ol \Omega_T$, and
\begin{align*}
	  \mathcal{L}[\wtd U](t,x):=\int_{\td g(t)}^{\td h(t)}\mathbf{J}(x-y)\circ \wtd U(t,y)\rd y-\wtd U(t,x),
	\end{align*}
where   $T\in (0,\yy)$ and
\[
\Omega_T:=\{(t,x): t\in (0,T], x\in(\td g(t),\td h(t))\}.
\] 	
 Then the following conclusions hold.
\begin{itemize}
	\item[{\rm (i)}]    If $\td g$, $\td h$ and  the continuous vector function $\wtd U=(\td u_i)$  satisfy
	\begin{align*}
	&\td g(t)\equiv  g(t),\ \ \td h(0)\geq h(0),&&\ t\in [0,T],\\
		&\wtd U(0,x)\succeq \mathbf{0},&&x\in [\td g(0),\td h(0)],\\
		&\wtd U(0,x)\succeq U(0,x), &&x\in [-h_0,h_0],
	\end{align*}
	and
	\begin{equation*}
	\begin{cases}
	\prt \wtd U\succeq   D\circ   {\mathcal{L}}[\wtd U](t,x)+F(\wtd U),\ \ \wtd U(t,x) \in [\mathbf{0},\mathbf{\hat u}], &t\in (0,T], x\in (g(t),\td h(t)),\\[2mm]
	\wtd U(t, g(t))\succeq \mathbf{0},\   \wtd U(t,\td h(t))\succeq \mathbf{0}, & t\in (0,T],\\[2mm]
	\dd \td h'(t)\geq  \sum_{i=1}^{m_0}\mu_i \int_{ g(t)}^{\td h(t)}\int_{\td h(t)}^{\yy}J_i(x-y)\td u_i(t,x)\rd y\rd x, &t\in (0,T],
	\end{cases}
	\end{equation*}
	then  
	\begin{align*}
	 \td h(t)\geq h(t),\ \ \wtd U(t,x)\succeq U(t,x) \ \ \ \ \ {\rm for}\  t\in (0, T],\; x\in [g(t), \td h(t)].
	\end{align*}

\item[{\rm (ii)}]  If $\td g$, $\td h$ and  the continuous vector function $\wtd U=(\td u_i) $  satisfy
		\begin{equation*}
	\begin{cases}
	\prt \wtd U\preceq   D\circ   {\mathcal{L}}[\wtd U](t,x)+F(\wtd U_+), \ \ \wtd U(t,x)\preceq  \mathbf{\hat u}, &t\in (0,T],\; x\in (\td g(t),\td h(t)),\\[2mm]
	\wtd U(t, \td g(t))\preceq \mathbf{0},\   \wtd U(t,\td h(t))\preceq \mathbf{0}, &t\in (0,T],\\[2mm]
	\dd \td h'(t)\leq  \sum_{i=1}^{m_0}\mu_i \int_{\td g(t)}^{\td h(t)}\int_{\td h(t)}^{\yy}J_i(x-y)\td u_i(t,x)\rd y\rd x, &t\in (0,T],\\
	\dd \td g'(t)\geq  -\sum_{i=1}^{m_0}\mu_i \int_{\td g(t)}^{\td h(t)}\int_{-\yy}^{g(t)}J_i(x-y)\td u_i(t,x)\rd y\rd x, &t\in (0,T],
	\end{cases}
	\end{equation*}
and
\begin{align*}
&g(0) \leq \td g(0)<0<\td h(0)\leq h(0),\\
&\wtd U(0,x)\preceq  U(0,x), \ \ \ \ x\in [\td g_0,\td h_0],
\end{align*}
then  
\begin{align*}
& g(t)\leq  \td g(t),\ \td h(t)\leq h(t), && t\in (0, T],\\
& \wtd U(t,x)\preceq U(t,x), && t\in (0, T],\; x\in [\td g(t), \td h(t)].\end{align*}
\end{itemize}
\end{lemma}
Here and in what follows, we use the notation $u_+:=(\max\{u_i,0\})$ for $u\in\R^m$.
\begin{proof}
Lemma \ref{lemma3.2} follows  from Lemma \ref{lemma3.1} by an argument similar to the proof of \cite[Theorem 3.1]{cdJFA}; since the required changes are rather obvious, we omit the details.
\end{proof}

\subsection{Finite spreading speed: Proof of Theorem  \ref{theorem1.3} (i)}

\begin{lemma}\label{lemma3.3}
Let $(\mathbf{J})$, $(\mathbf{J}_1)$ and $(\mathbf{f}_i)$  with $i=1,2, 3, 4$ be satisfied, $(U,g,h)$ be a solution of \eqref{1.1} with $U(0,x)\in [\mathbf{0},\mathbf{\hat u}]$ for $x\in [-h_0,h_0]$.  If spreading happens, then 
\begin{align}\label{3.1}
\limsup_{t\to\yy}\frac{h(t)}{t}\leq  c_0,
\end{align}
	where $c_0>0$ is given by Theorem \ref{prop2.3}. 
\end{lemma}
\begin{proof}
Let $(c_0,\Phi^{c_0})$  be the unique solution pair of  \eqref{2.1a}-\eqref{2.2a}. To simplify notations, we write $\Phi=\Phi^{c_0}=(\phi_i)$. For fixed small $\epsilon>0$, we define $\delta:=2\epsilon c_0$ and
	\[
	\begin{cases}
\bar h(t):=(c_0+\delta) t+K,  &\ \ \ t\geq 0,\\
\ol U(t,x):=(1+\epsilon) \Phi(x-\underline h(t)), &\ \ \ t\geq 0,\ \  x\leq \underline h(t),
\end{cases}
\]
where $K>0$ is a large constant to be determined. 

By $\mathbf{(f_4)}$, the unique solution $v(t)$ of the ODE system 
\begin{align*}
v'=F(v),    \ \  v(0)=(\max_{x} u_i(0,x))\in[\mathbf{0},\mathbf{\hat u}],
\end{align*} 
which also solves \eqref{2.3},
satisfies $v(t)\to \mathbf{u^*}$ as $t\to\infty$.
By Lemma \ref{lemma3.1} and $\mathbf{(f_1)}$, we 
easily see that $v(t)\succeq U(t,x)$ for $ t\geq 0, x\in [g(t),h(t)]$, and hence there is a large constant $t_0>0$ such that 
\begin{align*}
U(t+t_0,x)\preceq (1+\epsilon/2)\mathbf{u}^* \ \mbox{ for } \  t\geq 0,\ x\in [g(t+t_0), h(t+t_0)].
\end{align*}
Due to $\Phi(-\yy)=\mathbf{u}^*$, we may choose sufficient large $K>0$ such that $\underline h(0)=K>h(t_0)$, $-\underline h(0)=-K<g(t_0)$, and also
\begin{align}\label{3.2}
\ol U(0,x)=(1+\epsilon) \Phi(-K)\ggs (1+\epsilon/2)\mathbf{u}^* \succeq  U(t_0,x)\ \mbox{ for } \ x\in [g(t_0),h(t_0)].
\end{align}

Next we verify that, with $\ol U=(\bar u_i)$,
\begin{align}\label{3.3}
\bar h'(t)\geq  \sum_{i=1}^{m_0}\mu_i \int_{-\bar h(t)}^{\bar h(t)} \int_{\bar h(t)}^{\yy} J_i(x-y) \bar u_i(t,x)\rd y\rd x\ \mbox{ for } \ \ t> 0,
\end{align}
  and for $t>0$, $x\in (g(t+t_0),\underline h(t))$, 
  \begin{align}\label{3.4}
  \ol U_t(t,x)\succeq &D\circ \int_{g(t+t_0)}^{\bar h(t)}  \mathbf{(J)}(x-y)\circ \ol U(t,y)\rd y -D\circ\ol U(t,x)+F(\ol U(t,x)).
  \end{align}

A direct  calculation gives, for $t>0$,
\begin{align*}
&\sum_{i=1}^{m_0}\mu_i \int_{g(t+t_0)}^{\bar h(t)} \int_{\bar h(t)}^{\yy} J_i(x-y) \bar u_i(t,x)\rd y\rd x
\leq  \sum_{i=1}^{m_0}\mu_i \int_{-\yy}^{\bar h(t)} \int_{\bar h(t)}^{\yy} J_i(x-y) \bar u_i(t,x)\rd y\rd x\\
= & \sum_{i=1}^{m_0}\mu_i (1+\epsilon) \int_{-\yy}^{0} \int_{0}^{\yy} J_i(x-y) \bar \phi_i(x)\rd y\rd x
=  (1+\epsilon) c_0<c_0+\delta=\bar h'(t),
\end{align*}
which yields \eqref{3.3}. 

Using \eqref{2.1a} and $\mathbf{(f_2)}$, we deduce
 \begin{align*}
 \ol U_t(t,x)=&-(1+\epsilon)(c_0+\delta)\Phi'(x-\bar h(t))\succeq -(1+\epsilon)c_0\Phi'(x-\bar h(t))\\ 
 = &\ (1+\epsilon) \lf[D\circ \int_{- \yy}^{\bar h(t)}  \mathbf{J}(x-y)\circ \ol \Phi(y-\bar h(t))\rd y -D\circ\Phi(x-\bar h(t))+F(\Phi(x-\bar h(t)))\rr]\\ 
 = &\ D\circ \int_{-\yy}^{\bar h(t)}  \mathbf{J}(x-y)\circ \ol U(t,y)\rd y -D\circ \ol U(t,x)+(1+\epsilon)F(\Phi(x-\bar h(t)))\\
 \succeq   &\ D\circ \int_{g(t_0+t)}^{\bar h(t)}  \mathbf{J}(x-y)\circ \ol U(t,y)\rd y -D\circ\ol U(t,x)+F(\ol U(t,x)),
 \end{align*}
which proves \eqref{3.4}.

Moreover, we have
\begin{align*}
\ol U(t,g(t+t_0))>0,\ \  \ol U(t,\bar h(t))=(1+\epsilon) \Phi(\bar h(t)-\bar h(t))= 0 \ \mbox{ for } \ t\geq 0.
\end{align*}
 This combined with \eqref{3.2}, \eqref{3.3} and \eqref{3.4} allows us to use Lemma \ref{lemma3.2} (i) 
  to conclude that
\begin{align*}
& h(t+t_0)\leq  \bar h(t), \ &&t\geq 0,\\
&U(t+t_0,x)\preceq \ol U(t,x),&&t\geq 0,\ x\in [ g(t+t_0),\underline h(t)].
\end{align*}
Hence 
\begin{align*}
\limsup_{t\to\yy} \frac{h(t)}{t}\leq    \limsup_{t\to\yy} \frac{\bar h(t-t_0)}{t}=c_0+\delta,
\end{align*}
which yields \eqref{3.1}  by letting $\delta\to 0$.
\end{proof}

\begin{lemma}\label{lemma3.4} Under the conditions of Lemma \ref{lemma3.3}, we have
	\begin{align}\label{3.5}
	\liminf_{t\to\yy}\frac{h(t)}{t}\geq c_0.
	\end{align}
\end{lemma}
\begin{proof} With $\Phi=(\phi_i)$ as in the proof of Lemma \ref{lemma3.3},  we define, for small $\epsilon>0$,  $\delta:=2c_0\epsilon$ and 
	\[
	\begin{cases}
	\underline h(t):=(c_0-\delta) t+K,  \ \ \ t\geq 0,\\
	\underline U(t,x):=(1-\epsilon) \big[\Phi(x-\underline h(t))+\Phi(-x-\underline h(t))-\mathbf{u}^*\big], \ \ \ t\geq 0, \ x\in [-\underline h(t),\underline h(t)],
	\end{cases}
	\]
	for some large  $K>0$ to be determined.  
		
For later analysis involving $\underline h$ and $\underline U$, we first prepare some simple estimates. Due to $\Phi(-\yy)=\mathbf{u}^*$  there is $K_0>0$ such that 
	\begin{align*}
	\Phi(-K_0)\ggs (1-\epsilon)\mathbf{u}^*,
	\end{align*}
which implies 
	\begin{align}\label{3.6}
	 \Phi(x-\underline h(t)), \Phi(-x-\underline h(t))\in \big[(1-\epsilon)\mathbf{u}^*, \mathbf{u}^*\big] \ \mbox{ for } \ x\in [-\underline h(t)+K_0,\underline h(t)-K_0],\ t\geq 0
	\end{align}
	provided htat $\underline h(0)=K>K_0$.
	Define
	\begin{align*}
	\epsilon_1:=\min_{1\leq i\leq m}\inf_{x\in [-K_0,0]}|\phi_i'(x)|>0.
	\end{align*}
	Then for every $i\in\{1,..., m\}$,
	\begin{equation}\label{3.7}
	\begin{cases}
	\phi_i'(x-\underline h(t))\leq -\epsilon_1  &\mbox{ for } x\in  [\underline h(t)-K_0,\underline h(t)],\ t\geq 0,\\
	\phi_i'(-x-\underline h(t))\leq -\epsilon_1  &\mbox{ for } x\in  [-\underline h(t),-\underline h(t)+K_0],\ t\geq 0.
	\end{cases}
	\end{equation}
	Define 
		\begin{align*}
	\epsilon_2:=\frac{\epsilon_1\delta}{2mL}, \ \mbox{ with  $L=L(\mathbf{u^*})$ is given by \eqref{1.3a}.}
	\end{align*}

{\bf Step 1.} We prove,  with  $\underline U=(\underline u_i)$,
\begin{align}\label{3.8}
&\underline h'(t)\leq \sum_{i=1}^{m_0}\mu_i \int_{-\underline h(t)}^{\underline h(t)} \int_{\underline h(t)}^{\yy} J_i(x-y) \underline u_i(t,x)\rd y\rd x \ \ \mbox{ for }\ \ t> 0,
\end{align}
and 
\begin{align*}
&-\underline h'(t)\geq - \sum_{i=1}^{m_0}\mu_i \int_{-\underline h(t)}^{\underline h(t)} \int_{-\yy}^{-\underline h(t)} J_i(x-y) \underline u_i(t,x)\rd y\rd x  \ \mbox{ for } \ t> 0.
\end{align*}
  Since $\underline U(t,x)=\underline U(t,-x)$ and $\mathbf{J}(x)=\mathbf{J}(-x)$, we just need to verify \eqref{3.8}.

A simple calculation yields  
\begin{align*}
&\sum_{i=1}^{m_0}\mu_i \int_{-\underline h(t)}^{\underline h(t)} \int_{\underline h(t)}^{\yy} J_i(x-y) \underline u_i(t,x)\rd y\rd x\\
=&(1-\epsilon)\sum_{i=1}^{m_0}\mu_i \int_{-2\underline h(t)}^{0} \int_{0}^{\yy} J_i(x-y) \phi_i(x)\rd y\\
&+(1-\epsilon)\sum_{i=1}^{m_0}\mu_i \int_{-2\underline h(t)}^{0} \int_{0}^{\yy} J_i(x-y)[\phi_i(-x-2\underline h(t))-u_i^*]\rd y\rd x\\
=&(1-\epsilon)c_0-(1-\epsilon)\sum_{i=1}^{m_0}\mu_i \int_{-\yy}^{-2\underline h(t)} \int_{0}^{\yy} J_i(x-y) \phi_i(x)\rd y\rd x\\
&-(1-\epsilon)\sum_{i=1}^{m_0}\mu_i \int_{-2\underline h(t)}^{0} \int_{0}^{\yy} J_i(x-y)[u_i^*-\phi_i(-x-2\underline h(t))]\rd y\rd x.
\end{align*}
In view of $\mathbf{(J_1)}$, there is a large  constant $K_1>0$ such that  for $K>K_1$, 
\begin{align*}
&(1-\epsilon)\sum_{i=1}^{m_0}\mu_i \int_{-\yy}^{-2\underline h(t)} \int_{0}^{\yy} J_i(x-y) \phi_i(x)\rd y\rd x\\
\leq &\ (1-\epsilon)\sum_{i=1}^{m_0}\mu_i u_i^*\int_{-\yy}^{-2K} \int_{0}^{\yy} J_i(x-y) \rd y\rd x<\frac{c_0\epsilon}{4}. 
\end{align*}
Moreover, 
\begin{align*}
0\leq &\ (1-\epsilon)\sum_{i=1}^{m_0}\mu_i \int_{-2\underline h(t)}^{0} \int_{0}^{\yy} J_i(x-y)[u_i^*-\phi_i(-x-2\underline h(t))]\rd y\rd x\\
\leq  &\ (1-\epsilon)\sum_{i=1}^{m_0}\mu_i \int_{-\yy}^{0} \int_{0}^{\yy} J_i(x-y)[u_i^*-\phi_i(-x-2\underline h(t))]\rd y\rd x\\
\leq &\ (1-\epsilon)\sum_{i=1}^{m_0}\mu_i u_i^*\int_{-\yy}^{-2K_1} \int_{0}^{\yy} J_i(x-y)\rd y\rd x\\
&+(1-\epsilon)\sum_{i=1}^{m_0}\mu_i\int_{-2K_1}^{0} \int_{0}^{\yy} J_i(x-y)[u_i^*-\phi_i(-x-2\underline h(t))]\rd y\rd x\\
<&\ \frac{c_0\epsilon}{4}+(1-\epsilon)\sum_{i=1}^{m_0}\mu_i[u_i^*-\phi_i(2K_1-2K)]\int_{-2K_1}^{0} \int_{0}^{\yy} J_i(x-y)\rd y\rd x\\
<&\ \frac{c_0\epsilon}{2} \ \mbox{ for $ t\geq 0$ and $K>K_2\gg K_1$. }
\end{align*}
Hence, for $K>K_2$ and all $t\geq 0$,
\begin{align*}
\sum_{i=1}^{m_0}\mu_i \int_{-\underline h(t)}^{\underline h(t)} \int_{\underline h(t)}^{\yy} J_i(x-y) \underline u_i(t,x)\rd y\geq (1-\epsilon)c_0-\frac{3c_0\epsilon}{4}\geq c_0-2c_0 \epsilon> \underline h'(t),
\end{align*}
which finishes the proof of \eqref{3.8}.

{\bf Step 2.} We show that for $  t>0$ and $ x\in (-\underline h(t),\underline h(t)),$
\begin{align}\label{3.9}
\underline U_t(t,x)\preceq &D\circ \int_{-\underline h(t)}^{\underline h(t)}  \mathbf{J}(x-y)\circ \underline U(t,y)\rd y -D\circ\underline U(t,x)+F(\underline U(t,x)_+).
\end{align}

By \eqref{2.1a}, for $t>0$ and  $x\in (-\underline h(t),\underline h(t))$, 
\begin{align*}
&-c_0\Phi'(x-\underline h(t))=D\circ\int_{-\yy}^{\underline h(t)}\mathbf{J} (x-y)\circ \Phi(y-\underline h(t)) {\rm d}y-D\circ\Phi(x-\underline h(t))+F(\Phi(x-\underline h(t)))
\end{align*}
and
\begin{align*}
-c_0\Phi'(-x-\underline h(t))= D\circ\!\!\int_{-\underline h(t)}^{\yy}\mathbf{J} (x-y)\circ \Phi(-y-\underline h(t)) {\rm d}y-D\circ\Phi(-x-\underline h(t))+F(\Phi(-x-\underline h(t))).
\end{align*}
Thus, for such $t$ and $x$,
\begin{align*}
\underline U_t(t,x)=&-(1-\epsilon)(c_0-\delta)[\Phi'(x-\underline h(t))+\Phi'(-x-\underline h(t))]\\
=&\ (1-\epsilon)\delta\,\big[\Phi'(x-\underline h(t))+\Phi'(-x-\underline h(t))\big]\\
&+(1-\epsilon) \bigg[D\circ\int_{-\yy}^{\underline h(t)}\mathbf{J} (x-y)\circ \Phi(y-\underline h(t)) {\rm d}y-D\circ\Phi(x-\underline h(t))\\
&\hspace{2cm}+D\circ\int_{-\underline h(t)}^{\yy}\mathbf{J} (-x-y)\circ \Phi(-y-\underline h(t)) {\rm d}y-D\circ\Phi(-x-\underline h(t))\bigg]\\
&+(1-\epsilon) \bigg[F(\Phi(x-\underline h(t)))+ F(\Phi(-x-\underline h(t))) \bigg]\\
=&\ (1-\epsilon)\delta\,\big[\Phi'(x-\underline h(t))\!+\!\Phi'(-x-\underline h(t))\big]\!+\!D\!\circ\!\! \int_{-\underline h(t)}^{\underline h(t)}  \!\!\mathbf{J}(x-y)\circ \underline U(t,y)\rd y -D\circ\underline U(t,x) \\
&+(1\!-\!\epsilon) \bigg[D\circ\int_{-\yy}^{\underline h(t)}\mathbf{J} (x-y)\circ [\Phi(y-\underline h(t))-\mathbf{u}^*] {\rm d}y\\
&\hspace{2cm}+D\circ\int_{-\underline h(t)}^{\yy}\mathbf{J} (-x-y)\circ [\Phi(-y-\underline h(t)) {\rm d}y-\mathbf{u}^*]\rd y\bigg]\\
&+(1-\epsilon) \bigg[F(\Phi(x-\underline h(t)))+ F(\Phi(-x-\underline h(t))) \bigg]\\
\preceq &\ D\circ \int_{-\underline h(t)}^{\underline h(t)}  \mathbf{J}(x-y)\circ \underline U(t,y)\rd y -D\circ \underline U(t,x)+ F(\underline U(t,x)_+)+\Delta(t,x),
\end{align*}
with
\begin{align*}
\Delta(t,x):=&(1-\epsilon)  \delta\big[\Phi'(x-\underline h(t))+\Phi'(-x-\underline h(t))\big]\\
&+(1-\epsilon)\big[F(\Phi(x-\underline h(t)))+ F(\Phi(-x-\underline h(t)))\big]- F(\underline U(t,x)_+).
\end{align*}

To complete the proof of Step 2, it remains to verify that
\begin{align*}
\Delta(t,x)=(\Delta_i(t,x))\preceq  {\bf 0} \mbox{ for }  t>0,\  x\in (-\underline h(t),\underline h(t)).
\end{align*}
Next we  estimate $\Delta(t,x)$ separately for $x$ in the following   three intervals:  
\[
I_1(t):=[\underline h(t)-K_0,\underline h(t)], \ I_2(t):= [-\underline h(t),-\underline h(t)+K_0],\ I_3(t):=[-\underline h(t)+K_0,\underline h(t)-K_0].
\]
For $x\in I_1(t)$, we have
\begin{align*}
\mathbf{0}\ggs \Phi(-x-\underline h(t))-\mathbf{u}^*\succeq \Phi(K_0-2\underline h(t))-\mathbf{u}^*\succeq  \Phi(K_0-2K)-\mathbf{u}^*\succeq -\epsilon_2\mathbf{1}
\end{align*}
provided  $K>K_2$ for some $K_2\gg K_0$. Then by \eqref{1.3a} and $\mathbf{(f_2)}$,
\begin{align*}
&F(\Phi(-x-\underline h(t)))=F(\Phi(-x-\underline h(t)))-F(\mathbf{u}^*)\preceq mL\epsilon_2\mathbf{1}
\end{align*}
and
\begin{align*}
F(\underline U(t,x)_+)\geq& (1-\epsilon)\Big[F\big([\Phi(x-\underline h(t))+\Phi(-x-\underline h(t))-\mathbf{u}^*]_+\big)\Big]\\
\succeq&(1-\epsilon)\big[F(\Phi(x-\underline h(t)))-mL\epsilon_2\mathbf{1}\big].
\end{align*}
Thus from \eqref{3.7} and the definition of $\epsilon_2$, we deduce
\begin{align*}
\Delta(t,x)\preceq&\ (1-\epsilon)  \big(\delta\big[\Phi'(x-\underline h(t))+\Phi'(-x-\underline h(t))\big]+2mL\epsilon_2\mathbf{1}\big)\\
\preceq &\ (1-\epsilon)\big(-\delta \epsilon_1 +2mL\epsilon_2 \big)\mathbf{1}= \mathbf{0}.
\end{align*}
For $x\in I_2(t)$, since $\Delta(t,x)$ is even in $x$, the above inequality is also valid. 

Finally  we consider  $x\in I_3(t)$. By $\mathbf{(f_3)}$, for each $i\in\{1,...,m\}$, either
\begin{itemize}
\item[{\rm (i)}] $\dd \sum_{j=1}^m \partial_jf_i(\mathbf{u^*})u^*_j<0$, or
\item[{\rm (ii)}] $\dd \sum_{j=1}^m \partial_jf_i(\mathbf{u^*})u^*_j=0$ and $f_i(u)$ is linear in $[\mathbf{u^*}-\epsilon_0\mathbf{1},\mathbf{u^*}]$
for some small $\epsilon_0>0$.
\end{itemize}
Note that since $f_i(\mathbf{u^*})=0$, in case (ii) we have, for $u\in[\mathbf{u^*}-\epsilon_0\mathbf{1},\mathbf{u^*}]$, 
\[
f_i(u)=\sum_{j=1}^m a_j (u_j-u_j^*) \mbox{ for some constants $a_j$, $j=1,..., m$}.
\]
For  $x\in I_3(t)$, \eqref{3.6} holds and hence
 \begin{align*}
-(3-2\epsilon)\epsilon \mathbf{u}^* \preceq    \underline U(t,x)-\mathbf{u}^*\preceq -\epsilon\mathbf{u}^*.
\end{align*}
Therefore, by choosing $\epsilon>0$ sufficiently small, we may assume that for $t>0$ and $x\in I_3(t)$, 
\[
 \Phi(x-\underline h(t)), \Phi(-x-\underline h(t)), (1-\epsilon)^{-1}\underline U(t,x)\in [\mathbf{u^*}-\epsilon_0\mathbf{1},\mathbf{u^*}].
\]
Thus in case (ii) we have, in view of $\mathbf{(f_2)}$, for such $t>0$ and $x$,
\begin{align*}
\Delta_i(t,x)\leq &\ (1-\epsilon)\big[f_i(\Phi(x-\underline h(t)))+ f_i(\Phi(-x-\underline h(t)))\big]- f_i(\underline U(t,x))\\
\leq &\ (1-\epsilon)\Big[f_i\big(\Phi(x-\underline h(t))\big)+ f_i\big(\Phi(-x-\underline h(t))\big)- f_i((1-\epsilon)^{-1}\underline U(t,x))\Big]\\
 =&\ (1-\epsilon)\sum_{j=1}^m a_j \big[\phi_j(x-\underline h(t))-u^*_j +\phi_j(-x-\underline h(t))-u^*_j-(1-\epsilon)^{-1}\underline u_j(t,x)+u^*_j\big]\\
 =&\ 0.
\end{align*}
In case (i), we have, for such $t$ and $x$,
\begin{align*}
\Delta_i(t,x) \leq &\ (1-\epsilon)\big[f_i(\Phi(x-\underline h(t)))+ f_i(\Phi(-x-\underline h(t)))\big]- f_i(\underline U(t,x))\\
=&\ (1-\epsilon)\sum_{j=1}^m \partial_jf_i(\mathbf{u}^*) \big[\phi_j(x-\underline h(t))-\mathbf{u}_j^*\big]+(1-\epsilon)\sum_{j=1}^m \partial_jf_i(\mathbf{u}^*) \big[\phi_j(-x-\underline h(t))-{u}_j^*\big]\\
&\ -\sum_{j=1}^m \partial_jf_i(\mathbf{u}^*) \big[\underline u_j(x)-{u}_j^*\big]+o(\epsilon)\\
=&\ \epsilon\sum_{j=1}^m \partial_jf_i(\mathbf{u}^*){u}_j^*+o(\epsilon)<0,
\end{align*}
provided that $\epsilon>0$ is small enough.
The proof of Step 2 is finished.

{\bf Step 3.} We prove \eqref{3.5} by using Lemma \ref{lemma3.2}.

It is clear that 
\begin{align*}
\underline U(t,\pm \underline h(t))=(1-\epsilon) [\Phi(-2\underline h(t))-\mathbf{u}^*]\preceq \mathbf{0}\ \mbox{ for } \ \ t\geq 0. 
\end{align*}
Since spreading happens for $(U,g,h)$, for fixed $K>\max\{K_0,K_2,K_2\}$ there exists a large constant $t_0>0$ such that 
\begin{align*}
&g(t_0)<-K=-\underline h(0)\ {\rm and}\ \underline h(0)=K< h(t_0),\\
&U(t_0,x)\succeq 1-\epsilon\succeq \underline U(0,x) \ \mbox{ for } \ \ x\in [-\underline h(0),\underline h(0)],
\end{align*}
which combined with the conclusions proved in Steps 1 and 2 allows us to apply Lemma \ref{lemma3.2}  to conclude that 
\begin{align*}
&g(t+t_0)\leq -\underline h(t), \ h(t+t_0)\geq \underline h(t), \ &&t\geq 0,\\
&U(t+t_0,x)\succeq \underline U(t,x),&&t\geq 0,\ x\in [-\underline h(t),\underline h(t)].
\end{align*}
Therefore 
\begin{align*}
\liminf_{t\to\yy} \frac{h(t)}{t}\geq   \liminf_{t\to\yy} \frac{\underline h(t-t_0)}{t}=c_0-\delta,
\end{align*}
and \eqref{3.5} follows by letting $\delta\to 0$.
\end{proof}

\noindent
\underline{Proof of Theorem \ref{theorem1.3} (i):} By the above lemmas, we have
\[
\lim_{t\to\infty}\frac{h(t)}t=c_0.
\]
It remains to show
\[
\lim_{t\to\infty}\frac{g(t)}t=-c_0.
\]
Let $\wtd U(t,x):=U(t,-x)$, $\td h(t):=-g(t)$ and $\td g(t):=-h(t)$.
Then $(\wtd U,\td g,\td h)$ satisfies \eqref{1.1} with  initial function $\td U(0,x):=U(0,-x)$, and spreading happens. 
Hence we can apply Lemmas \ref{lemma3.3} and  \ref{lemma3.4} to conclude that 
\[
\lim_{t\to\infty}\frac{-g(t)}t=\lim_{t\to\infty}\frac{\td h(t)}t=c_0.
\]
\hfill $\Box$

\subsection{Accelerated spreading: Proof of Theorem  \ref{theorem1.3} (ii)}
Throughout this subsection, we assume that the conditions in Theorem  \ref{theorem1.3} (ii) are satisfied.

For each positive integer $n$ and $i\in\{1,..., m_0\}$, we define 
\[
J_{i}^n(x):=\begin{cases} J_{i}(x)& \mbox{ for } |x|\leq n,\\
\frac{n}{|x|}J_{i}(x)& \mbox{ for } |x|\geq n.
\end{cases}
\]
Then clearly $J_{i}^n(x)\leq J_{i}(x)$,  $|x|J_i^n(x)\leq nJ_i(x)$, and for each $\alpha>0$, there exists $c_{n,\alpha}>0$ depending on $n$ and $\alpha$ such that 
\[
e^{\alpha |x|}J_i^n(x)\geq c_{n,\alpha}e^{\frac\alpha 2 |x|}J_i(x).
\] 
Therefore
\begin{equation}\label{Jn}
\int_0^\infty xJ_{i}^n(x)dx\leq \frac n2,\ \ \int_0^\infty e^{\alpha x}J_{i}^n(x)dx\geq c_{n,\alpha}\int_0^\infty e^{\frac\alpha 2 x}J_i(x)dx.
\end{equation}

Next we define $ \mathbf{\widetilde J}^n=(\wtd J_i^n)$ with
\[
\mbox{   $\dd\wtd J_i^n(x):=\frac{J_i^n(x)}{\| J_i^n\|_{L^1(\R)}} $ for $i\in\{1,..., m_0\}$,
\ $\wtd J_i^n(x)\equiv 0$ for $i\in\{m_0+1,..., m\}$.}
\]
 Then by \eqref{Jn} we see that $\mathbf{\widetilde J}^n$ satisfies $\mathbf{(J)}$ and $\mathbf{(J_1)}$, but not $\mathbf{(J_2)}$. Moreover,  $\mathbf{J}^n$ is nondecreasing in $n$, and 
\begin{align}\label{3.10a}
\lim_{n\to\yy} \mathbf{J}^n(x) =\lim_{n\to\yy} \mathbf{\widetilde J}^n(x)=\mathbf{J}(x) \mbox{ locally uniformly for }  x\in\R.
\end{align}

We  now consider  the following variation of \eqref{1.1}:
\begin{equation}\label{3.10}
\begin{cases}
\dd \wtd U_t=D\circ \lf[\int_{ \td g(t)}^{ \td h(t)} \mathbf{J}^n(x-y)\circ  \wtd U(t,y) \rd y - \wtd U(t,x)\rr] +F( \wtd U), & t>0, \; x\in ( \td g(t),  \td h(t)) ,\\
\wtd U(t,\td g(t))=\wtd U(t,\td h(t))=\mathbf{0}, &t>0,\\
\dd \td g'(t)= -\sum_{i=1}^{m_0} \mu_i \int_{\td g(t)}^{\td h(t)}\int_{-\yy}^{\td g(t)}J_i^n(x-y) \td u_i(t,x)\rd y\rd x, & t>0,\\[3mm]
\dd \td h'(t)= \sum_{i=1}^{m_0}\mu_i \int_{\td g(t)}^{\td h(t)}\int_{\td h(t)}^{\yy}J_i^n(x-y) \td u_i(t,x)\rd y\rd x, & t >0.\\
\td g(0)=g(T),\; \td h(0)=h(T),\; \wtd U(0,x)=U(T,x),&x\in [g(T),h(T)],
\end{cases}
\end{equation}
where $T>0$, $\wtd U=(\td u_i)$,    and $(U,g,h)$ with $U=(u_i)$ is the solution of \eqref{1.1} with $U(0,x)\in[\mathbf{0},\mathbf{\tilde u}]$ and spreading happens.

\begin{lemma}\label{lem3.5} For all large $n$,  problem \eqref{3.10} admits a unique positive solution  $(\wtd U,\td g,\td h)=(U^n,g^n,h^n)$ with $U^n=(u_i^n)$, and
	\begin{equation}\label{3.11}
\begin{cases}
[g^n(t),h^n(t)]\subset [g(t+T),h(t+T)], &t\geq 0,\\
U^n(t,x)\preceq U(t+T,x),& t\geq 0,\; x\in [g^n(t),h^n(t)].
\end{cases}
	\end{equation}
\end{lemma}
\begin{proof}
The first equation in \eqref{3.10} can be rewritten as 
\begin{align*}
\dd \wtd U_t=D^n\circ \lf[\int_{ \td g(t)}^{ \tilde h(t)} \mathbf{\wtd J}^n(x-y)\circ  \wtd U(t, y) \rd y - \wtd U(t, x)\rr] + F^n( \wtd U),
\end{align*}
where $D^n=(d_i^n)$ with $d_i^n:=d_i||J_i^n||_{L^1(\R)}$ and
\begin{align}\label{3.13a}
 F^n(w)=F(w)+(D^n-D) \circ w,\ \ \ w\in \R_+^m.
\end{align}
Since $ F$ satisfies $\mathbf{(f_i)}$ for $i\in \{1,2,3\}$, it is easily seen that for all large $n$, $F^n$ also satisfies  $\mathbf{(f_i)}$ for $i\in \{1,2,3\}$ (subject to some obvious modifications). One modification is to replace $\mathbf{u^*}$ by some $\mathbf{u^*_n}$ with $\mathbf{u^*_n}\to \mathbf{u^*}$ as $n\to \infty$, and $F^n(\mathbf{u^*_n})={\bf 0}$. 
The existence and uniqueness of the solution to \eqref{3.10} now follow by  similar arguments to those in \cite[Theorem 4.1]{dn}. Moreover, due to $\mathbf{J}^n\preceq \mathbf{J}$  we can apply   Lemma \ref{lemma3.2} to obtain \eqref{3.11}. 
\end{proof}

\begin{lemma}\label{lem3.6}  For all large $n$, the system
\begin{equation}\label{3.12}
\begin{cases}
\dd D\circ\int_{-\yy}^{0}\mathbf{J}^n (x-y)\circ \Phi(y) {\rm d}y-D\circ\Phi(x)+c\Phi'(x)+F(\Phi(x))=0,&
-\yy<x<0,\\
\Phi(-\yy)=\mathbf{u^*_n},\ \ \Phi(0)=\mathbf{0},
\end{cases}
\end{equation}
and 
\begin{align}\label{3.13}
c= \sum_{i=1}^{m_0}\mu_i\int_{-\yy}^{0}\int_{0}^{\yy}J_i^n(x-y)\phi_i(x) {\rm d}y{\rm d}x,
\end{align}
has a solution pair $(c,\Phi)=(c^n, \Phi^n)$ with 
$\Phi^n(x)=(\phi_i^n(x))$ monotone in $x$.
\end{lemma}
\begin{proof}
 We already see that for all large $n$, $F^n$ satisfies  $\mathbf{(f_i)}$ for $i\in \{1,2,3\}$ (subject to some obvious modifications). 
 In the first equation of \eqref{3.12}, we can replace $(D, \mathbf{J}^n, F)$ by $(D^n, \mathbf{\wtd J}^n, F^n)$ as in the proof of Lemma \ref{lem3.5}. Moreover, since $\mathbf{\wtd J}^n$ does not satisfy $\mathbf{(J_2)}$, we can argue as in the proof of Theorem \ref{TW-J2} to show that the corresponding traveling wave problem of \eqref{3.12} with any $c>0$ has no monotone solution. 
Therefore we can apply Lemma \ref{lem2.9} to conclude that for any $c>0$, \eqref{3.12} has a monotone solution $\Phi=\Phi_c$.
Furthermore, since $\mathbf{\wtd J}^n$ satisfies $\mathbf{(J_1)}$, by Theorem \ref{lemma2.12} and the proof of Theorem \ref{thm2.12}, we see that there exists a unique $c^n>0$ such that
\eqref{3.13} is satisfied with $c=c^n$ and $(\phi_i)=\Phi^n:=\Phi_{c_n}$.
\end{proof}

\begin{lemma}\label{lem3.7} For all large $n$, we have 
	\begin{align*}
 	\liminf_{t\to\yy} \frac{-g(t)}{t}\geq c^n,\ \liminf_{t\to\yy} \frac{h(t)}{t}\geq c^n,
	\end{align*}
where $c^n$ is given in Lemma \ref{lem3.6}. 
\end{lemma}
\begin{proof}
Similarly to the proof of Lemma \ref{lemma3.4}, we construct a lower solution to \eqref{3.10} by defining, for small $\epsilon>0$,  $\delta:=2c^n\epsilon$ and 
\[
\begin{cases}
\underline h^n(t):=(c^n-\delta) t+K,  \ \ \ t\geq 0,\\
\underline U^n(t,x):=(1-\epsilon) [\Phi^n(x-\underline h^n(t))+\Phi^n(-x-\underline h^n(t))-\mathbf{u_n^*}], \ \ \ t\geq 0,\  x\in [-\underline h^n(t),\underline h^n(t)],
\end{cases}
\]
where $K>0$ is to be determined.  Denote $\underline U^n=(\underline u_i^n)$.  As in Steps 1 and  2 of the proof of Lemma \ref{lemma3.4}, one can choose a large enough constant $K$ such that 
\begin{equation*}
\begin{cases}
\dd\prt \underline U^n\preceq D\int_{-\underline h^n(t)}^{\underline h^n(t)} \mathbf{J}^n(x-y)\circ \underline U^n(y) \rd y -D\circ \underline U^n +F(\underline U^n), & t>0, \; x\in (- \underline h^n(t), \underline h^n(t)) ,\\
-\dd (\underline h^n)'(t)\geq  -\sum_{i=1}^{m_0}\mu_i \int_{-\underline  h^n(t)}^{\underline h^n(t)}\int_{\underline h^n(t)}^{\yy}J_i^n(x-y)\underline u_i^n(t,x)\rd y\rd x, & t>0,\\[3mm]
\dd (\underline h^n)'(t)\leq  \sum_{i=1}^{m_0}\mu_i \int_{-\underline  h^n(t)}^{\underline h^n(t)}\int_{\underline h^n(t)}^{\yy}J_i^n(x-y)\underline u_i^n(t,x)\rd y\rd x, & t >0.\\
\end{cases}
\end{equation*}
Clearly, $\underline U_n(t,\pm \underline h_n(t))=(1-\epsilon) [\Phi_n(-2\underline h^n(t))-\mathbf{u_n}^*]\llp \mathbf{0}$ for $ t\geq 0$,
and
\[
\underline U^n(0,x)=(1-\epsilon)\big[ \Phi^n(x-K)+\Phi^n(-x-K)-\mathbf{u^*_n}\big]\prec (1-\epsilon)\mathbf{u^*_n}\to (1-\epsilon)\mathbf{u^*}
\]
as $n\to\infty$.
 Since spreading happens for $(U,g,h)$,  there exists a large constant $T^n>0$ such that 
\begin{align*}
&[-\underline h^n(0),\underline h^n(0)]=[-K,K]\subset [g(T^n), h(T^n)],\\
&\underline U^n(0,x)\preceq U(T^n,x) \ \mbox{ for } \ x\in [-K, K].
\end{align*}
Taking $T=T^n$ in  \eqref{3.10},  we have $g_n(0)=g(T^n)$, $h_n(0)=h(T^n)$ and 
$U_n(0,x)=U(T^n,x)$, and hence we can apply  Lemma \ref{lemma3.2} to obtain
\begin{align*}
& [-\underline h^n(t),\underline h^n(t)]\subset [g^n(t),h^n(t)], \ \ \ t\geq 0,\\
& \underline U^n(t,x)\preceq U^n(t,x),\ \ \ \ \ \ \ \ \ \ \ \ \ \ \ \ \ t\geq 0,\ x\in [-\underline h^n(t),\underline h^n(t)].
\end{align*}
It follows that
\begin{align*}
\liminf_{t\to\yy} \frac{h^n(t)}{t}\geq   \liminf_{t\to\yy} \frac{\underline h^n(t)}{t}=c^n-\delta,
\end{align*}
Letting $\delta\to 0$ we obtain $\liminf_{t\to\yy} \frac{h^n(t)}{t}\geq c^n$. We may now use \eqref{3.11} to deduce
 \begin{align*}
\liminf_{t\to\yy} \frac{h(t)}{t}\geq \liminf_{t\to\yy} \frac{h^n(t-T)}{t}\geq c^n.
 \end{align*}
Analogously we have $\liminf_{t\to\yy} \frac{-g(t)}{t}\geq c^n$.
\end{proof}

\begin{lemma}\label{lem3.8}  $\lim_{n\to\yy }c^n=\yy$.
\end{lemma}
\begin{proof} Suppose on the contrary that there is a  subsequence of $\{c^n\}_{n=1}^{\yy}$, still denoted by itself, converging to a finite number $\hat c$.  Arguing as in the proof of Lemma \ref{lem2.9}, we see that one of the following two statements must be true:
\begin{itemize}
	\item[{\rm (i)}] Problem \eqref{2.1a} with $c=\hat c$ has a monotone solution $\Phi=(\phi_i)$, and  along some sequence $n_k\to\infty$, 
	\[
	\Phi^{n_k}(x)\to \Phi(x) \ \ {\rm locally\ uniformly\ for }\  x\in (-\yy,0].
	\]
		\item[{\rm (ii)}] Problem \eqref{2.3} has a monotone traveling wave solution with speed $\hat c$.
\end{itemize}
Since $\mathbf{(J_2)}$ does not hold, it follows from Theorem \ref{TW-J2}  that there is no monotone traveling wave solution for \eqref{2.3}. 
Therefore (ii) cannot happen, and  necessarily 
 (i) holds. 
 
 From \eqref{3.13} we obtain
 \[
 c^n\geq \sum_{i=1}^{m_0}\mu_i\int_{-L_2}^{-L_1}\int_{0}^{\yy}J^n_i(x-y)\phi^n_i(x) {\rm d}y{\rm d}x
 \]
 for any $L_2>L_1>0$. Take $n=n_k$ and let $k\to\infty$.
 By the dominated convergence theorem  we obtain
\begin{align*}
\hat c\geq \sum_{i=1}^{m_0} \mu_i\int_{-L_2}^{-L_1}\int_{0}^{\yy}J_i(x-y)\phi_i(x) {\rm d}y{\rm d}x.
\end{align*}
Since $\Phi(-\infty)=\mathbf{u^*}\ggs\mathbf{0}$, we can fix $L_1>0$ large enough such that $\phi_i(x)\geq u_i^*/2$ for $x\leq -L_1$,
and hence
\begin{align*}
\hat c\geq \sum_{i=1}^{m_0} \mu_i\frac{u_i^*}2 \int_{-L_2}^{-L_1}\int_{0}^{\yy}J_i(x-y) {\rm d}y{\rm d}x \mbox{ for all } L_2>L_1.
\end{align*}
This implies that, for any $i\in\{1,..., m_0\}$ with $\mu_i>0$,
\[
\int_{-\infty}^0\int_0^\infty J_i(x-y)dydx=\int_0^\infty x J_i(x)dx<\infty,
\]
which is a contradiction to the assumption that $\mathbf{(J_1)}$ is not satisfied. Therefore, $\lim_{n\to\yy }c^n=\yy$.
\end{proof}

Theorem \ref{theorem1.3} (ii) clearly follows directly from Lemmas \ref{lem3.7} and \ref{lem3.8}.

\section{Sharper estimates for the semi-wave and  spreading speed}

\subsection{Asymptotic behaviour of semi-wave solutions to \eqref{2.3}} The purpose of this subsection is to prove the following three theorems, which imply Theorem \ref{theorem1.6}.

\begin{theorem}\label{thm4.1} Suppose that $F$ satisfies $\mathbf{(f_1)-(f_4)}$ and the kernel functions satisfy $\mathbf{(J)}$ and $\mathbf{(J^\alpha)}$ for some $\alpha>0$. If  $\Phi(x)=(\phi_i(x))$  is a monotone solution of \eqref{2.1a} for some $c>0$, then for every $i\in\{1,..., m\}$,
\[
\int_{-\infty}^{-1}[u_i^*-\phi_i(x)]|x|^{\alpha-1}dx<\infty,
\]
which implies, by the monotonicity of $\Phi(x)$,
 \[
 0<|x|^\alpha [u_i^*-\phi_i(x)]\leq C  \mbox{ for some $C>0$ and all $x<0$,\; $i\in\{1,..., m\}$}.
 \]
\end{theorem}

Under the condition $\mathbf{(J)}$, if the kernel functions satisfy $\mathbf{(J^\alpha)}$ for some $\alpha=\alpha_0\geq 1$, then it is easily seen that
$\mathbf{(J^\alpha)}$ is satisfied for all $\alpha\in [0, \alpha_0]$. Therefore if $\mathbf{(J^\alpha)}$ is satisfied for some but not for all $\alpha\geq 1$, then there exists $\alpha^*\in [1,\infty)$ such that the kernel functions satisfy $\mathbf{(J^\alpha)}$ if and only if $\alpha\in I^{\alpha^*}=[0,\alpha^*)$ or $[0,\alpha^*]$ (depending on whether or not $\mathbf{J^{\alpha^*}}$ is satisfied), namely 

\[\begin{cases}
\dd\ \ \ \ \sum _{i=1}^{m_0}\int_0^\infty x^\alpha J_i(x)dx<\infty & \mbox{ for } \alpha\in I^{\alpha^*},\\
\dd\ \ \ \ \sum_{i=1}^{m_0}\int_0^\infty x^\alpha J_i(x)dx=\infty & \mbox{ for  } \alpha\in (0,\infty)\setminus I^{\alpha^*}.
\end{cases}
\]
Therefore, by Theorem \ref{thm4.1} we have
\bes\label{<*}
\sum_{i=1}^{m}\int_{-\infty}^{-1}[u_i^*-\phi_i(x)]|x|^{\alpha-1}dx<\infty \mbox{ for every } \alpha\in I^{\alpha^*}.
\ees
The next result shows that this estimate is sharp.
\begin{theorem}\label{thm4.2} Suppose that $F$ satisfies $\mathbf{(f_1)-(f_4)}$ and the kernel functions satisfy $\mathbf{(J)}$.
If $\mathbf{(J^\alpha)}$ is not satisfied for  some $\alpha>0$, and  $\Phi(x)=(\phi_i(x))$  is a monotone solution of \eqref{2.1a} for some $c>0$, then 
\bes\label{>*}
\sum_{i=1}^{m}\int_{-\infty}^{-1}[u_i^*-\phi_i(x)]|x|^{\alpha-1}dx=\infty.
\ees
\end{theorem}
\begin{theorem}\label{thm4.3} Suppose that $F$ satisfies $\mathbf{(f_1)-(f_4)}$ and the kernel functions satisfy $\mathbf{(J)}$.
If $\mathbf{(J_2)}$ holds, and  $\Phi(x)=(\phi_i(x))$  is a monotone solution of \eqref{2.1a} for some $c>0$, then there exist positive constants $\beta$ and $C$ such that
\bes\label{exp}
0<u_i^*-\phi_i(x)\leq C e^{\beta x} \mbox{ for all } x<0,\; i\in\{1,..., m\}.
\ees
\end{theorem}
\medskip

The following three lemmas play a crucial role in the proof of Theorem \ref{thm4.1}.
\begin{lemma}\label{lem4.1}
Suppose that $J(x)$ has the properties described in $\mathbf{(J)}$ and satisfies $\mathbf{(J^\alpha)}$ for some $\alpha\geq 1$. If $\psi\in L^1((-\infty, 0])$ is nonnegative, continuous and nondecreasing in $(-\infty, 0]$, and
\begin{equation}\label{4.1}
\int_{-\infty}^0|x|^\beta\psi(x)dx<\infty \mbox{ for some } \beta\geq 0,
\end{equation}
then for any $\sigma\in(0,\min\{\beta+1,\alpha\}]$, there exists $C>0$ such that 
\[
I=I_M:=\int_{-M}^0|x|^{\sigma}\left [\int_{-\infty}^0J(x-y)\psi(y)dy-\psi(x)\right]dx\in[-C, C]
\mbox{ for all } M>0.
\]
\end{lemma}
\begin{proof}
For fixed $M>0$ we have
\begin{align*}
&\int_{-M}^0\int_{-\infty}^0|x|^\sigma J(x-y)\psi(y)dydx\\
&=\int_0^M\int_{-\infty}^x x^\sigma J(y)\psi(y-x)dydx\\
&=\int_0^M\int_{-\infty}^0 x^\sigma J(y)\psi(y-x)dydx+\int_0^M\int_{0}^x x^\sigma J(y)\psi(y-x)dydx\\
&=\int_{-\infty}^0\int_{0}^{M} x^{\sigma}J(y)\psi(y-x)dxdy+\int_0^M\int_{y}^M x^{\sigma}J(y)\psi(y-x)dxdy\\
&=\int_{-\infty}^0\int_{-y}^{M-y} (x+y)^{\sigma}J(y)\psi(-x)dxdy+\int_0^M\int_{0}^{M-y} (x+y)^{\sigma}J(y)\psi(-x)dxdy,
\end{align*}
and
\[
\int_{-M}^0|x|^{\sigma}\psi(x)dx=\int_{\R}\int_0^M x^{\sigma}J(y)\psi(-x)dxdy.
\]
Therefore we can write
\[
I=\sum_{j=1}^3I_j
\]
with
\begin{align*}
I_1:=& \int_{-\infty}^0\int_{-y}^{M-y} \left[(x+y)^{\sigma}-x^{\sigma}\right]J(y)\psi(-x)dxdy\\
&+\int_0^M\int_{0}^{M-y} \left[(x+y)^{\sigma}-x^{\sigma}\right]J(y)\psi(-x)dxdy,
\\
I_2:=& \int_{-\infty}^0\int_{M}^{M-y} x^{\sigma}J(y)\psi(-x)dxdy-\int_{-\infty}^0\int_{0}^{-y} x^{\sigma}J(y)\psi(-x)dxdy,\\
I_3:=& -\int_0^M\int_{M-y}^{M} x^{\sigma}J(y)\psi(-x)dxdy-\int_M^{\infty}\int_{0}^{M} x^{\sigma}J(y)\psi(-x)dxdy.
\end{align*}

To estimate $I_1$ we will make use of some elementary inequalities. If $s,t>0$ and $\sigma\in (0,1]$, then it is easily checked that
\bes\label{sigma<1}
(s+t)^\sigma-s^\sigma\leq t^\sigma.
\ees
If $\sigma=n+\theta$ with $n\geq 1$ an integer, and $\theta\in (0,1]$, then by the mean value theorem
\begin{align*}
(s+t)^\sigma-s^\sigma &=\sigma (s+\zeta t)^{\sigma-1}t\leq \sigma t(s+t)^{\sigma-1}=\sigma t s^{\sigma-1}+\sigma t\left[(s+t)^{\sigma-1}-s^{\sigma-1}\right]\\
&\leq \sum_{k=1}^n\left[\Pi_{j=0}^{k-1} (\sigma-j)t^ks^{\sigma-k}\right]+\Pi_{j=0}^{n-1}(\sigma-j)t^n\left[(s^\theta+t^\theta)-s^\theta\right]\\
&\leq \sum_{k=1}^n\left[\Pi_{j=0}^{k-1} (\sigma-j)t^ks^{\sigma-k}\right]+\Pi_{j=0}^{n-1}(\sigma-j) t^{n+\theta}\\
&=\sum_{k=1}^nc_kt^ks^{\sigma-k} +c_{n+1}t^{\sigma}
\end{align*}
where $\zeta\in [0,1]$, and $c_k=c_k(\sigma)>0$ for $k\in\{1,..., n+1\}$.

Applying this inequality to $(x+y)^{\sigma}-x^{\sigma}$ with $x+y>0$ and $x>0$, we obtain, for the case $\sigma>1$,
\[
|(x+y)^{\sigma}-x^{\sigma}|\leq \sum_{k=1}^nc_k|y|^kx^{\sigma-k} +c_{n+1}|y|^{\sigma}
\]
with $\sigma-n=\theta\in (0, 1]$ and $n\geq 1$ an integer, $c_k=c_k(\sigma)>0$ for $k\in\{1,..., n+1\}$.

Therefore, in the case $\sigma>1$,
\begin{align*}
|I_1|\leq & \ \int_{-\infty}^0\int_{-y}^{M-y} \left[  \sum_{k=1}^{n}c_k|y|^kx^{\sigma-k} +c_{n+1}|y|^{\sigma}  \right]J(y)\psi(-x)dxdy\\
&+\int_0^M\int_{0}^{M-y} \left[\sum_{k=1}^{n}c_k|y|^kx^{\sigma-k} +c_{n+1}|y|^{\sigma}\right]J(y)\psi(-x)dxdy\\
\leq & \ 2\sum_{k=1}^nc_k\int_0^\infty x^{\sigma-k}\psi(-x)dx\int_0^\infty y^kJ(y)dy +2c_{n+1}\int_0^\infty \psi(-x)dx\int_0^\infty y^{\sigma}J(y)dy \\
:=& \ C_1.
\end{align*}
Since $1\leq k\leq n<\sigma\leq \min\{\beta+1, \alpha\}$, by the assumptions on $J$ and $\psi$ we see that $ C_1$  is a finite number.

If $\sigma\in (0, 1]$, then
\begin{align*}
|I_1|\leq & \int_{-\infty}^0\int_{-y}^{M-y} |y|^\sigma J(y)\psi(-x)dxdy+\int_0^M\int_{0}^{M-y} |y|^\sigma J(y)\psi(-x)dxdy\\
\leq & \ 2\int_0^\infty \psi(-x)dx\int_0^\infty y^\sigma J(y)dy:=\tilde C_1<\infty.
\end{align*}

Since $\psi(x)$ is nondecreasing, from \eqref{4.1} we easily deduce
\[
\psi(-x)\leq \frac{M_1}{x^{\sigma}} \mbox{ for some $M_1>0$ and all $x>0$}.
\]
Similarly, using $\mathbf{(J^\alpha)}$ we obtain
\[
M\int_M^\infty J(y)dy\leq M^{1-\alpha}\int_M^\infty y^\alpha J(y)dy\leq \int_1^\infty y^\alpha J(y)dy:= M_2 \mbox{ for } M\geq 1,
\] 
and hence
\[
M\int_M^\infty J(y)dy\leq \min\left\{\int_0^\infty J(y), M_2 \right\}:=M_3<\infty \mbox{ for all } M>0.
\]
Therefore
\begin{align*}
|I_2|&\leq  \int_{-\infty}^0\int_{M}^{M-y} M_1 J(y)dxdy+\int_{-\infty}^0\int_{0}^{-y} M_1J(y)dxdy\\
& = 2M_1\int_0^\infty yJ(y)dy:=C_2<\infty,
\end{align*}
and
\begin{align*}
|I_3|&\leq \int_0^MM_1yJ(y)dy+\int_M^{\infty}M_1MJ(y)dy\\
&\leq M_1\int_0^\infty yJ(y)dy+M_1M_3:=C_3<\infty.
\end{align*}
We thus have
\[
|I|\leq C_1+\tilde C_1+C_2+C_3:=C<\infty \mbox{ for all } M>0.
\]
The proof is complete.
\end{proof}

\begin{lemma}\label{lemma4.4a} Suppose that $J(x)$ has the properties described in $\mathbf{(J)}$ and satisfies $\mathbf{(J^\alpha)}$ for some $\alpha\in  (0,1)$. Let   $\psi$ be nonnegative, continuous and nondecreasing in $(-\infty, 0]$. Then there exists $C>0$ such that
		\[
		S=S_M:=\int_{-M}^0|x|^{\alpha-1}\left [\int_{-\infty}^0J(x-y)\psi(y)dy-\psi(x)\right]dx\leq C
		\mbox{ for all } M>0.
		\]
\end{lemma}
\begin{proof}
	As in the proof of Lemma \ref{lemma4.1},	 we deduce for fixed $M>0$ and $\sigma>-1$,
	\begin{align*}
	&\int_{-M}^0\int_{-\infty}^0|x|^\sigma J(x-y)\psi(y)dydx\\
	=&\int_{-\infty}^0\int_{-y}^{M-y} (x+y)^{\sigma}J(y)\psi(-x)dxdy+\int_0^M\int_{0}^{M-y} (x+y)^{\sigma}J(y)\psi(-x)dxdy.
	\end{align*}
	and 
	\[
	\int_{-M}^0|x|^{\sigma}\psi(x)dx=\int_{\R}\int_0^M |x|^{\sigma}J(y)\psi(-x)dxdy.
	\]
	Hence 
	\begin{align*}
	S=\sum_{i=1}^3\td I_i 
	\end{align*}
	with
	\begin{align*}
	\td I_1:=& \int_{-\infty}^0\int_{-y}^{M-y} \left[(x+y)^{\sigma}-x^{\sigma}\right]J(y)\psi(-x)dxdy\\
	&+\int_0^M\int_{0}^{M-y} \left[(x+y)^{\sigma}-x^{\sigma}\right]J(y)\psi(-x)dxdy,
	\\
	\td I_2:=& \int_{-\infty}^0\int_{M}^{M-y} x^{\sigma}J(y)\psi(-x)dxdy-\int_{-\infty}^0\int_{0}^{-y} x^{\sigma}J(y)\psi(-x)dxdy,\\
	\td I_3:=& -\int_0^M\int_{M-y}^{M} x^{\sigma}J(y)\psi(-x)dxdy-\int_M^{\infty}\int_{0}^{M} x^{\sigma}J(y)\psi(-x)dxdy.
	\end{align*}
	
Take $\sigma=\alpha-1$. It is clear that $\td I_3\leq 0$. For $\td I_1$,  since $\sigma<0$,
	\begin{align*}
	(x+y)^{\sigma}-x^{\sigma}< 0 \ \mbox{ when }  x>0 \ \mbox{ and }  y> 0,
	\end{align*}
	and hence, by  $\mathbf{(J^\alpha)}$ and $\sigma+1=\alpha\in (0,1)$,
	\begin{align*}
	\td I_1\leq& \int_{-\infty}^0\int_{-y}^{M-y} \left[(x+y)^{\sigma}-x^{\sigma}\right]J(y)\psi(-x)dxdy\\
	\leq& \psi(0)\int_{-\infty}^0\int_{-y}^{M-y} \left[(x+y)^{\sigma}-x^{\sigma}\right]J(y)dxdy\\
	=&\frac{\psi(0)}{\sigma+1}\int_{-\yy}^{0}[M^{\sigma+1}-(M-y)^{\sigma+1}+(-y)^{\sigma+1}]J(y)\rd y\\
	\leq &\frac{\psi(0)}{\sigma+1}\int_{-\yy}^{0}(-y)^{\sigma+1}J(y)\rd y=\frac{\psi(0)}{\sigma+1}\int_{0}^{\yy} y^{\sigma+1}J(y)\rd y:=C_1<\yy.
	\end{align*}
	
	Moreover, by $\mathbf{(J^\alpha)}$, $\sigma+1=\alpha\in (0,1)$ and \eqref{sigma<1},
	\begin{align*}
	\td I_2\leq& \int_{-\yy}^{0}\int_{M}^{M-y}x^{\sigma}J(y)\psi(-x)\rd x\rd y\leq \psi(0)\int_{-\yy}^{0}\int_{M}^{M-y}x^{\sigma}J(y)\rd x\rd y\\
	=&\frac{\psi(0)}{\sigma+1}\int_{-\yy}^{0}[(M-y)^{\sigma+1}-M^{\sigma+1}]J(y)\rd y\\
	\leq&\frac{\psi(0)}{\sigma+1}\int_{0}^{\yy}y^{\sigma+1}J(y)\rd y:=C_2<\yy.
	\end{align*}
	Therefore, 
	\[
	S\leq C_1+C_2:=C<\infty \mbox{ for all } M>0.
	\]
	The proof is complete. 
\end{proof}

 Denote 
 \[
 \mbox{$\Psi(x)=(\psi_i(x)):=\mathbf{u}^*-\Phi(x)$ and $G(u)=(g_i(u)):=-F(\mathbf{u}^*-u)$.}
 \]
  Then $\Psi$ satisfies
\begin{equation}\label{3.11a}
\begin{cases}\dd
\mathbf{0}=D\circ\int_{-\yy}^{0}\mathbf{J} (x-y)\circ \Psi(y) {\rm d}y-D\circ\Psi+D\circ \mathbf{u}^*\circ\int_0^\yy\mathbf{J} (x-y) {\rm d}y\\
\ \ \ \ \ +c\Psi'(x)+G(\Psi(x)) \mbox{ for }
-\yy<x<0,\\
\Psi(-\yy)=\mathbf{0},\ \ \Psi(0)=\mathbf{u}^*.
\end{cases}
\end{equation}
Since $\mathbf{u}^*$ is  stable and $\nabla F(\mathbf{u}^*)=\nabla G(\mathbf{0})$ is invertible,  the eigenvalues of $\nabla F(\mathbf{u}^*)$ are all negative. Therefore we can use the same reasoning as in the proof of Lemma  \ref{lemma2.1a} to find two vectors $\wtd A=(\td a_i)\ggs \mathbf{0}$ and $\wtd B=(\td b_i)\llp \mathbf{0}$  such that, for $U=(u_i)\in [\mathbf{0}, \epsilon\mathbf{1}]$ with $\epsilon>0$ sufficiently small, 
\begin{align*}
\sum_{i=1}^{m}\td a_ig_i(U)\leq  \sum_{i=1}^{m} \td b_iu_i\leq -\hat b \sum_{j=1}^m\tilde a_iu_i,
\end{align*}
for some $\hat  b>0$. 

Since $\Psi(-\yy)=\mathbf{0}$ and $\Psi(x)=(\psi_i(x))\ggs \mathbf{0}$ for $x<0$,  we have $0<\psi_i(x)<\epsilon$ for $x\ll -1$, and so 
\begin{align}\label{4.2}
\sum_{i=1}^m \td a_ig_i(\Psi(x))\leq -\hat b \wtd \psi(x)\ \  \mbox{ for $x\ll -1$, with }
\end{align}
\begin{align}\label{td-psi}
\wtd \psi(x):=\sum_{j=1}^m \td a_j\psi_j(x).
\end{align}

\begin{lemma}\label{lemma4.1}
	Suppose  $\mathbf{(J)}$ and $\mathbf{(f_1)-(f_4)}$ are satisfied. If $\mathbf{(J^\alpha)}$ holds for some $\alpha\geq 1$,
	then 
	\begin{align*}
	\int_{-\yy}^{0}	\wtd \psi(x) \rd x<\yy. 
	\end{align*}
\end{lemma}
\begin{proof}
	A simple calculation gives 
	\begin{align*}
	&D\circ\int_{-\yy}^{0}\mathbf{J} (x-y)\circ \Psi(y) {\rm d}y-D\circ\Psi+D\circ \mathbf{u}^*\circ\int_{0}^{\yy}\mathbf{J} (x-y) {\rm d}y\\
	=&-D\circ\int_{-\yy}^{0}\mathbf{J} (x-y)\circ \Phi(y) {\rm d}y+D\circ\Phi.
	\end{align*}
	Integrating the equation satisfied by $\wtd \psi$  over the interval  $(x,y)$ with $x<y\ll -1$, and making use of \eqref{4.2}, we obtain
	\begin{align*}
	&c(\wtd \psi(y)-\wtd \psi(x))
	+\sum_{i=1}^m \int_{x}^{y}\td a_id_i  \lf[\int_{-\yy}^0 J_i(z-w)\psi_i(w){\rm d}w-  \psi_i(z) \rr]{\rm d}z
	\\
	&+\sum_{i=1}^m \int_{x}^{y}\td a_id_i u_i^* \int_{0}^\yy J_i(z-w){\rm d}w{\rm d}z\\
	=&c(\wtd \psi(y)-\wtd \psi(x))-\sum_{i=1}^m \int_{x}^{y}\td a_id_i  \lf[\int_{-\yy}^0 J_i(z-w)\phi_i(w){\rm d}w-  \phi_i(z) \rr]{\rm d}z\\
	=&-\int_{x}^{y}\sum_{i=1}^m \td a_i  g_i(\Psi(z))\rd z\geq \hat b\int_{x}^{y} \wtd \psi(z) \rd z.
	\end{align*} 
	We extend $\Phi$ to $\R$ by define $\phi_i(x)=0$ for $x>0$. Then the new function $\Phi$ is differentiable on $\R$ except at $x=0$.  
	Due to $\mathbf{(J^\alpha)}$,  we have, for $i\in\{1,..., m_0\}$,
	\begin{align*}
	&\  \lf|\int_{x}^{y} \lf(\int_{-\yy}^{0} J_i(z-w)\phi_i(w){\rm d}w-\phi_i(z)\rr){\rm d}z\rr|= \lf|\int_{x}^{y} \lf(\int_{\R} J_i(z-w)\phi_i(w){\rm d}w-\phi_i(z)\rr){\rm d}z\rr|\\
	=&\  \lf|\int_{x}^{y} \int_{\R} J_i(w)(\phi_i(z+w)-\phi(z)) {\rm d}w{\rm d}z\rr|=\lf|\int_{x}^{y} \int_{\R} J_i(w) \int_{0}^1 w\phi'_i(z+sw) {\rm d}s{\rm d}w{\rm d}z\rr|\\
	=&\ \lf|\int_{\R} wJ_i(w) \int_{0}^1 [\phi_i(y+sw)-\phi_i(x+sw)] {\rm d}s{\rm d}w\rr|\\
	\leq & \  a_i^* \int_{\R} |y|J_i(y) {\rm d}y=:M_i<\yy.
	\end{align*}
	Thus, for $x<y\ll -1$,
	\begin{align*}
	\hat b\int_{x}^{y} \wtd \psi(z) \rd z\leq c(\wtd \psi(y)-\wtd \psi(x))+\sum_{i=1}^m \td a_id_iM_i\leq  \sum_{i=1}^m \td a_i(cu_i^*+d_iM_i),
	\end{align*}
	which implies $\dd \int_{-\yy}^0  \wtd\psi(z) {\rm d}z<\yy$. 
	\end{proof}
	
	\noindent
\underline{Proof of Theorem \ref{thm4.1}}:	
{\bf Case 1}. $\alpha\geq 1$. 

With $\wtd \psi=\sum_{i=1}^m\td a_i \psi_i$ given by \eqref{td-psi}, it suffices to show
	\begin{align*}
	 \ \int_{-\yy}^0 \wtd\psi(x) |x|^{\alpha-1}\rd x<\yy.
	\end{align*}
	By Lemma \ref{lemma4.1} we have 
\[
\int_{-\infty}^0\td\psi(x)dx<\infty \mbox{ and hence } \int_{-\infty}^0\psi_i(x)dx<\infty \mbox{ for } i\in\{1,..., m\}.
\]
So there is nothing to prove if $\alpha=1$, and we only need to consider the case $\alpha>1$.

 Suppose $\alpha>1$ and
\begin{equation}\label{gamma}
\int_{-\infty}^0|x|^\gamma \td\psi(x)dx<\infty \mbox{ for some  } \gamma\geq 0.
\end{equation}
Then  by Lemma \ref{lem4.1}, for any $\beta$ satisfying $0< \beta\leq \min\{\gamma+1,\alpha\}$,
 and $i\in\{1,..., m_0\}$,
\begin{equation}\label{beta}
\int_{-M}^{0} \lf[\int_{-\yy}^0 J_i(x-y)\psi_i(y){\rm d}y-  \psi_i(x) \rr]|x|^{\beta}{\rm d}x\leq C \mbox{ for some $C>0$ and all $M<0$}.
\end{equation}
Moreover, if we fix $M_0>1$ so that \eqref{4.2} holds for $x\leq -M_0$, then for $M<M_0$ and $\beta$ as above, we have
\begin{align*}
&\hat b \int_{-M}^{-M_0}\tilde\psi (x)|x|^\beta dx\\
\leq& -\sum_{i=1}^m \int_{-M}^{-M_0}\tilde a_i g_i(\Psi(x))|x|^\beta dx\\
=&\ c\int_{-M}^{-M_0}\tilde\psi'(x)|x|^\beta dx+ \sum_{i=1}^{m_0} \td a_id_i \int_{-M}^{-M_0} \lf[\int_{-\yy}^0 J_i(x-y)\psi_i(y){\rm d}y-  \psi_i(x) \rr]|x|^\beta{\rm d}x
	\\
	&+\sum_{i=1}^{m_0} \td a_id_i u_i^*\int_{-M}^{-M_0} \int_{0}^\yy |x|^\beta J_i(x-y){\rm d}y{\rm d}x.
\end{align*}

By \eqref{beta},
\begin{align*}
&\sum_{i=1}^{m_0} \td a_id_i \int_{-M}^{-M_0} \lf[\int_{-\yy}^0 J_i(x-y)\psi_i(y){\rm d}y-  \psi_i(x) \rr]|x|^{\beta}{\rm d}x\\
&\leq C\sum_{i=1}^{m_0} \td a_id_i-\sum_{i=1}^{m_0} \td a_id_i \int_{-M_0}^{0} \lf[\int_{-\yy}^0 J_i(x-y)\psi_i(y){\rm d}y-  
\psi_i(x) \rr]|x|^{\beta}{\rm d}x\\
&:= C_1<\infty \mbox{ for  all } M>M_0.
\end{align*}

Moreover, if we assume additionally that $\beta\leq \alpha-1$, then we have, for $i\in\{1,..., m_0\}$,
\begin{align*}
&\int_{-M}^{-M_0} \int_{0}^\yy |x|^\beta J_i(x-y){\rm d}y{\rm d}x\\
&\leq \int^{M}_0\int_0^\infty x^\beta J_i(x+y)dydx=\int_0^M\int_x^\infty x^\beta J_i(y)dydx\\
&\leq \int_0^\infty\int_x^\infty x^{\beta} J_i(y)dy= \frac 1{\beta+1}\int_0^\infty y^{\beta+1}J_i(y)dy:=C_2<\infty.
\end{align*}
Therefore, for $\beta\in (0,\min\{\gamma+1,\alpha-1\}]$ and $M>M_0$,
\begin{align*}
&\hat b \int_{-M}^{-M_0}\tilde\psi (x)|x|^\beta dx
\leq\ c\int_{-M}^{-M_0}\tilde\psi'(x)|x|^\beta dx+ C_1+\sum_{i=1}^m \td a_id_i u_i^* C_2\\
&\leq c\int_1^Mx^\beta \tilde\psi'(-x)dx+C_3\leq c\int_1^Mx^{\gamma+1} \tilde\psi'(-x)dx+C_3\\
&\leq c \tilde\psi(-1)+c\int_1^M(\gamma+1) x^{\gamma}\td\psi(-x)dx+C_3
:=C_4<\infty \mbox{ by \eqref{gamma}.} 
\end{align*}
It follows that
\begin{equation}\label{induc}
\int_{-\infty}^0\tilde\psi (x)|x|^\beta dx<\infty.
\end{equation}
Thus we have proved that \eqref{gamma} implies \eqref{induc} for any $\beta\in(0, \min\{\gamma+1,\alpha-1\}]$.

If we write $\alpha-1=n+\theta$ with $n\geq 0$ an integer and $\theta\in (0, 1]$. Then by the above conclusion and an induction argument we see that \eqref{induc} holds with $\beta=n$. Thus \eqref{gamma} holds for $\gamma=n$. So applying the above conclusion once more we see that \eqref{induc} holds for every $\beta\in (0, \min\{n+1,\alpha-1\}]=(0,\alpha-1] $, as desired.

{\bf Case 2}. $\alpha\in  (0,1)$. 

Let $\beta=\alpha-1$. As in Case 1, for $M>M_0$, 
\begin{align*}
&\hat b \int_{-M}^{-M_0}\tilde\psi (x)|x|^\beta dx\\
\leq &\ c\int_{-M}^{-M_0}\tilde\psi'(x)|x|^\beta dx+ \sum_{i=1}^{m_0} \td a_id_i \int_{-M}^{-M_0} \lf[\int_{-\yy}^0 J_i(x-y)\psi_i(y){\rm d}y-  \psi_i(x) \rr]|x|^\beta{\rm d}x
\\
&+\sum_{i=1}^{m_0} \td a_id_i u_i^*\int_{-M}^{-M_0} \int_{0}^\yy |x|^\beta J_i(x-y){\rm d}y{\rm d}x\\
\leq &\ c\int_{-M}^{-M_0}\tilde\psi'(x)|x|^\beta dx+\wtd C_1+\sum_{i=1}^{m_0} \td a_id_i u_i^*\int_{-M}^{-M_0} \int_{0}^\yy |x|^\beta J_i(x-y){\rm d}y{\rm d}x,
\end{align*}
where $\wtd C_1>0$ is obtained by making use of Lemma \ref{lemma4.4a}. By  $\mathbf{(J^\alpha)}$ and $\beta+1=\alpha$,
\begin{align*}
&\int_{-M}^{-M_0} \int_{0}^\yy |x|^\beta J_i(x-y){\rm d}y{\rm d}x\leq \int_{0}^{\yy} \int_{x}^{\yy} x^\beta J_i(y)\rd y\rd x\\
=&\frac{1}{\alpha}\int_{0}^{\yy}y^\alpha J_i(y)\rd y:=\wtd C_2<\yy.
\end{align*}
Due to $\beta<0$, we have
\begin{align*}
&\int_{-M}^{-M_0}\tilde\psi'(x)|x|^\beta dx= \int_{M_0}^{M}\tilde\psi'(-x) x^\beta dx\\
=& \tilde\psi(-M_0)M_0^\beta -\tilde\psi(-M)M^{\beta}+\beta \int_{M_0}^{M}\tilde\psi(-x)x^{\beta-1} dx\\
\leq&\tilde\psi(-M_0)M_0^\beta:=\wtd C_3<\yy.
\end{align*}
Hence
\begin{align*}
\hat b \int_{-M}^{-M_0}\tilde\psi (x)|x|^\beta dx\leq \wtd C_1+\wtd C_2\sum_{i=1}^{m_0} \td a_id_i u_i^*+c\wtd C_3<\yy
\end{align*}
for all $M>M_0$, which implies 
\begin{align*}
\int_{-\yy}^{-1}\tilde\psi (x)|x|^{\alpha-1} dx<\yy.
\end{align*}
The proof is completed.

\hfill $\Box$
\smallskip

\noindent
\underline{Proof of Theorem \ref{thm4.2}}:	 
We have
\[
|g_i(\Psi(x))|\leq L\sum_{j=1}^m\psi_j(x) :=L\hat\psi(x) \mbox{ for some $L>0$ and all } x<0,\; i\in\{1,..., m\}.
\]
Now for $M>1$ and $\beta=\alpha-1$, 
\begin{align*}
&L \int_{-M}^{-1}\hat\psi (x)|x|^\beta dx
\geq -\sum_{i=1}^m \int_{-M}^{-1} g_i(\Psi(x))|x|^\beta dx\\
=&\ c\int_{-M}^{-1}\hat\psi'(x)|x|^\beta dx+ \sum_{i=1}^{m_0} d_i \int_{-M}^{-1} \lf[\int_{-\yy}^0 J_i(x-y)\psi_i(y){\rm d}y- 
 \psi_i(x) \rr]|x|^\beta{\rm d}x
	\\
	&+\sum_{i=1}^{m_0} d_i u_i^*\int_{-M}^{-1} \int_{0}^\yy |x|^\beta J_i(x-y){\rm d}y{\rm d}x\\
	\geq& -\sum_{i=1}^{m_0} d_i \int_{-M}^{-1}  
 \psi_i(x) |x|^\beta{\rm d}x
	+\sum_{i=1}^{m_0} d_i u_i^*\int_{-M}^{-1} \int_{0}^\yy |x|^\beta J_i(x-y){\rm d}y{\rm d}x
\end{align*}
Therefore, with $\wtd L:=L+\sum_{i=1}^{m_0}d_i$, we have
\begin{align*}
\wtd L \int_{-M}^{-1}\hat\psi (x)|x|^\beta dx
&\geq \sum_{i=1}^{m_0} d_i u_i^*\int_{-M}^{-1} \int_{0}^\yy |x|^\beta J_i(x-y){\rm d}y{\rm d}x\\
&= \sum_{i=1}^{m_0} d_i u_i^*\int_1^M\int_x^\infty x^\beta J_i(y)dydx\\
&=  \sum_{i=1}^{m_0} d_i u_i^*\Big[\int_1^M\int_1^\infty-\int_1^M\int_1^x\Big] x^{\beta} J_i(y)dydx\\
&= \sum_{i=1}^{m_0} \frac{d_i u_i^*}{\beta+1}\Big[\int_1^\infty (M^{\beta+1}-1)J_i(y)dy+\int_1^M (y^{\beta+1}-M^{\beta+1})J_i(y)dy\Big]\\
&\geq \sum_{i=1}^{m_0} \frac{d_i u_i^*}{\beta+1}\Big[ \int_1^M y^{\beta+1}J_i(y)dy -\int_1^\infty J_i(y)dy     \Big]\to\infty  \mbox{ as $M\to\infty$,}
\end{align*}
 since $\beta+1=\alpha$. Therefore \eqref{>*} holds, as we wanted.
\hfill $\Box$

\medskip

To prove Theorem \ref{thm4.3}, we need the following lemma.

\begin{lemma}\label{lem4.7}
	Let  the assumptions in Theorem \ref{thm4.3} be satisfied and $\Psi(x)=(\psi_i(x))=:\mathbf{u^*}-\Phi(x)$. Then  for every small $\epsilon>0$, there exist $\beta=\beta(\epsilon)\in(0, \lambda]$ and $C=C(\epsilon)>0$ such that for all $M>0$ and $i\in\{1,..., m\}$,
	\begin{align}\label{3.14}
	Q^{(i)}=Q^{(i)}_M:= \int_{-M}^{0}e^{-\beta x}\int_{-\yy}^{0}J_i(x-y)\psi_i(y)\rd y\rd x\leq  (1+\epsilon)\int_{-M}^{0}e^{-\beta x}\psi_i(x)\rd x+C.
	\end{align}
\end{lemma}
\begin{proof} 
	By a change of variables, we deduce
	\begin{align*}
	Q^{(i)}=&\int_{-M}^{0}e^{-\beta x}\int_{-\yy}^{-x}J_i(y)\psi_i(x+y)\rd y\rd x=\int_{0}^{M}\int_{-\yy}^{x}e^{\beta x}J_i(y)\psi_i(y-x)\rd y\rd x\\
	=&\int_{0}^{M}\lf(\int_{-\yy}^{0}+\int_{0}^{x}\rr)e^{\beta x}J_i(y)\psi_i(y-x)\rd y\rd x\\
	=&\int_{-\yy}^{0}\int_{0}^{M}e^{\beta x}J_i(y)\psi_i(y-x)\rd x\rd y+\int_{0}^{M}\int_{y}^{M}e^{\beta x}J_i(y)\psi_i(y-x)\rd x\rd y\\
	=&\int_{-\yy}^{0}e^{\beta y}J_i(y)\int_{-y}^{M-y}e^{\beta x}\psi_i(-x)\rd x\rd y+\int_{0}^{M}e^{\beta y}J_i(y)\int_{0}^{M-y}e^{\beta x}\psi_i(-x)\rd x\rd y\\
	:=&I+II.
	\end{align*}
	
	We have
	\begin{align*}
	I
	=&\int_{-M}^{0}e^{\beta y}J_i(y)\int_{-y}^{M-y}e^{\beta x}\psi_i(-x)\rd x\rd y+\int_{-\yy}^{-M}e^{\beta y}J_i(y)\int_{-y}^{M-y}e^{\beta x}\psi_i(-x)\rd x\rd y\\
	=&\int_{-M}^{0}e^{\beta y}J_i(y)\lf(\int_{-y}^{M}+\int_{M}^{M-y}\rr)e^{\beta x}\psi_i(-x)\rd x\rd y+\int_{-\yy}^{-M}e^{\beta y}J_i(y)\int_{-y}^{M-y}e^{\beta x}\psi_i(-x)\rd x\rd y\\
	=&\int_{-M}^{0}e^{\beta y}J_i(y)\int_{-y}^{M}e^{\beta x}\psi_i(-x)\rd x\rd y+\int_{-M}^{0}e^{\beta y}J_i(y)\int_{M}^{M-y}e^{\beta x}\psi_i(-x)\rd x\rd y\\
	&+\int_{-\yy}^{-M}e^{\beta y}J_i(y)\int_{-y}^{M-y}e^{\beta x}\psi_i(-x)\rd x\rd y\\
	:=&B_1^{(i)}+A_1^{(i)}+A_2^{(i)},
	\end{align*}
	and
	\begin{align*}
	II=&\int_{0}^{M}e^{\beta y}J_i(y)\int_{0}^{M}e^{\beta x}\psi_i(-x)\rd x\rd y-\int_{0}^{M}e^{\beta y}J_i(y)\int_{M-y}^{M}e^{\beta x}\psi_i(-x)\rd x\rd y\\
	:=&B_2^{(i)}+A_3^{(i)}.
	\end{align*}
	Hence,
	\begin{align*}
	Q^{(i)}=&I+II=(B_1^{(i)}+B_2^{(i)})+(A_1^{(i)}+A_2^{(i)}+A_3^{(i)})\\
	\leq&\int_{-M}^{0}e^{\beta y}J_i(y)\int_{0}^{M}e^{\beta x}\psi_i(-x)\rd x\rd y+\int_{0}^{M}e^{\beta y}J_i(y)\int_{0}^{M}e^{\beta x}\psi_i(-x)\rd x\rd y\\
	&+(A_1^{(i)}+A_2^{(i)}+A_3^{(i)})\\
	=&\int_{-M}^{M}e^{\beta y}J_i(y)\rd y\int_{0}^{M}e^{\beta x}\psi_i(-x)\rd x+(A_1^{(i)}+A_2^{(i)}+A_3^{(i)}).
	\end{align*}

	Set
	\begin{align*}
	P(\gamma):=\int_{\R}e^{\gamma y}J_i(y) \rd y=\int_0^\infty [e^{\gamma y}+e^{-\gamma y}]J_i(y)dy.
	\end{align*}
	Clearly  $P(\gamma)$ is increasing and continuous in $\gamma\in [0,\alpha]$, with $P(0)=1$.
	Hence there
	exists small $\beta_*=\beta_*(\epsilon)\in (0, \lambda]$ such that for all $0<\beta\leq \beta_*(\epsilon)$,
	\begin{align*}
	P(\beta)=\int_{\R}e^{\beta y}J_i(y) \rd y\leq 1+\epsilon.
	\end{align*}
	Thus, for such $\beta$,
	\begin{align*}
	Q^{(i)}\leq &(1+\epsilon)\int_{0}^{M}e^{\beta x}\psi_i(-x)\rd x+(A_1^{(i)}+A_2^{(i)}+A_3^{(i)}).
	\end{align*}
	
	It remains to verify that $A_1^{(i)}+A_2^{(i)}+A_3^{(i)}$ has an upper bound which is independent of  $M\in (0,\yy)$.  Using the monotonicity of $\psi_i$, we deduce
	\begin{align*}
	A_1^{(i)}+A_3^{(i)}=&\int_{-M}^{0}e^{\beta y}J_i(y)\int_{M}^{M-y}e^{\beta x}\psi_i(-x)\rd x\rd y-\int_{0}^{M}e^{\beta y}J_i(y)\int_{M-y}^{M}e^{\beta x}\psi_i(-x)\rd x\rd y\\
	\leq&\psi_i(-M)\int_{-M}^{0}e^{\beta y}J_i(y)\int_{M}^{M-y}e^{\beta x}\rd x\rd y-\psi_i(-M)\int_{0}^{M}e^{\beta y}J_i(y)\int_{M-y}^{M}e^{\beta x}\rd x\rd y\\
	=&\frac{\psi_i(-M)}{\beta}\int_{-M}^{0}e^{\beta y}J_i(y)[e^{\beta (M-y)}-e^{\beta M}]\rd y-\frac{\psi_i(-M)}{\beta}\int_{0}^{M}e^{\beta y}J_i(y)[e^{\beta M}-e^{\beta (M-y)}]\rd y\\
	=&\frac{\psi_i(-M)e^{\beta M}}{\beta}\int_{-M}^{0}J_i(y)[1-e^{\beta y}]\rd y-\frac{\psi_i(-M)e^{\beta M}}{\beta}\int_{0}^{M}J_i(y)[e^{\beta y}-1]\rd y\\
	=&\frac{\psi_i(-M)e^{\beta M}}{\beta}\int_{0}^{M} J_i(y)[2-e^{-\beta y}-e^{\beta y}]\rd y\leq 0,
	\end{align*}
	and
	\begin{align*}
	A_2^{(i)}=&\int_{-\yy}^{-M}e^{\beta y}J_i(y)\int_{-y}^{M-y}e^{\beta x}\psi_i(-x)\rd x\rd y\leq \psi_i(0)\int_{-\yy}^{-M}e^{\beta y}J_i(y)\int_{-y}^{M-y}e^{\beta x}\rd x\rd y\\
	=&\frac{u_i^*}{\beta}\int_{-\yy}^{-M}e^{\beta y}J_i(y)[e^{\beta (M-y)}-e^{-\beta y}]\rd y=\frac{u_i^*(e^{\beta M}-1)}{\beta}\int^{\yy}_{M}J_i(y)\rd y \\
	\leq & \frac{u_i^*(e^{\beta M}-1)}{\beta} e^{-\beta M} \int_M^{\yy} e^{\beta y}J_i(y)\rd y\leq \frac{ u_i^*}{\beta}\int_0^{\yy}e^{\beta y}J_i(y)\rd y:=C<\infty,
	\end{align*}
	since $\beta\leq \lambda$.
	 Hence \eqref{3.14} holds. 
\end{proof}

\begin{proof}[\underline{\rm Proof of Theorem \ref{thm4.3}}] With $\wtd \psi=\sum_{i=1}^m\td a_i \psi_i$ given by \eqref{td-psi}, it suffices to show that there exists  $\beta\in (0,\lambda]$  such that
	\begin{align*}
	\wtd \psi(x)=O(e^{\beta x}) \mbox{ for large negative }  x.
	\end{align*}
By Lemma \ref{lem4.7}, there exist $\epsilon>0$ and $\beta\in (0,\lambda]$ small such that \eqref{3.14} holds and $\hat b\geq \sum_{i=1}^m \td a_id_i \epsilon+c \beta$. Multiplying $e^{-\beta x}$  on both sides of the equation satisfied by $\wtd \psi$ and then integrating the resulting equation over the interval $[-M,0]$ with an arbitrary $M>0$,  we obtain
	\bes\label{3.15} \begin{aligned}
	&-\sum_{i=1}^m \int_{-M}^{0}\td a_i  g_i(\Psi(x))e^{-\beta x}\rd x-\int_{-M}^{0}c\wtd \psi'(x)(-x)^{\beta}\rd x\\
	=&\sum_{i=1}^m \td a_i d_i \int_{-M}^{0}\lf[ \int_{-\yy}^0 J_i(x-y)\psi_i(y){\rm d}y-  \psi_i(x) \rr] e^{-\beta x}\rd x\\
	&+\sum_{i=1}^m \td a_i d_i u_i^*\int_{-M}^{0}e^{-\beta x}\int_0^\yy J_i(x-y) {\rm d}y\rd x=:S_1(M)+S_2(M).
	\end{aligned}
	\ees
	
In view of $\mathbf{(J_2)}$ and $\beta\in (0,\lambda]$, we have 
	\begin{align*}
	S_2(M)=&\sum_{i=1}^m \td a_i d_i u_i^*\int_{-M}^{0}e^{-\beta x}\int_{-x}^\yy J_i(y) {\rm d}y\rd x\leq \sum_{i=1}^m \td a_i d_i u_i^*\int_{-\yy}^{0}e^{-\beta x}\int_{-x}^\yy J_i(y) {\rm d}y\rd x\\
	=& \sum_{i=1}^m \td a_i d_i u_i^*\int_0^{\yy}\int_{-y}^0 e^{-\beta x}J_i(y) {\rm d}x\rd y=\sum_{i=1}^m \frac{\td a_i d_i u_i^*}{\beta}\int_0^{\yy} [e^{\beta y}-1]J_i(y) \rd y<\yy.
	\end{align*}
	This together with \eqref{3.14} implies 
	\begin{align}\label{3.16}
	S_1(M)+S_2(M)\leq \sum_{i=1}^m \td a_id_i\epsilon\int_{-M}^{0}e^{-\beta x}\psi_i(x)\rd x+C_1
	\end{align}
	for some $C_1>0$ independent of $M$.

	On the other hand,  by \eqref{4.2} and $\hat b\geq \sum_{i=1}^m \td a_id_i\epsilon+c \beta$  we obtain, for $M> M_0\gg 1$,
	\begin{align*}
	&-\sum_{i=1}^m \int_{-M}^{0}\td a_i  g_i(\Psi(x))e^{-\beta x}\rd x-\int_{-M}^{0}c\wtd \psi'(x)e^{-\beta x}\rd x\\
	\geq & \hat b\int_{-M}^{-M_0}\wtd \psi(x)e^{-\beta x}\rd x-\int_{-M}^{0}c\wtd \psi'(x)e^{-\beta x}\rd x\\
	&-\sum_{i=1}^m \int_{-M_0}^{0}\td a_i  g_i(\Psi(x))e^{-\beta x}\rd x\\
	=&\hat b\int_{-M}^{0}\wtd \psi(x)e^{-\beta x}\rd x-\int_{-M}^{0}c\wtd \psi'(x)e^{-\beta x}\rd x+C_2\\
	\geq&\sum_{i=1}^m \td a_id_i\epsilon\int_{-M}^{0}\wtd \psi(x)e^{-\beta x}\rd x+c\beta\int_{-M}^{0}\wtd \psi(x)e^{-\beta x}\rd x-\int_{-M}^{0}c\wtd \psi'(x)e^{-\beta x}\rd x+C_2\\
	=& \sum_{i=1}^m \td a_id_i\epsilon\int_{-M}^{0}\wtd \psi(x)e^{-\beta x}\rd x- c\int_{-M}^{0}[\wtd \psi(x)e^{-\beta x}]'\rd x+C_2\\
	=&\sum_{i=1}^m \td a_id_i\epsilon\int_{-M}^{0}\wtd \psi(x)e^{-\beta x}\rd x- c\wtd \psi(0)+c\wtd \psi(-M)e^{\beta M}+C_2,
	\end{align*}
	where
	\[
	 C_2:=-\sum_{i=1}^m \int_{-M_0}^{0}\td a_i  g_i(\Psi(x))e^{-\beta x}\rd x-\int_{-M_0}^{0}\wtd \psi(x)e^{-\beta x}\rd x.
	 \]
	 Therefore, by \eqref{3.15} and \eqref{3.16}, 
	\begin{align*}
	c\wtd \psi(-M)e^{\beta M}\leq c\wtd \psi(0)+C_1-C_2 \mbox{ for all } M>M_0,
	\end{align*}
	which implies $\wtd \psi(x)=O(e^{\beta x})$ for $x\ll -1$. The proof is completed.
\end{proof}

\subsection{Bounds for $c_0t-h(t)$, $c_0t+g(t)$ and $U(t,x)$}
Let us first observe that using the reasoning in the proof of Theorem \ref{theorem1.3} (i), it suffices to estimate $h(t)-c_0t$, since that for $g(t)+c_0t$ follows by a simple change of the initial function.

Theorem \ref{theorem1.4} will follow easily from the following two lemmas and their proofs, where more general and stronger conclusions are proved.

\begin{lemma}\label{lemma4.3}  In Theorem \ref{theorem1.3}, if additionally $\mathbf{(J^\alpha)}$ holds for some $\alpha\geq 1$, $F$ is $C^2$ and   
  $\mathbf{u}^*\nabla F(\mathbf{u}^*)  \llp \mathbf{0}$, then there exists $C>0$ such that for $t\geq 0$,
\[
	h(t)-c_0t\geq -C\left[1+\int_0^t(1+x)^{-\alpha}dx+ \int_0^{\frac{c_0}2t} x^2\hat J(x)dx +t \int_{\frac{c_0}2t}^{\yy}x \hat J(x) dx\right],
	\]
	where $c_0>0$ is given in Theorem \ref{prop2.3} and $\hat J(x):=\sum_{i=1}^{m_0} \mu_iJ_i(x)$. 
\end{lemma}
\begin{proof} 
   Let $(c_0,\Phi^{c_0})$  be the unique solution pair of  \eqref{2.1a}-\eqref{2.2a} obtained in Theorem \ref{prop2.3}. To simplify notations we write $\Phi^{c_0}(x)=\Phi(x)=(\phi_i(x))$.  By Theorem \ref{thm4.1}  there is $C>0$ such that 
\begin{align}\label{4.13}
 \sum_{i=1}^{m_0}\int_{0}^\yy J_i(y) |y|^\alpha \rd y\leq C,
\ \ 0<u_i^*-\phi_i(x)\leq \frac{C}{x^\alpha} \ \ \mbox{ for } \ x< 0,\; i\in\{1,..., m\}.
\end{align}
   
	Define   
	\[
	\begin{cases}
	&\underline h(t):=c_0t+\delta(t),  \ \ \ t\geq 0,\\
	&\underline U(t,x):= (1-\epsilon(t))[\Phi(x-\underline h(t))+\Phi(-x-\underline h(t))-\mathbf{u}^*], \ \ \ t\geq 0, x\in [-\underline h(t),\underline h(t)],
	\end{cases}
	\]
	where $\epsilon(t):=(t+\theta)^{-\alpha}$ and
	\begin{align*}
	\delta(t):=&K_1-K_2\int_{0}^{t}\epsilon(\tau)\rd \tau-2\sum_{i=1}^{m_0}\mu_iu_i^*\int_{0}^{t} \int_{-\yy}^{- \frac{c_0}2(\tau+\theta)} \int_{0}^{\yy} J_i(x-y) \rd y\rd x\rd \tau,
	\end{align*}
	with 
	   $\theta$, $K_1$ and $K_2$ large positive constants to be determined.
	 
	For any $M>0$ and $i\in\{1,..., m_0\}$,
	\bes\label{4.4b}
	\begin{aligned}
	&\int_{-\yy}^{-M}\int_{0}^{\yy} J_i(x-y) \rd y\rd x=\int_{M}^{\yy}\int_{x}^{\yy} J_i(y) \rd y\rd x\nonumber\\
	=&\int_{M}^{\yy}\int_{M}^{y} J_i(y) \rd x\rd y=\int_{M}^{\yy}(y-k)J_i(y)\rd y\leq \int_{M}^{\yy}yJ_i(y)\rd y.
	\end{aligned}
	\ees
Hence, due to $\dd\int_0^\infty yJ_i(y)\rd y<\yy$, we have
\begin{align*}
&2\sum_{i=1}^{m_0}\mu_iu_i^*\int_{0}^{t} \int_{-\yy}^{- \frac{c_0}2(\tau+\theta)} \int_{0}^{\yy} J_i(x-y) \rd y\rd x\rd \tau\\
\leq& \ 2\sum_{i=1}^{m_0}\mu_iu_i^*\int_{0}^{t} \int_{-\yy}^{- \frac{c_0}2\theta} \int_{0}^{\yy} J_i(x-y) \rd y\rd x\rd \tau\\
\leq &\lf[2\sum_{i=1}^{m_0}\mu_iu_i^* \int_{\frac{c_0}2\theta}^{\yy} y J_i(y) \rd y \rr]t\leq \frac{c_0}4 t
\end{align*}
provided that $\theta>0$ is large enough, say $\theta\geq \theta_0$.

For any given small $\epsilon_0>0$, due to $\Phi(-\yy)=\mathbf{u}^*$  there is $K_0=K_0(\epsilon_0)>0$ such that 
	\begin{align*}
(1-\epsilon_0)\mathbf{u}^*\preceq \Phi(-K_0),
	\end{align*}
	which implies that
	\begin{align}\label{4.5a}
	 \Phi(x-\underline h(t)), \Phi(-x-\underline h(t))\in[(1-\epsilon_0)\mathbf{u}^*,  \mathbf{u}^*]\ \ \mbox{for }  \ x\in [-\underline h(t)+K_0,\underline h(t)-K_0],
	\end{align}
	where we have assumed $\underline h(0)= K_1>K_0$.

Clearly 
\begin{align*}
&K_2\int_{0}^{t}(\tau+\theta)^{-\alpha}\rd \tau\leq K_2\theta^{-\alpha}t\leq \frac{c_0}4 t
\end{align*}
provided 
$
\theta\geq  (4K_2/c_0)^{1/\alpha}$.
 Therefore  
	\begin{align}\label{4.16}
	\underline h(t)\geq \frac{c_0}2 t+K_1\geq \frac{c_0}2(t+\theta)>K_0 \mbox{ for all $t\geq 0$ provided that }
	\end{align} 
	\begin{align}\label{4.3a}
K_1\geq \frac{c_0}2\theta \mbox{ and } \theta\geq  \max\left\{(4K_2/c_0)^{1/\alpha}, \theta_0, 2K_0/c_0\right\}.
\end{align}

	Define
	\begin{align*}
	\epsilon_1:=\inf_{1\leq i\leq m}\inf_{x\in [-K_0,0]}{|\phi_i'(x)|}>0.
	\end{align*}
	Then 
	\begin{equation}\label{4.14}
	\begin{cases}
	\Phi'(x-\underline h(t))<-\epsilon_1\mathbf{1}  &\mbox{ for } x\in  [\underline h(t)-K_0,\underline h(t)],\\
	\Phi'(-x-\underline h(t))<-\epsilon_1\mathbf{1}  &\mbox{ for } x\in  [-\underline h(t),-\underline h(t)+K_0].
	\end{cases}
	\end{equation}

	{\bf Claim 1:} With   $\underline U=(\underline u_i)$, and suitably chosen $\theta,\ K_1,\ K_2$,  we have
	\begin{align}\label{4.15}
	&\underline h'(t)\leq \sum_{i=1}^m\mu_i \int_{-\underline h(t)}^{\underline h(t)} \int_{\underline h(t)}^{\yy} J_i(x-y) \underline u_i(t,x)\rd y,\ \ \ \ t> 0
	\end{align}
	and 
	\begin{align*}
	&-\underline h'(t)\geq - \sum_{i=1}^m\mu_i \int_{-\underline h(t)}^{\underline h(t)} \int_{-\yy}^{-\underline h(t)} J_i(x-y) \underline u_i(t,x)\rd y,\ \ \ \ t> 0.
	\end{align*}
	
	Due to $\underline U(t,x)=\underline U(t,-x)$ and $\mathbf{J}(x)=\mathbf{J}(-x)$, we just need to verify \eqref{4.15}.
	We calculate  
	\begin{align*}
	&\sum_{i=1}^{m_0}\mu_i \int_{-\underline h(t)}^{\underline h(t)} \int_{\underline h(t)}^{\yy} J_i(x-y) \underline u_i(t,x)\rd y\rd x\\
	=&(1-\epsilon)\sum_{i=1}^{m_0}\mu_i \int_{-2\underline h(t)}^{0} \int_{0}^{\yy} J_i(x-y) \phi_i(x)\rd y\rd x\\
	&+(1-\epsilon)\sum_{i=1}^{m_0}\mu_i \int_{-2\underline h(t)}^{0} \int_{0}^{\yy} J_i(x-y)[\phi_i(-x-2\underline h(t))-u_i^*]\rd y\rd x\\
	=&(1-\epsilon)c_0-(1-\epsilon)\sum_{i=1}^{m_0}\mu_i \int_{-\yy}^{-2\underline h(t)} \int_{0}^{\yy} J_i(x-y) \phi_i(x)\rd y\rd x\\
	&-(1-\epsilon)\sum_{i=1}^{m_0}\mu_i \int_{-2\underline h(t)}^{0} \int_{0}^{\yy} J_i(x-y)[u_i^*-\phi_i(-x-2\underline h(t))]\rd y\rd x.
	\end{align*}
	From \eqref{4.16}, for $t\geq 0$,
	\begin{align*}
	&(1-\epsilon)\sum_{i=1}^{m_0}\mu_i \int_{-\yy}^{-2\underline h(t)} \int_{0}^{\yy} J_i(x-y) \phi_i(x)\rd y\rd x\\ &+(1-\epsilon)\sum_{i=1}^{m_0}\mu_i \int_{-2\underline h(t)}^{-\underline h(t)} \int_{0}^{\yy} J_i(x-y)[u_i^*-\phi_i(-x-2\underline h(t))]\rd y\rd x\\
	\leq &\ 2\sum_{i=1}^{m_0}\mu_i u_i^*\int_{-\yy}^{-\underline h(t)} \int_{0}^{\yy} J_i(x-y) \rd y\rd x\\
	\leq &\ 2\sum_{i=1}^{m_0}\mu_i u_i^*\int_{-\yy}^{-\frac{c_0}2(t+\theta)} \int_{0}^{\yy} J_i(x-y) \rd y\rd x.
	\end{align*}
	And by \eqref{4.13}, we have, for $t>0$,
	\begin{align*}
	&(1-\epsilon)\sum_{i=1}^{m_0}\mu_i \int_{-\underline h(t)}^{0} \int_{0}^{\yy} J_i(x-y)[u_i^*-\phi_i(-x-2\underline h(t))]\rd y\rd x\\
	\leq& \sum_{i=1}^{m_0}\mu_i [u_i^*-\phi_i(-\underline h(t))]\int_{-\underline h(t)}^{0} \int_{0}^{\yy} J_i(x-y)\rd y\rd x\\
	\leq&\sum_{i=1}^{m_0}\mu_i \frac{C}{h(t)^\alpha}\int_{ -\yy}^{0} \int_{0}^{\yy} J_i(x-y)\rd y\rd x\\
	=&\sum_{i=1}^{m_0}\mu_i \frac{C}{h(t)^\alpha}\int_{0}^{\yy} yJ_i(y)\rd y
	\leq\sum_{i=1}^{m_0}\mu_i \frac{C^2}{(c_0/2)^\alpha (t+\theta)^\alpha}\leq \frac{K_2-c_0}{(t+\theta)^{\alpha}}
	\end{align*}
	if 
	\begin{align}\label{4.7a}
	  K_2\geq c_0+ \frac{C^2}{(c_0/2)^\alpha}\sum_{i=1}^m\mu_i.
	\end{align}
	Hence, when $\theta, K_1$ and $K_2$ are chosen such that \eqref{4.3a} and \eqref{4.7a} hold, then
	\begin{align*}
	&\sum_{i=1}^m\mu_i \int_{-\underline h(t)}^{\underline h(t)} \int_{\underline h(t)}^{\yy} J_i(x-y) \underline u_i(t,x)\rd y\rd x\\
	\geq&\ (1-\epsilon)c_0-2\sum_{i=1}^m\mu_i u_i^*\int_{-\yy}^{-\frac{c_0}2(t+\theta)} \int_{0}^{\yy} J_i(x-y) \phi_i(x)\rd y\rd x-\frac{K_2-c_0}{(t+\theta)^{\alpha}}\\
	=&\ c_0-K_2\epsilon(t)-2\sum_{i=1}^m\mu_i u_i^*\int_{-\yy}^{-\frac{c_0}2(t+\theta)} \int_{0}^{\yy} J_i(x-y) \phi_i(x)\rd y\rd x\\
	=&\  h'(t) \ \ \mbox{ for all } t>0,
	\end{align*}
	which finishes the proof of \eqref{4.15}.
	\smallskip
	
	{\bf Claim 2:} With $\theta,\; K_1, \ K_2$ chosen such that  \eqref{4.3a} and \eqref{4.7a} hold, and $K_2$ suitably further enlarged (see \eqref{4.8b} below),
	 $\theta_0\gg 1$ and $0<\epsilon_0\ll 1$, we have, for all $t>0$ and $x\in (-\underline h(t),\underline h(t))$, 
	\begin{align}
	\underline U_t(t,x)\preceq &D\circ \int_{-\underline h(t)}^{\underline h(t)}  \mathbf{J}(x-y)\circ \underline U(t,y)\rd y -D\circ \underline U(t,x)\nonumber+ F(\underline U(t,x)).
	\end{align}

	A simple calculation gives
	\begin{align*}
	\underline U_t= &  -\epsilon'(t)[\Phi(x-\underline h(t))+\Phi(-x-\underline h(t))-\mathbf{u}^*]\\
	&-(1-\epsilon)h'(t)[\Phi'(x-\underline h(t))+\Phi'(-x-\underline h(t))]\\
	=& \alpha(t+\theta)^{-\alpha-1}[\Phi(x-\underline h(t))+\Phi(-x-\underline h(t))-\mathbf{u}^*]\\
	&-(1-\epsilon)[c_0+\delta'(t)][\Phi'(x-\underline h(t))+\Phi'(-x-\underline h(t))],
	\end{align*}
	and using the equation satisfied by $\Phi$ we deduce
	\begin{align*}
&-(1-\epsilon)c_0[\Phi'(x-\underline h(t))+\Phi'(-x-\underline h(t))]\\
=&(1-\epsilon) \bigg[D\circ\int_{-\yy}^{\underline h(t)}\mathbf{J} (x-y)\circ \Phi(y-\underline h(t)) {\rm d}y-D\circ\Phi(x-\underline h(t))\\
	&\ \ \ \ \ \ \ \ \ \  +D\circ\int_{-\underline h(t)}^{\yy}\mathbf{J} (-x-y)\circ \Phi(-y-\underline h(t)) {\rm d}y-D\circ\Phi(-x-\underline h(t))\bigg]\\
	&+(1-\epsilon) \bigg[F(\Phi(x-\underline h(t)))+ F(\Phi(-x-\underline h(t))) \bigg]\\
	=&D\circ \lf[\int_{-\underline h(t)}^{\underline h(t)}  \mathbf{J}(x-y)\circ \underline U(t,y)\rd y -\underline U(t,x) \rr]\\
	&+(1-\epsilon) \bigg[D\circ\int_{-\yy}^{-\underline h(t)}\mathbf{J} (x-y)\circ [\Phi(y-\underline h(t))-\mathbf{u}^*] {\rm d}y\\
	&\ \ \ \ \ \ \ \ \ \ \ \ \ \ +D\circ\int_{\underline h(t)}^{\yy}\mathbf{J} (-x-y)\circ [\Phi(-y-\underline h(t)) {\rm d}y-\mathbf{u}^*]\rd y\bigg]\\
	&+(1-\epsilon) \bigg[F(\Phi(x-\underline h(t)))+ F(\Phi(-x-\underline h(t))) \bigg]\\
	\preceq & D\circ \lf[\int_{-\underline h(t)}^{\underline h(t)}  \mathbf{J}(x-y)\circ \underline U(t,y)\rd y -\underline U(t,x)\rr]\\
	&+(1-\epsilon) \bigg[F(\Phi(x-\underline h(t)))+ F(\Phi(-x-\underline h(t))) \bigg].
	\end{align*}
	Hence
	\begin{align*}
		\underline U_t
	\preceq  D\circ \int_{-\underline h(t)}^{\underline h(t)}  \mathbf{J}(x-y)\circ \underline U(t,y)\rd y -\underline U(t,x)
	+ F(\underline U(t,x))+A_1(t,x)+A_2(t,x),
	\end{align*}
	where
	\begin{align*}
	A_1(t,x):= &\alpha(t+\theta)^{-\alpha-1}[\Phi(x-\underline h(t))+\Phi(-x-\underline h(t))-\mathbf{u}^*],\\
	A_2(t,x):=&- (1-\epsilon) \delta'(t)[\Phi'(x-\underline h(t))+\Phi'(-x-\underline h(t))]\\
	&+(1-\epsilon)[F(\Phi(x-\underline h(t)))+ F(\Phi(-x-\underline h(t)))]- F(\underline U(t,x)).
	\end{align*}
	
	To finish the proof of Claim 2, it remains to check that
	\begin{align*}
	A_1(t,x)+A_2(t,x)\preceq  {\bf 0} \ \mbox{ for}\ \ t>0,\ x\in (-\underline h(t),\underline h(t)).
	\end{align*}
	We next prove this inequality for $x$ in the following  three intervals, separately:  
	\[
	I_1(t):=[\underline h(t)-K_0,\underline h(t)], \ I_2(t):= [-\underline h(t),-\underline h(t)+K_0], \ I_3(t):=[-\underline h(t)+K_0,\underline h(t)-K_0].
	\]
	
	For $x\in I_1(t)$, by \eqref{4.13},
	\begin{align*}
	\mathbf{0}\succ \Phi(-x-\underline h(t))-\mathbf{u}^*\succeq \Phi(K_0-2\underline h(t))-\mathbf{u}^*\succeq \Phi(-\underline h(t))-\mathbf{u}^* \succeq  \frac{-C}{h(t)^\alpha}\mathbf{1}
	\end{align*}
	Then by \eqref{1.3a} and $(\mathbf{f}_2)$,
	\begin{align*}
	&F(\Phi(-x-\underline h(t)))=F(\Phi(-x-\underline h(t)))-F(\mathbf{u}^*)\preceq L\frac{C}{h(t)^\alpha}\mathbf{1}
	\end{align*}
	and
	\begin{align*}
	F(\underline U(t,x))\succeq& (1-\epsilon)F\bigg(\Phi(x-\underline h(t))+\Phi(-x-\underline h(t))-\mathbf{u}^*\bigg)\\
	\succeq&(1-\epsilon)\big[F(\Phi(x-\underline h(t)))-L\frac{C}{h(t)^\alpha}\mathbf{1}\big].
	\end{align*}
	Thus from  the definition of $\delta(t)$, \eqref{4.16} and \eqref{4.14}, we deduce
	\begin{align*}
	A_2(t,x)\preceq& (1-\epsilon)  \bigg[\delta'(t)[\Phi'(x-\underline h(t))+\Phi'(-x-\underline h(t))]+F(\Phi(x-\underline h(t)))\\
	&+ F(\Phi(-x-\underline h(t)))-F\bigg(\Phi(x-\underline h(t))+\Phi(-x-\underline h(t))-\mathbf{u}^*\bigg)\bigg]\\
	\preceq & (1-\epsilon)\lf[-\delta'(t) \epsilon_1 +2L\frac{C}{h(t)^\alpha}  \rr]\mathbf{1}
	\preceq (1-\epsilon)\lf[- K_2(t+\theta)^{-\alpha}  \epsilon_1  +\frac{2LC}{h(t)^\alpha} \rr]\mathbf{1}\\
	\preceq & (1-\epsilon)(t+\theta)^{-\alpha}\big[- K_2  \epsilon_1 +2LC(2/c_0)^\alpha \big]\mathbf{1}.
	\end{align*}
	Moreover,
	\[
	A_1(t,x)\preceq \alpha (t+\theta)^{-\alpha-1}\mathbf{u}^*\leq 2|\mathbf{u}^*|(1-\epsilon)\alpha (t+\theta)^{-\alpha-1}\mathbf{1},
	\]
	where $|\mathbf{u}^*|:=\max_{1\leq i\leq m} u_i^*$ and by enlarging $\theta_0$ we have assumed that  $\epsilon(t)\leq \theta_0^{-\alpha}<1/2$.  Hence
	\begin{align*}
	A_1(t,x)+A_2(t,x)
	\preceq (1-\epsilon)(t+\theta)^{-\alpha} \Big[- K_2  \epsilon_1 +2LC(2/c_0)^\alpha + 2|\mathbf{u}^*|\alpha \theta_0^{-1}\Big]\mathbf{1}\preceq \mathbf{0}
	\end{align*}
	if additionally  
	\begin{align}\label{4.8b}
K_2\geq \epsilon_1^{-1}\Big[ 2LC(2/c_0)^\alpha + 2|\mathbf{u}^*|\alpha \theta_0^{-1}   \Big].
	\end{align}
	This proves the desired inequality for $x\in I_1(t)$.
	
Since $A_1(t,x)+A_2(t,x)$ is even in $x$, the desired inequality is also valid for  $x\in I_2(t)=-I_1(t)$. It remains to prove the desired inequality for $x\in I_3(t)$.
	
	The case $x\in I_3(t)$ requires some preparations.
	Define
	\begin{align*}
	G(u,v)=(g_i(u,v)):=(1-\epsilon)[F(u)+ F(v)]- F((1-\epsilon)(u+v-\mathbf{u}^*)),\ \ u,v\in \R^m.
	\end{align*}
	For $u\in  [\mathbf{0},\mathbf{u}^*]$ and $v\in  [\mathbf{0},\mathbf{u}^*]$, and each $i\in\{1,..., m\}$, we may apply the mean value theorem to the function
	\[
	\xi_i(t):=g_i(\mathbf{u^*}+t(u-\mathbf{u^*}),\mathbf{u^*}+t(v-\mathbf{u^*})
	\]
	to obtain
	\[
	\xi_i(1)=\xi_i(0)+\xi_i'(\zeta_i) \mbox{ for some } \zeta_i\in [0,1].
	\]
	Denote
	\[
	\tilde u=\tilde u^i:=\mathbf{u^*}+\zeta_i(u-\mathbf{u^*}), \ \ \td v=\tilde v^i:=\mathbf{u^*}+\zeta_i(v-\mathbf{u^*}).
	\]
	Then the above identity is equivalent to
	\begin{align*}
	g_i(u,v)=&g_i(\mathbf{u}^*,\mathbf{u}^*)+\nabla_{\! u}\,g_i(\td u,\td v)\cdot (u-\mathbf{u}^*)+\nabla_{\! v}\,g_i(\td u,\td v)\cdot (v-\mathbf{u}^*)\\
	=&-f_i((1-\epsilon)\mathbf{u}^*)+(1-\epsilon)\nabla f_i(\td u)\cdot (u-\mathbf{u}^*)+(1-\epsilon)\nabla f_i(\td v)\cdot  (v-\mathbf{u}^*)\\
	&-(1-\epsilon)\nabla f_i\big((1-\epsilon)(\td u+\td v-\mathbf{u}^*)\big)\cdot (u-\mathbf{u}^*)\\
	&-(1-\epsilon)\nabla f_i\big((1-\epsilon)(\td u+\td v-\mathbf{u}^*)\big)\cdot (v-\mathbf{u}^*).
	\end{align*}
	Let us note that $\td u\in [u,\mathbf{u}^*]$ and $\td v\in [v,\mathbf{u}^*]$. Since $F\in C^2$, there is $C_1$ such that 
	\begin{align*}
	|\partial_{jk} f_i(u)|\leq C_1 \ \mbox{ for } \ \ u\in [{\bf 0},\mathbf{\hat u}],\ i,j,k\in\{1,..., m\}.
	\end{align*}
	A simple calculation gives 
	\begin{align*}
	&(1-\epsilon)\nabla f_i(\td u)(u-\mathbf{u}^*)-(1-\epsilon)\nabla f_i\big((1-\epsilon)(\td u+\td v-\mathbf{u}^*)\big)\cdot(u-\mathbf{u}^*)\\
	= &(1-\epsilon)\bigg[\nabla f_i(\td u)-\nabla f_i\big((1-\epsilon)(\td u+\td v-\mathbf{u}^*)\big)\bigg]\cdot (u-\mathbf{u}^*)\\
	\leq &(1-\epsilon) b_1\sum_{j=1}^m (u_j^*-u_j),
	\end{align*}
	where  
		\begin{align*}
	b_1:=&\ C_1 | \td u-(1-\epsilon)(\td u+\td v-\mathbf{u}^*)|\\
	=&\ C_1 |\epsilon \td u-(1-\epsilon)(\td v-\mathbf{u}^*)|\leq C_1\sum_{j=1}^m[\epsilon \td u_j+(1-\epsilon)({u}_j^*-\td v_j)]\\
	\leq&\ C_2\epsilon+C_1\sum_{j=1}^m({u}_j^*- v_j) \ \mbox{with $C_2:=C_1\sum_{j=1}^{ m} u_j^*$. }
	\end{align*}
		 Similarly,
	\begin{align*}
	&(1-\epsilon)\nabla f_i(\td u)\cdot (v-\mathbf{u}^*)-(1-\epsilon)\nabla f_i\big((1-\epsilon)(\td u+\td v-\mathbf{u}^*)\big)\cdot (v-\mathbf{u}^*)\\
	\leq &\ (1-\epsilon) b_2 \sum_{j=1}^m (u_j^*-v_j),
	\end{align*}
	where 
	\begin{align*}
	b_2:= C_1|\epsilon \td u-(1-\epsilon)(\td u-\mathbf{u}^*)|
	\leq\ C_2\epsilon+C_1\sum_{j=1}^m({u}_j^*- u_j).
	\end{align*}
	Thus 
	\begin{align*}
	g_i(u,v)\leq & -f_i((1-\epsilon)\mathbf{u}^*)+(1-\epsilon) b_1 \sum_{j=1}^m({u}_j^*- v_j)+(1-\epsilon) b_2 \sum_{j=1}^m({u}_j^*- u_j)\\
	\leq &-f_i((1-\epsilon)\mathbf{u}^*)+\lf[C_2\epsilon+C_1\sum_{j=1}^m({u}_j^*- v_j)\rr]\sum_{k=1}^m(u_k^*-u_k)\\
	&+\lf[C_2\epsilon+C_1\sum_{j=1}^m({u}_j^*- u_j)\rr]\sum_{k=1}^m(u_k^*-v_k)\\
	= &-f_i((1-\epsilon)\mathbf{u}^*)+
	 C_2\epsilon\sum_{k=1}^m\Big[(u_k^*-u_k)+(u_k^*-v_k)\Big]\\
	&+C_1\sum_{j, k=1}^m({u}_j^*- v_j)(u_k^*-u_k)+C_1\sum_{j,k=1}^m({u}_j^*- u_j)(u_k^*-v_k)\\
	= &\ \epsilon\nabla f_i(\mathbf{u}^*)\cdot \mathbf{u}^*+o(\epsilon)+
	 C_2\epsilon\sum_{k=1}^m\Big[(u_k^*-u_k)+(u_k^*-v_k)\Big]\\
	&+2C_1\sum_{j,k=1}^m({u}_j^*- v_j)(u_k^*-u_k),
	\end{align*}
	where $o(\epsilon)/\epsilon\to 0$ as $\epsilon\to 0$. Since $\epsilon(t)\leq \theta_0^{-\alpha}$, we see that  
	\[
	o(\epsilon)=o(1)\epsilon \mbox{ with } o(1)\to 0 \mbox{ as } \theta_0\to\infty.
	\]

For our discussions below, it is convenient to introduce the notations
\begin{align*}
P(t,x)=(p_i(t,x)):=\mathbf{u}^*-\Phi(x-\underline h(t)), \ \ Q(t,x)=(q_i(t,x)):=\mathbf{u}^*-\Phi(-x-\underline h(t)). 
\end{align*}
Then 
	by \eqref{4.5a} we have
	\bes\label{4.13b}
		 P(t,x),\ Q(t,x)\in [\mathbf{0}, \epsilon_0\mathbf{u}^*]
	\mbox{ for } x\in I_3(t), t>0.
	\ees
	Moreover, since $\min\{x-\underline h(t), -x-\underline h(t)\}\leq -\underline h(t)$ always holds, 
	  by \eqref{4.13} and \eqref{4.16}, if we denote $C_3:=C(c_0/2)^{-\alpha}$, then
	 \begin{align}\label{4.10a}
	p_j(t,x)q_k(t,x)\leq  \frac{C\epsilon_0}{\underline h(t)^\alpha}\leq C_3\epsilon_0\epsilon(t)\ \mbox{ for }  x\in I_3(t),\ t>0,\ \  j, k\in\{1,..., m\}.
	 \end{align}
Let $A_2^i$ denote the $i$-th component of $A_2$. Now due to $\delta'(t)<0$ and $\Phi'\llp\mathbf{0}$, we have,  by \eqref{4.13b} and \eqref{4.10a},
	 \begin{align*}
	 A_2^i(t,x)\leq&\  g_i(\mathbf{u}^*-P, \mathbf{u}^*-Q)\\
	 \leq &\ \epsilon\nabla f_i(\mathbf{u}^*)\cdot \mathbf{u}^*+o(\epsilon)+
	 C_2\epsilon\sum_{k=1}^m (p_k+ q_k)+2C_1\sum_{j,k=1}^mp_jq_k\\
	\leq  &\ \epsilon\nabla f_i(\mathbf{u}^*)\cdot \mathbf{u}^*+o(\epsilon)+ C_4\epsilon_0\epsilon\\
	=&\ \epsilon\Big[\mathbf{u}^*\cdot \nabla f_i(\mathbf{u}^*)+o(1)+ C_4\epsilon_0\Big]\ \ \ 
	\mbox{ for }  x\in I_3(t), \    t>0,\ i\in\{1,..., m\},
	 \end{align*}
	 with
	 $
C_4:=2m(C_2+C_1C_3).
	 $ 
	Since
	 \[
	A^i_1(t,x)\leq \alpha (t+\theta)^{-\alpha-1}u_i^*\leq \alpha| u_i^*| \theta_0^{-1}\epsilon(t)
	\]
	and 
	\[
	\mathbf{u}^*[\nabla F(\mathbf{u}^*)]^T \llp {\bf 0},
	\]
	 we thus obtain
	 \[
	 A^i_1+A^i_2\leq \epsilon\Big(\mathbf{u}^*\cdot \nabla f_i(\mathbf{u}^*)+\Big[o(1)+
	  C_4\epsilon_0+\alpha u_i^*\theta_0^{-1}\Big]\Big)<{0} \ \ \ 
	\mbox{ for }  x\in I_3(t), \    t>0,\ i\in\{1,..., m\}
	 \]
	 provided that $\theta_0$ is sufficiently large and $\epsilon_0$ is sufficiently small.
The proof of Claim 2 is now complete.
\smallskip
	
	{\bf Claim 3:} There exists $t_0>0$ such that
	\bes\label{Uh}
	\begin{cases}
	&g(t+t_0)\leq -\underline h(t), \ h(t+t_0)\geq \underline h(t) \mbox{ for } t\geq 0,\\
	&U(t+t_0,x)\succeq \underline U(t,x) \mbox{ for } t\geq 0,\ x\in [-\underline h(t),\underline h(t)].
	\end{cases}
\ees
	It is clear that 
	\begin{align*}
	\underline U(t,\pm \underline h(t))=(1-\epsilon(t)) [\Phi(-2\underline h(t))-\mathbf{u}^*]\prec \mathbf{0} \mbox{ for } \ t\geq 0. 
	\end{align*}
	Since spreading happens for $(U,g,h)$,  there exists a large constant $t_0>0$ such that 
	\begin{align*}
	&g(t_0)<-K_1=-\underline h(0)\ {\rm and}\ \underline h(0)=K_1< h(t_0),\\
	&U(t_0,x)\succeq (1-\theta^{-\alpha})\mathbf{u}^*\succeq \underline U(0,x) \ \mbox{ for } \ \ x\in [-\underline h(0),\underline h(0)].
	\end{align*}
	which together with the inequalities proved in Claims 1 and 2 allows us to apply Lemma \ref{lemma3.2} 	to conclude that \eqref{Uh} is valid.
	\smallskip

	{\bf Claim 4:} There exists $C>0$ such that
	 \[
	\delta(t)\geq -C\left[1+\int_0^t(1+x)^{-\alpha}dx+ \int_0^{\frac{c_0}2t} x^2\hat J(x)dx +t \int_{\frac{c_0}2t}^{\yy}x \hat J(x) dx\right].
	\]
	
Clearly 
\begin{align*}
\int_0^t\epsilon(\tau)\rd \tau=\int_{0}^{t}(x+\theta)^{-\alpha}\rd x <\int_{0}^{t}(x+1)^{-\alpha}\rd x. 
\end{align*}
By changing order of integrations we have
\begin{align*}
&\int_{0}^{t} \int_{-\yy}^{- \frac{c_0}2(\tau+\theta)} \int_{0}^{\yy} J_i(x-y) \rd y\rd x\rd \tau\leq\int_{0}^{t} \int_{-\yy}^{- \frac{c_0}2 \tau} \int_{0}^{\yy} J_i(x-y) \rd y\rd x\rd \tau \\
=&\int_{0}^{t}\int_{\frac{c_0}2 \tau}^\infty\Big[y-\frac{c_0}2\tau\Big]  J_i(y) \rd y \rd \tau
\leq\int_{0}^{t}\int_{\frac{c_0}2\tau}^\infty yJ_i(y) \rd y \rd \tau\\
=& \ \frac{c_0}2\int_{0}^{\frac{c_0}2 t} y^2J_i(y)dy+t \int_{\frac{c_0}2t}^{\yy}y J_i(y) dy.
\end{align*}
The desired inequality now follows directly from  the definition of $\delta(t)$.
\end{proof}

\smallskip

Next we prove an upper bound for $h(t)-c_0t$. Let us note that we do not need the condition $\mathbf{(J^\alpha)}$ in the following result.
\begin{lemma}\label{lemma4.4}
 Under the assumptions of Theorem \ref{theorem1.3} (i), if $(\mathbf{J_1})$ holds, and  additionally $F$ is $C^2$ and   
  $\mathbf{u}^*[\nabla F(\mathbf{u}^*)]^T  \llp \mathbf{0}$, then there exits $C>0$ such that
	\begin{align}\label{4.18}
	&h(t)-c_0t\leq  C\ \  \mbox{ for all }\ t>0.
	\end{align}
\end{lemma}

\begin{proof} As in the proof of Lemma \ref{lemma4.3}, $(c_0,\Phi^{c_0})$  denotes the unique solution pair of  \eqref{2.1a}-\eqref{2.2a} obtained in Theorem \ref{prop2.3}, and to simplify notations we write $\Phi^{c_0}(x)=\Phi(x)=(\phi_i(x))$. 
	
	  For fixed $\beta>1$, and some large constants $\theta>0$ and $K_1>0$ to be determined, define
	\[\begin{cases}
	&\bar h(t):=c_0+\delta(t), \ \ \ t\geq 0,\\
	&\ol U(t,x):=(1+\epsilon(t)) \Phi(x-\bar h(t)),\ \ \ t\geq 0,\ x\leq \bar h(t),
	\end{cases}\]
	where $\epsilon(t):=(t+\theta)^{-\beta}$ and
	\begin{align*}
	\delta(t):=K_1+\frac{c_0}{1-\beta}[(t+\theta)^{1-\beta}-\theta^{1-\beta}] .
	\end{align*}

	Clearly,  there is a large constant $t_0>0$ such that 
	\begin{align*}
	U(t+t_0,x)\preceq (1+\frac 12\epsilon(0))\mathbf{u}^*\ \mbox{ for }\ \  t\geq 0,\ x\in [g(t), h(t)].
	\end{align*}
	Due to $\Phi(-\yy)=\mathbf{u}^*$, we may choose sufficient large $K_1>0$ such that $\underline h(0)=K_1>2h(t_0)$, $-\underline h(0)=-K_1<2g(t_0)$, and also
	\begin{align}\label{4.17a}
	\ol U(0,x)=(1+\epsilon(0)) \Phi(-K_1/2)\succ (1+\frac 12\epsilon(0))\mathbf{u}^* \succeq  U(t_0,x) \ \mbox{ for } \ x\in [g(t_0),h(t_0)].
	\end{align}
	
	{\bf Claim 1:}	 We have, with  $\ol U=(\bar u_i)$,
	\begin{align*}
	\bar h'(t)\geq  \sum_{i=1}^m\mu_i \int_{g(t+t_0)}^{\bar h(t)} \int_{\bar h(t)}^{+\yy} J_i(x-y) \bar u_i(t,x)\rd y \ \mbox{ for }\ \ \ t> 0.
	\end{align*}

	A direct  calculation shows 
	\begin{align*}
	&\sum_{i=1}^m\mu_i \int_{g(t+t_0)}^{\bar h(t)} \int_{\bar h(t)}^{+\yy} J_i(x-y) \bar u_i(t,x)\rd y\\
	\leq & \sum_{i=1}^m\mu_i \int_{-\yy}^{\bar h(t)} \int_{\bar h(t)}^{+\yy} J_i(x-y) \bar u_i(t,x)\rd y\\
	= &  (1+\epsilon)\sum_{i=1}^m\mu_i \int_{-\yy}^{0} \int_{0}^{+\yy} J_i(x-y) \phi_i(x)\rd y\\
	= & (1+\epsilon)c_0=  \bar h'(t),
	\end{align*}
	as desired.
	\smallskip
	
	{\bf Claim 2:} If $\theta>0$ is sufficiently large, then for $t>0$ and $x\in (g(t+t_0),\underline h(t))$, we have
	\begin{align}\label{4.18b}
	\ol U_t(t,x)\succeq &D\circ \int_{g(t+t_0)}^{\bar h(t)}  \mathbf{J}(x-y)\circ \ol U(t,y)\rd y -D\circ\ol U(t,x)+F(\ol U(t,x)).
	\end{align}
	
	By \eqref{2.1a}, we have
	\begin{align*}
	\ol U_t(t,x)=&-(1+\epsilon)[c_0+\delta'(t)]\Phi'(x-\bar h(t))+\epsilon'(t)\Phi(x-\underline h(t))\\
	=& -(1+\epsilon)c_0\Phi'(x-\bar h(t))-(1+\epsilon)\delta'(t)\Phi'(x-\bar h(t))-\beta(t+\theta)^{-\beta-1}\Phi(x-\underline h(t))\\ 
	\succeq   &D\circ \int_{g(t_0+t)}^{\bar h(t)}  \mathbf{J}(x-y)\circ \ol U(t,y)\rd y -D\circ\ol U(t,x)+F(\ol U(t,x))+A(t,x)
	\end{align*}
	with
	\begin{align*}
	A(t,x):=&(1+\epsilon)F(\Phi(x-\bar h(t)))-F((1+\epsilon)\Phi(x-\bar h(t)))\\
	&-(1+\epsilon)\delta'(t)\Phi'(x-\bar h(t))-\beta(t+\theta)^{-\beta-1}\Phi(x-\underline h(t)).
	\end{align*}
	To prove the claim, we need to show 
	\[A(t,x)\succeq \mathbf{0} \ \mbox{ for $x\in [g(t_0+t),\bar h(t)]$ and $t>0$}.
	\]
	
	Let $\epsilon_0,\ \epsilon_1$ and $K_0$ be given as in the proof of Lemma \ref{lemma4.3}. For $x\in [\bar h(t)-K_0,\bar h(t)]$ and $t>0$, by \eqref{4.14}, we have
	\begin{align*}
	A(t,x)\succeq& -(1+\epsilon)\delta'(t)\Phi'(x-\bar h(t))-\beta(t+\theta)^{-\beta-1}\Phi(x-\underline h(t))\\
	=&-(1+\epsilon)c_0(t+\theta)^{-\beta}\Phi'(x-\bar h(t))-\beta(t+\theta)^{-\beta-1}\Phi(x-\underline h(t))\\
	\succeq&c_0(t+\theta)^{-\beta}\epsilon_1\mathbf{1}-\beta(t+\theta)^{-\beta-1}\mathbf{u}^*\\
	\succeq &(t+\theta)^{-\beta-1}\big[c_0\theta\epsilon_1\mathbf{1}-\beta\mathbf{u}^*\big]
	\succeq  \mathbf{0},
	\end{align*}
	provided $\theta$ is large enough.

	We next estimate $A(t,x)$ for $x\in [g(t+t_0),\underline h(t)-K_0]$. Define
	\begin{align*}
	G(u)=(g_i(u)):=(1+\epsilon)F(u)- F((1+\epsilon)u),\ \ u,v\in \R^m.
	\end{align*}
	Then for $u, v\in  [\mathbf{0},\mathbf{u}^*]$ and $i\in\{1,..., m\}$,
	\begin{align*}
	g_i(u)=&g_i(\mathbf{u}^*)+\nabla g_i(\td u)\cdot (u-\mathbf{u}^*)\\
	=&-f_i((1+\epsilon)\mathbf{u}^*)+(1+\epsilon)\nabla f_i(\td u)\cdot (u-\mathbf{u}^*)-(1+\epsilon)\nabla f_i((1+\epsilon)\td u)\cdot (u-\mathbf{u}^*)\\
	=&-f_i((1+\epsilon)\mathbf{u}^*)+(1+\epsilon)\bigg[\nabla f_i(\td u)-\nabla f_i((1+\epsilon)\td u)\bigg]\cdot (u-\mathbf{u}^*)
	\end{align*}
	for some $\td u=\td u^i\in [u,\mathbf{u}^*]$. Since $F\in C^2$, there exists $C_1>0$ such that 
	\begin{align*}
	|\partial_{jk} f_i(u)|\leq C_1\ \mbox{ for } \ u\in [0,\mathbf{\hat u}],\ i,j,k\in\{1,..., m\}.
	\end{align*}
	Therefore
	\begin{align*}
	g_i(u)\geq& -f_i((1+\epsilon)\mathbf{u}^*)-(1+\epsilon)b_1\sum_{j=1}^m(u_j^*-u_j)
	\end{align*}
	with  
	\begin{align*}
	b_1:= C_1 |\epsilon \td u|\leq C_1\epsilon|\mathbf{u}^*|:= C_2\epsilon.
	\end{align*}
	Thus 
	\begin{align*}
	g_i(u)\geq& -\epsilon\nabla f_i(\mathbf{u}^*)\cdot \mathbf{u}^*+o(\epsilon)-2C_2\epsilon\sum_{j=1}^m(u_j^*-u_j).
	\end{align*}
	By \eqref{4.5a} we have 
	\begin{align}\label{4.19}
	-\epsilon_0\mathbf{u}^*	\preceq \Phi(x-\bar h(t))-\mathbf{u}^*\llp \mathbf{0} \ \mbox{ for }\ \ x\in [g(t_0+t),\underline h(t)-K_0],\ t>0.
	\end{align}
	Using \eqref{4.5a}, $\delta'>0$, $\Phi'\preceq \mathbf{0}$ and $\epsilon=(t+\theta)^{-\beta}\leq \theta^{-\beta}$, we obtain
	\begin{align*}
	A^i(t,x)\geq &(1+\epsilon)f_i(\Phi(x-\bar h(t)))-f_i((1+\epsilon)\Phi(x-\bar h(t)))-\beta(t+\theta)^{-\beta-1}\phi_i(x-\underline h(t))\\
	=&\ g_i(\Phi(x-\bar h(t))-\beta(t+\theta)^{-\beta-1}\phi_i(x-\underline h(t))\\
	\geq&\ \epsilon\bigg[-\mathbf{u}^*\cdot \nabla f_i(\mathbf{u}^*)+o(1)-2\epsilon_0C_2 \sum_{j=1}^m u_j^*  -\beta\theta^{-\beta-1}u_i^*\bigg]\\
	>&\ {0}\ \  \mbox{ for }\ \ x\in [g(t_0+t),\underline h(t)-K_0],\ t>0,\ i\in\{1,..., m\},
	\end{align*}
	provided $\theta$ is large enough and $\epsilon_0>0$ is small enough,  since $\mathbf{u}^*[\nabla F(\mathbf{u}^*)]^T\llp \mathbf{0}$. We have now proved \eqref{4.18b}.

Due to the inequalities proved in Claims 1 and 2,   \eqref{4.17a} and
	\begin{align*}
	\ol U(t,g(t+t_0))>0,\ \  \ol U(t,\bar h(t))=(1+\epsilon) \Phi(\bar h(t)-\bar h(t))= 0 \ \mbox{ for }\ \ t\geq 0,
	\end{align*}
	we are now able to apply Lemma \ref{lemma3.2}  to conclude that
	\begin{align*}
	& h(t+t_0)\leq  \bar h(t), \ &&t\geq 0,\\
	&U(t+t_0,x)\preceq \ol U(t,x),&&t\geq 0,\ x\in [ g(t+t_0),\underline h(t)].
	\end{align*}
	The desired inequality \eqref{4.18} follows directly from $\delta(t)\leq K_1+\frac{c_0}{\beta-1}\theta^{1-\beta}$  and $h(t+t_0)\leq  \bar h(t)$. The proof is complete.
\end{proof}

\begin{proof}[\underline{\rm Proof of Theorem \ref{theorem1.4}}] Since $\alpha\geq 2$, from the proof of Lemmas \ref{lemma4.3} and \ref{lemma4.4}, it is easily seen that 
\[
C_0:=\sup_{t>0} \big[|\bar h(t)-c_0t|+|\underline h(t)-c_0t|\big]<\infty.
\]
Hence for large fixed $\theta>0$ and all large $t$, say $t\geq t_0$,
\[
[g(t), h(t)]\supset [-\underline h(t-t_0), \underline h(t-t_0)]\supset [-c_0t+C, c_0t-C] \mbox{ with } C:=C_0+c_0t_0,
\]
and
\[
U(t,x)\succeq \underline U(t,x)\succeq (1-\epsilon(t)]\big[\Phi^{c_0}(x-c_0t+C)+\Phi^{c_0}(-x-c_0t+C)-\mathbf{u}^*\big]
\]
for $x\in [-c_0t+C, c_0t-C]$, where $\epsilon(t)=(t+\theta)^{-\alpha}$. This inequality for $U(t,x)$ also holds for $x\in [g(t), h(t)]$  if we assume that $\Phi^{c_0}(x)=0$ for $x>0$, since when $x$
lies outside of $[-c_0t+C, c_0t-C]$ the right side is $\prec {\bf 0}$.

Using the reasoning in the proof of Theorem \ref{theorem1.3}, from the proof of Lemma \ref{lemma4.4} we see that the following analogous inequalities hold:
\[
g(t)\geq -\bar h(t-t_0),\; U(t,x)\preceq (1+\epsilon(t))\Phi^{c_0}(-x-\bar h(t-t_0))
\]
for $t>t_0$ and $x\in [g(t), h(t)]$. We thus have
\[
[g(t), h(t)]\subset [-\bar h(t-t_0), \bar h(t-t_0)]\subset [-c_0t-C, c_0t+C],
\]
and
\[
U(t,x)\preceq \underline U(t,x)\preceq (1-\epsilon(t))\min\Big\{\Phi^{c_0}(x-c_0t-C), \Phi^{c_0}(-x-c_0t-C)\Big\}
\]
for $t>t_0$ and $x\in [g(t), h(t)]$. The proof is complete.
\end{proof}

\section{The growth orders of $c_0t-h(t)$ and $c_0t+g(t)$ }\label{section5}

Recall that $(U(t,x), g(t), h(t))$ is the unique positive solution of \eqref{1.1}, and we assume that spreading happens. Under the assumptions of Theorem \ref{theorem1.3}, we have
\[
-\lim_{t\to\yy}\frac{g(t)}{t}=\lim_{t\to \yy}\frac{h(t)}{t}=c_0>0. 
\]

In this section we determine the growth order of $c_0t-h(t)$ and $c_0t+g(t)$ when the kernel functions satisfy, for some
 ${\gamma}\in (2,3]$, $\omega\in ({\gamma}-1,\gamma]$, $C>0$ and all $|x|\geq 1$,
\begin{equation}\label{5.1}
\begin{cases}
\dd J_i(x)\approx  |x|^{-\gamma}& \mbox{ if $ i\in\{1,..., m_0\}$ and $\mu_i\not=0$},\\[2mm]
\dd  J_i(x)\leq C |x|^{-\omega}  & \mbox{ if $ i\in\{1,..., m_0\}$ and $\mu_i=0$}.
\end{cases}
\end{equation}
 Clearly, $(\mathbf{\hat J^\gamma})$ implies \eqref{5.1}.

The main result of this section is the following theorem.
\begin{theorem}\label{prop5.1}
In Theorem \ref{theorem1.3}, if additionally $\mathbf{(J^1)}$,   \eqref{5.1} and \eqref{1.7a} hold, then  for  $t\gg 1$,
	\[\left\{\begin{array}{ll}
	c_0t+g(t),\ c_0t- h(t) \approx \ t^{3-{\gamma}}  &{\rm if }\ {\gamma}\in (2,3],\\
	c_0t+g(t),\ c_0t-h(t) \approx  \ \ln t &{\rm if }\ {\gamma}=3.
	\end{array}\right.\]
\end{theorem}
It is clear that the conclusion of Theorem \ref{th1.5b}  follows directly from Theorem \ref{prop5.1}. Note that if $\omega>2$ in \eqref{5.1}, then $\mathbf{(J^1)}$ automatically holds. 

By $\mathbf{(f_1)}$ and the Perron-Frobenius theorem, we know that the matrix  $\nabla F(0)-\wtd D$ with $\wtd D={\rm diag}(d_1,..., d_m)$ has a principal eigenvalue $\tilde \lambda_1$
with a corresponding eigenvector $V^*=(v_1^*,\cdots,v_m^*)\ggs \mathbf{0}$, namely
\begin{equation}\label{D}
V^*\Big([\nabla F(0)]^T-\wtd D\Big)=\tilde\lambda_1 V^*.
\end{equation}

To prove Theorem \ref{prop5.1}, the difficult part is to find the lower bound for $c_0 t-h(t)$, which will be established according to the following two cases:  (i)  $\td\lambda_1<0$,\ (ii)  $\td\lambda_1\geq 0$.

As before, we will only estimate $c_0t-h(t)$, since the estimate for $c_0t+g(t)$ follows by making the variable change $x\to -x$ in the initial functions.
\subsection{The case $\td\lambda_1<0$}
\begin{lemma}\label{lemma5.2}
Suppose that the assumptions in Theorem \ref{prop5.1} are satisfied. 	If   $\td\lambda_1<0$,  then there 
exists $\sigma=\sigma({\gamma})>0$ such that for all large $t>0$,
	\begin{equation}\label{7.3}
	\begin{cases}
	c_0t-h(t)\geq \sigma\,t^{3-{\gamma}}& {\rm if}\ {\gamma}\in (2,3),\\
	c_0t-h(t)\geq \sigma\ln t &{\rm if}\ {\gamma}=3.
	\end{cases}
	\end{equation}
\end{lemma}
\begin{proof}  Let $\beta:={\gamma}-2\in (0,1]$, and $(c_0,\Phi)$  be the solution of  \eqref{2.1a}-\eqref{2.2a}. Define
	\begin{align*}
	\epsilon(t):=K_1(t+\theta)^{-\beta}, \ 	\  \delta(t):=K_2-K_3\int_{0}^{t}\epsilon(\tau)\rd \tau
	\end{align*}
	and
	\[\begin{cases}
	\bar h(t):=c_0 t+\delta(t),  &\ \ \ t\geq 0,\\
	\ol U(t,x):=(1+\epsilon(t)) \Phi(x-\bar h(t))+\rho(t,x), &\ \ \ t\geq 0,\  x\leq \bar h(t),
	\end{cases}
	\]
	where 
	\[
	\rho(t,x):=K_4\xi(x-\ol h(t))\epsilon(t)V^*,\ 
	\]
	with $\xi\in C^2(\R)$ satisfying
	\begin{align}\label{7.4a}
	\ 0\leq \xi(x)\leq 1, \ \ \ \xi(x)=1\ {\rm for}\ |x|<\tilde{\epsilon},\ \xi(x)=0\ {\rm for}\ |x|>2\tilde{\epsilon},
	\end{align}
	and the positive constants $\theta$, $K_1,  K_2, K_3, K_4$, $\tilde{\epsilon}$ are to be  determined.  
	
	We are going to show that, it is possible to choose these constants and some $t_0>0$ such that
	\begin{align}\label{7.28}
	\hspace{-1.5cm}\ol U_t(t,x)\succeq  D\circ \int_{g(t+t_0)}^{\bar h(t)} \ \mathbf{J}(x-y) \circ\ol U (t,y)\rd y -\ol U(t,x)+F(\ol U(t,x))
	\end{align}
	\hspace{10.9cm} for $t>0, \ x\in (g(t+t_0),\overline h(t))$, 
	\begin{align}
	&\bar h'(t)\geq \sum_{i=1}^{m_0}\mu_i \int_{g(t+t_0)}^{\bar h(t)} \int_{\bar h(t)}^{+\yy} J_i(x-y) \bar u_i(t,x)\rd y&& \mbox{ for }   t> 0,\label{h'(t)}\\
	&\ol U(t,g(t+t_0))\succeq {\bf 0},\ \  	\ol U(t,\bar h(t))\succeq {\bf 0} &&\mbox{ for }   t\geq 0,\label{g(t)}\\
	&\ol U(0,x)\succeq    U(t_0,x),\; \ol h(0)\geq h(t_0) &&  \mbox{ for }   x\in [g(t_0),h(t_0)].\label{t=0}
	\end{align}
	If these inequalities are proved, then by the comparison principle, we obtain
	\[
	\ol h(t)\geq h(t+t_0),\; \ol U(t,x)\succeq U(t+t_0,x) \mbox{ for } t>0,\; x\in [g(t+t_0), h(t+t_0)],
	\]
	and the desired inequality for $c_0t-h(t)$ follows easily from the definition of $\ol h(t)$.
	
	Therefore, to complete the proof, it suffices to prove the above inequalities. We divide the arguments below into several steps.

	Firstly, by Theorem \ref{theorem1.3}, there  is $C_1>1$ such that 
	\begin{align}\label{7.26a}
         -g(t), h(t) \leq (c_0+1)t+C_1 \ \mbox{ for } \  t\geq 0.
	\end{align}
		Let us also note that \eqref{g(t)} holds trivially.
	
	{\bf Step 1}. Choose $t_0=t_0(\theta)$ and $K_2=K_2(\theta)$ so that \eqref{t=0} holds.
	
	For later analysis, we need to find $t_0=t_0(\theta)$ and $K_2=K_2(\theta)$ so that \eqref{t=0} holds and at the same time
	they have less than linear growth in $\theta$.
	
	 Let  $W^*\ggs {\bf 0}$ be an eigenvector corresponding to the maximal eigenvalue $\td \lambda$ of  $\nabla F(\mathbf{u}^*)$. 
	 By our assumptions on $F$, we have $\td \lambda<0$.  Hence there exists small $\epsilon_*>0$ such that for any $k\in (0,\epsilon_*]$,  
	\begin{align*}
	&F(\mathbf{u}^*+kW^*)= kW^* \Big([\nabla F(\mathbf{u^*})]^T+o(1){\bf I}_m\Big) \preceq \frac{k}2 \td\lambda W^* \llp {\bf 0},\\
	&F(\mathbf{u}^*-kW^*)=- kW^* \Big([\nabla F(\mathbf{u^*})]^T+o(1){\bf I}_m\Big) \succeq  -\frac{k}2 \td\lambda W^* \ggs {\bf 0}.
	\end{align*}
	It follows that,  for $\td\sigma =\td\lambda/2$,
	\begin{align*}
	\ol W(t)=\mathbf{u}^*+ \epsilon_*e^{\td\sigma t}W^*,\ \ \underline W(t)=\mathbf{u}^*-\epsilon_* e^{\td\sigma t}W^*
	\end{align*}
	are  a pair of upper and lower  solution of the ODE system  $W'=F(W)$  with initial data $W(0)\in [\mathbf{u}^*-\epsilon_*W^*, \mathbf{u}^*+\epsilon_*W^*]$.  

	By $\mathbf{(f_4)}$, the unique solution of the ODE system
	\[W'=F(W),  \ W(0)= (\|u_{10}\|_\infty,\cdots,\|u_{m0}\|_\infty)
	\]  
	satisfies
	 $\lim_{t\to \yy} W(t)=\mathbf{u}^*$. Hence there exists $t_*>0$ such that
	 \[
	 W(t_*)\in [\mathbf{u}^*-\epsilon_*W^*, \mathbf{u}^*+\epsilon_*W^*]. 
	 \]
	 Using the above defined upper solution $\ol W(t)$ we obtain
	 \[
	 W(t+t_*)\preceq  \mathbf{u^*}+\epsilon_* e^{\tilde\sigma t} W^*\preceq (1+\tilde\epsilon_* e^{\tilde\sigma t})\mathbf{u^*} \mbox{ for } t\geq 0,
	 \]
	 where $\tilde\epsilon_*>0$ is chosen such that $\epsilon_*W^*\leq \tilde\epsilon_* \mathbf{u^*}.$
	 By the comparison principle we deduce
	 \[
	 U(t+t_*,x)\preceq W(t+t_*)\preceq  (1+\tilde\epsilon_* e^{\tilde\sigma t})\mathbf{u^*}  \mbox{ for } t\geq 0,\; x\in [g(t+t_*), h(t+t_*)].
	 \]
	 Hence
	 \[
	 U(t_0, x)\preceq (1+\frac{\epsilon(0)}2)\mathbf{u^*} \mbox{ for } x\in [g(t_0), h(t_0)]
	 \]
	 provided that
	 \[
	 t_0=t_0(\theta):=\frac\beta{|\tilde\sigma|}\ln \theta+\frac{\ln (2\tilde\epsilon_*/K_1)}{|\td\sigma|}+t_*.
	 \]
	 
	 By \eqref{5.1}, for any fixed  $\omega_*\in (\beta,\omega-1)$,  we have
	\begin{align*}
	\int_{\R} J(x)|x|^{\omega_*}\rd x<\yy.
	\end{align*}
	Then by Theorem \ref{theorem1.6}, there is $C_2$ such that 
	\begin{align*}
	\mathbf{u^*}-\Phi(x)\leq \frac{C_2}{|x|^{\omega_*}} \mathbf{u^*}\mbox{ for }\ x\leq -1.
	\end{align*}
	Hence, for $K>1$ we have
	\begin{align*}
	&(1+\epsilon(0)) \Phi(-K)- (1+\epsilon(0)/2)\mathbf{u}^*\\
	\succeq &(1+\epsilon(0))\big[1-C_2K^{-\omega_*}\big]\mathbf{u}^*- (1+\epsilon(0)/2)\mathbf{u}^*\\
	=&\big[K_1\theta^{-\beta}/2-C_2K^{-\omega_*}(1+K_1\theta^{-\beta})\big]\mathbf{u}^*\\
	\succeq &{\bf 0}
	\end{align*}
	provided that
	\[
	K^{\omega_*}\geq 2C_2+\frac{2C_2}{K_1}\theta^\beta.
	\]
	Therefore, for all $K_1\in (0, 1]$, $\theta\geq 1$ and $K\geq (4C_2/K_1)^{1/\omega_*}\theta^{\beta/\omega_*}$, we have
	\[
	(1+\epsilon(0)) \Phi(-K)- (1+\epsilon(0)/2)\mathbf{u}^*
	\succeq {\bf 0}.
	\]
	Now define
	\begin{equation}\label{K2}
	K_2(\theta):=2\max\left\{(4C_2/K_1)^{1/\omega_*}\theta^{\beta/\omega_*}, (c_0+1)t_0(\theta)+C_1\right\}.
	\end{equation}
Then for $K_2=K_2(\theta)$ we have
\[
\ol h(0)=K_2>K_2/2 \geq (c_0+1)t_0+C_1\geq h(t_0),
\]
and
	 for $ x\in [g(t_0),h(t_0)]$, 
	\begin{align*}
	\ol U(0,x)=(1+\epsilon(0)) \Phi(x-K_2)\succeq  (1+\epsilon(0)) \Phi(-K_2/2)
	\succeq (1+\epsilon(0)/2)\mathbf{u}^*.	
	\end{align*}
	Thus \eqref{t=0} holds if $t_0$ and $K_2$ are chosen as above, for any $\theta\geq 1$, $K_1\in (0, 1]$.

	{\bf Step 2.} We verify that \eqref{h'(t)} holds if $\theta$, $K_1, K_3$ and $ K_4$ are chosen suitably.
	
	Denote
	\begin{align}\label{7.5}
	C_3:=\sum_{i=1}^{m_0}\mu_i\int_{-\yy}^{0} \int_{0}^{+\yy} J_i(x-y) \rd y\rd x=\sum_{i=1}^{m_0}\mu_i\int_{0}^{+\yy} J_i(y)y \rd y.
	\end{align}
With $\rho=(\rho_i)$,	a direct  calculation shows 
	\begin{align*}
	&\sum_{i=1}^{m_0} \mu_i\int_{g(t+t_0)}^{\bar h(t)} \int_{\bar h(t)}^{+\yy} J_i(x-y)  \bar u_i(t,x)\rd y\rd x\\
	= & \sum_{i=1}^{m_0} \mu_i \int_{-\yy}^{\bar h(t)} \int_{\bar h(t)}^{+\yy} J_i(x-y)  \bar u_i(t,x)\rd y\rd x-\sum_{i=1}^{m_0} \mu_i \int_{-\yy}^{g(t+t_0)} \int_{\bar h(t)}^{+\yy} J_i(x-y)  \bar u_i(t,x)\rd y\rd x\\
	=&\sum_{i=1}^{m_0} \mu_i \int_{-\yy}^{0} \int_{0}^{+\yy} J_i(x-y) [(1+\epsilon) \phi_i(x)+\rho_i(t, x+\ol h(t))]\rd y\rd x\\
	&-\sum_{i=1}^{m_0} \mu_i \int_{-\yy}^{g(t+t_0)-\bar h(t)} \int_{0}^{+\yy} J_i(x-y) [(1+\epsilon) \phi_i(x)+\rho_i(t, x+\ol h(t))]\rd y\rd x\\
	\leq& (1+\epsilon) c_0+ C_3K_4\epsilon |V^*|
	-\sum_{i=1}^{m_0} \mu_i \int_{-\yy}^{g(t+t_0)-\bar h(t)} \int_{0}^{+\yy} J_i(x-y) (1+\epsilon) \phi_i(x)\rd y\rd x\\
	\leq&(1+\epsilon) c_0+ C_3K_4\epsilon |V^*|
	-\sum_{i=1}^{m_0} \mu_i \int_{-\yy}^{g(t+t_0)-\bar h(t)} \int_{0}^{+\yy} J_i(x-y)  \phi_i(x)\rd y\rd x,
	\end{align*}
	where
	\[ |V^*|:=\max_{1\leq i\leq m} v_i^*.
	\]
	By elementary calculus, for any $k>1$, 
	\begin{equation}\label{7.24b}
	\begin{aligned}
	&\int_{-\yy}^{-k}\int_{0}^{\yy} \frac{1}{|x-y|^{2+\beta}} \rd y\rd x=\int_{-\yy}^{-k}\int_{-x}^{\yy} \frac{1}{y^{2+\beta}} \rd y\rd x=\int_{k}^{\yy}\int_{x}^{\yy} \frac{1}{y^{2+\beta}} \rd y\rd x\\
	=&\int_{k}^{\yy}\int_{k}^{y} \frac{1}{y^{2+\beta}} \rd x\rd y=\int_{k}^{\yy}\frac{y-k}{y^{2+\beta}} \rd y=\beta^{-1}(1+\beta)^{-1} k^{-\beta}.
	\end{aligned}
	\end{equation}
	From \eqref{5.1} 
	and \eqref{7.26a}, there exists $C_4>0$ such that
	\begin{equation}\label{7.7}
	\begin{aligned}
	&\sum_{i=1}^{m_0} \mu_i\int_{-\yy}^{g(t+t_0)-\bar h(t)} \int_{0}^{+\yy} J_i(x-y)  \phi_i(x)\rd y\rd x\\
	\geq& C_4\lf[\min_{1\leq i\leq m}\phi_i(g(t+t_0)-\bar h(t)) \rr]\int_{-\yy}^{g(t+t_0)-\bar h(t)} \int_{0}^{+\yy} \frac{1}{|x-y|^{2+\beta}}  \rd y\rd x\\
	\geq& \phi_*{C_4} \int_{-\yy}^{g(t+t_0)-\bar h(t)} \int_{0}^{+\yy} \frac{1}{|x-y|^{2+\beta}}  \rd y\rd x =\frac{\phi_*C_4}{\beta (1+\beta) }(|g(t+t_0)|+\bar h(t))^{-\beta}\\
	\geq &\frac{\phi_*C_4}{\beta (1+\beta) }[(c_0+1)(t+t_0)+C_1+c_0t+K_2]^{-\beta}\\
	=&\frac{\phi_*C_4}{\beta (1+\beta)(2c_0+1)^\beta }\lf[t+\frac{(c_0+1)t_0+C_1+K_2}{(2c_0+1)}\rr]^{-\beta},
	\end{aligned}
	\end{equation}
	where  $\dd\phi_*=\min_{1\leq i\leq m}\phi_i(-1)\leq \min_{1\leq i\leq m}\phi_i(-K_2)\leq \min_{1\leq i\leq m}\phi_i(g(t+t_0)-\bar h(t))$.
	Therefore, for all large $\theta>0$ so that
	\begin{align}\label{7.26}
	\theta>\frac{(c_0+1)t_0+C_3+K_2}{(2c_0+1)},
	\end{align}
	which is possible since $t_0(\theta)$ and $K_2(\theta)$ grow slower than linearly in $\theta$, we have
	\begin{align*}
	&\sum_{i=1}^{m_0} \mu_i\int_{g(t+t_0)}^{\bar h(t)} \int_{\bar h(t)}^{+\yy} J_i(x-y) \bar u_i(t,x)\rd y\rd x\\
	\leq&   (1+\epsilon(t)) c_0+ C_4K_4\epsilon(t) |V^*|-\frac{ \phi_*C_4}{\beta (1+\beta)(2c_0+1)^\beta }\lf(t+\theta\rr)^{-\beta}\\
		=&c_0+\epsilon(t)\lf[c_0+ C_4K_4|V^*|-\frac{\phi_*C_4}{K_1\beta (1+\beta)(2c_0+1)^\beta }\rr]\\
	\leq& c_0-K_3\epsilon(t)=h'(t)
	\end{align*}
	provided that $K_1, K_3$ and $K_4$ are small enough so that
	\begin{align}\label{7.27}
	K_1(c_0+C_4K_4|V^*|+K_3)\leq \frac{\phi_* C_4}{\beta (1+\beta)(2c_0+1)^\beta}.
	\end{align}
	Therefore \eqref{h'(t)} holds if we first fix $K_1, K_3, K_4$ small so that \eqref{7.27} holds, and then choose $\theta$ large such that 
	\eqref{7.26} is satisfied. 
	
	{\bf Step 3.} We show that  \eqref{7.28} holds when $K_3$ and $K_4$ are chosen suitably small and $\theta$ is large.

	From \eqref{2.1a}, we deduce
	\begin{align*}
	\ol U_t(t,x)=&-(1+\epsilon)[c_0+\delta'(t)]\Phi'(x-\bar h(t))+\epsilon'(t)\Phi(x-\underline h(t))+\rho_t(t,x),
	\end{align*}
	and
	\begin{align*}
	& -(1+\epsilon)c_0\Phi'(x-\bar h(t))\\ 
	=  & (1+\epsilon) \lf[D\circ \int_{- \yy}^{\bar h(t)}  \mathbf{J}(x-y)\circ  \Phi(y-\bar h(t))\rd y -D\circ\Phi(x-\bar h(t))+F(\Phi(x-\bar h(t)))\rr]\\ 
	= &D\circ \int_{-\yy}^{\bar h(t)}  \mathbf{J}(x-y)\circ [\ol U(t,y)-\rho(t,y)]\rd y -D\circ  [\ol U(t,x)-\rho(t,x)]+(1+\epsilon)F(\Phi(x-\bar h(t)))\\
	= &D\circ \int_{g(t+t_0)}^{\bar h(t)}  \mathbf{J}(x-y)\circ \ol U(t,y)\rd y -D\circ \ol U(t,x)+F(\ol U(t,x))\\
	&+D\circ\lf[\rho(t,x)-\int_{-\yy}^{\bar h(t)}\mathbf{J}(x-y)\circ\rho(t,y) \rd y\rr]+(1+\epsilon)F(\Phi(x-\bar h(t)))-F(\ol U(t,x)).
	\end{align*}
	Hence
	\begin{align*}
	\ol U_t(t,x)= &D\circ \int_{g(t+t_0)}^{\bar h(t)}  \mathbf{J}(x-y)\circ \ol U(t,y)\rd y -D\circ \ol U(t,x)+F(\ol U(t,x))\\
	& + A(t,x)
	\end{align*}
	 with
	\begin{align*}
	A(t,x):=	&D\circ\lf[\rho(t,x)-\int_{-\yy}^{\bar h(t)}\mathbf{J}(x-y)\circ\rho(t,y) \rd y\rr]+(1+\epsilon)F(\Phi(x-\bar h(t)))-F(\ol U(t,x))\\
	&-(1+\epsilon)\delta'(t)\Phi'(x-\bar h(t))+\epsilon'(t)\Phi(x-\underline h(t))+\rho_t(t,x).
	\end{align*}
	Therefore to complete this step, it suffices to show that we can choose $K_3, K_4$ and $\theta$  such that $A(t,x)\succeq \mathbf{0}$.
	We will do that for $x\in [\bar h(t)-\tilde{\epsilon},\bar h(t)]$ and for $x\in [g(t_0+t), \bar h(t)-\tilde{\epsilon}]$ separately.
	\smallskip
	
	{\bf Claim 1.} If $\tilde{\epsilon}>0$ in \eqref{7.4a} is sufficiently small and $\theta$ is sufficiently large, then
	\begin{equation}\label{7.27b}\begin{aligned}
	&D\circ\lf[\rho(t,x)-\int_{-\yy}^{\bar h(t)}\mathbf{J}(x-y)\circ\rho(t,y) \rd y\rr]+(1+\epsilon)F(\Phi(x-\bar h(t)))-F(\ol U(t,x))\\
	\succeq & \ \frac{|\td\lambda_1|}4  \rho(t,x)\succ {\bf 0}\ \ \  \mbox{ for }\ \ x\in [\bar h(t)-\tilde{\epsilon},\bar h(t)].
	\end{aligned}
	\end{equation}
	 Since $\td\lambda_1<0$ and $D\circ V^*=V^* \wtd D $,  using \eqref{D} we deduce,  for $x\in [\bar h(t)-\tilde{\epsilon},\bar h(t)]$,
	\begin{align*}
	& D\circ\lf[\rho(t,x)-\int_{-\yy}^{\bar h(t)}\mathbf{J}(x-y)\circ\rho(t,y) \rd y\rr]\\
	= &K_4\epsilon(t) \lf[D\circ V^*-D\circ \int_{-\yy}^{0}\mathbf{J}(x-\bar h(t)-y)\circ \xi(y)V^*\rd y\rr]\\
	\succeq  &K_4\epsilon(t) \lf[D\circ V^*-D\circ \int_{-2\tilde{\epsilon}}^{0}\mathbf{J} (x-\bar h(t)-y)\circ V^*\rd y\rr]\\
	=&K_4 \epsilon(t)\lf[V^*\nabla F(0) - \td\lambda_1 V^* -D\circ \int_{\bar h-x-2\tilde{\epsilon}}^{\bar h(t)-x}\mathbf{J}(y)\circ V^*\rd y\rr]\\
	\succeq &K_4\epsilon(t) \lf[V^*\nabla F(0) - \td\lambda_1 V^* -D\circ \int_{-2\tilde{\epsilon}}^{\tilde{\epsilon}}\mathbf{J}(y)\circ V^*\rd y\rr]\\
	\succeq & K_4\epsilon(t)\lf[ V^* \nabla F(0) -\frac{\td \lambda_1 }{2}V^*\rr] =\rho(t,x) \nabla F(0)-\frac{\td \lambda_1 }{2}\rho(t,x),
	\end{align*}
	provided $\tilde{\epsilon}\in (0, \epsilon_1]$ for some small $\epsilon_1>0$. 
	
	On the other hand,   for $x\in [\bar h(t)-\tilde{\epsilon},\bar h(t)]$, by $\mathbf{(f_2)}$ we obtain
	\begin{align*}
	&(1+\epsilon)F(\Phi(x-\bar h(t)))-F(\ol U(t,x))\\
	\succeq &F((1+\epsilon)\Phi(x-\bar h(t)))-F(\ol U(t,x))\\
	=& F(\ol U(t,x)-\rho(t,x))-F(\ol U(t,x)),
	\end{align*}
	and 
	\begin{align*}
	{\bf 0}\preceq \ol U(t,x)\preceq (1+\epsilon) \Phi(\tilde{\epsilon})+K_4\epsilon V^*\preceq 2\Phi(\tilde{\epsilon})+\theta^{-\beta}V^*,
	\end{align*}
	So the components of $\ol U(t,x)$ and $\rho(t,x)$ are small for   small  $\tilde{\epsilon}$ and large  $\theta$. It follows that
	\begin{align*}
	&F(\ol U(t,x)-\rho(t,x))-F(\ol U(t,x))=-\rho(t,x)[\nabla F(\ol U(t,x))+o(1){\bf I}_m]\\
	&=-\rho(t,x)[\nabla F(0)+o(1){\bf I}_m] \succeq  -\rho(t,x) \nabla F(0) +\frac{\tilde\lambda_1}{4}\rho(t,x) 
	\end{align*}
	 for $x\in [\bar h(t)-\tilde{\epsilon},\bar h(t)]$,
	provided that   $\tilde{\epsilon}$ is small and  $\theta$ is large. Hence, \eqref{7.27b} holds. 
	\medskip
	
	Denote 
	\[
	M:=\dd\max_{1\leq i\leq m}\sup_{x\leq 0}|\phi_i'(x)|. 
	\]
	 For $x\in [\bar h-\tilde{\epsilon},\bar h]$, by \eqref{7.27b} we have
	\begin{align*}
	A(t,x)\succeq  &\frac{|\tilde\lambda_1|}4 \rho(t,x) -(1+\epsilon)\delta'(t)\Phi'(x-\bar h(t))+\epsilon'(t)\Phi(x-\underline h(t))+\rho_t(t,x)
	\\
	\succeq &\epsilon(t)\bigg[\frac{|\tilde\lambda_1|}4 K_4V^*-2K_3M\mathbf{1} -\beta(t+\theta)^{-1}\mathbf{u}^*-K_4\beta(t+\theta)^{-1}V^*\bigg]\\
	\succeq  &\epsilon(t)\bigg[\frac{|\tilde\lambda_1|}4 K_4V^*-2K_3M{\bf 1} -\theta^{-1}\beta\Big(\mathbf{u^*}+K_4 V^*\Big)\bigg]\\
	\succeq &\ {\bf 0}
	\end{align*}
	provided that we first fix $K_3$ and $K_4$ so that 
	\eqref{7.27} holds and at the same time
	\begin{equation}\label{7.27-1}
	\frac{|\tilde\lambda_1|}4 K_4V^*-2K_3M{\bf 1}\ggs{\bf 0},
	\end{equation}
	and then choose $\theta$ sufficiently large.

	\medskip
	
	Next, for fixed small $\tilde{\epsilon}>0$, we  estimate $A(t,x)$	for $x\in [g(t+t_0), \bar h(t)-\tilde{\epsilon}]$. 
	
{\bf Claim 2.} For any given  $1\gg\eta>0$, there is $c_1=c_1(\eta)$ such that 
\begin{align}\label{01}
(1+\epsilon)F(v)-F((1+\epsilon)v)\succeq  c_1\epsilon\mathbf{1} \ \mbox{ for } \ v\in  [\eta\mathbf{1},\mathbf{u^*}]\ {\rm and}\ 0<\epsilon\ll 1.
\end{align} 

Indeed, by  \eqref{1.7a} there exists $c_1>0$ depending on $\eta$ such that
\[
F(v)-v[\nabla F(v)]^T\succeq 2c_1 \mathbf{1}\  \mbox{ for } \ v\in [\eta\mathbf{1},\mathbf{u^*}].
\]
Since
\begin{align*}
\lim_{\epsilon\to 0}\frac{ (1+\epsilon)F(v)-F((1+\epsilon)v)}{\epsilon}
=&\lim_{\epsilon\to 0}\frac{ \epsilon F(v)-[F(v+\epsilon v)-F(v)]}{\epsilon}\\
=&F(v)-v [\nabla F(v)]^T\succeq 2c_1\mathbf{1} 
\end{align*}
uniformly for $v\in [\eta\mathbf{1},\mathbf{u^*}]$, there exists $\epsilon_0>0$ small so that
\[
\frac{ (1+\epsilon)F(v)-F((1+\epsilon)v)}{\epsilon}\succeq c_1\mathbf{1}
\]
for $v\in [\eta\mathbf{1},\mathbf{u^*}]$ and $\epsilon\in (0, \epsilon_0]$.
 This proves Claim 2. 

\smallskip

By Claim 2 and the Lipschitz continuity of $F$, there exist positive constants $C_l$ and $C_f$ such that,   
for $v=\Phi(x-\bar h(t))\in [\Phi(-\tilde{\epsilon}),\mathbf{u}^*]$, 
	\begin{align*}
	&(1+\epsilon)F(v)-F((1+\epsilon)v+\rho)\\
	=&(1+\epsilon)F(v)-F((1+\epsilon)v)+F((1+\epsilon)v)-F((1+\epsilon)v+\rho)\\
	\succeq  &C_l\epsilon \mathbf{1}-C_fK_4\epsilon \mathbf{1}
	\end{align*}
	when $\epsilon=\epsilon(t)$ is small.
	
	We also have
	\begin{align*}
	&D\circ\lf[\rho(t,x)-\int_{-\yy}^{\bar h(t)}\mathbf{J}(x-y)\circ\rho(t,x) \rd y\rr]\succeq  -D\circ \int_{-\yy}^{\bar h(t)}\mathbf{J}(x-y)\circ\rho(t,x) \rd y\\
	\succeq & -K_4\epsilon (t) D\circ V^*\succeq  -C_dK_4\epsilon(t)\mathbf{1}
	\end{align*}
	for some $C_d>0$,   and 
	\begin{align*}
	\rho_t(t,x)=&-\xi'\bar h' K_4\epsilon(t)V^*+\xi K_4\epsilon'(t)V^*\\
	\succeq &  -\xi_*K_4\epsilon(t)V^*- K_4\beta(t+\theta)^{-1}\epsilon(t)V^*\\
	\succeq &-(\xi_* +\beta\theta^{-1})K_4\epsilon(t)V^*,
	\end{align*}
	with $\xi_*:=c_0\max_{x\in \R}|\xi'(x)|$. 
	
	Using these we obtain, for $x\in [g(t_0+t), \bar h(t)-\tilde{\epsilon}]$,
	\begin{align*}
	A(t,x)\succeq  	&-C_dK_4\epsilon(t)\mathbf{1}+(1+\epsilon)F(\Phi(x-\bar h(t)))-F(\bar U(t,x))+2M \delta'(t)\mathbf{1} +\epsilon'(t)\mathbf{u}^*+\rho_t(t,x)\\
	\succeq  	&C_l\epsilon(t) \mathbf{1}-(C_f+C_d)K_4\epsilon(t) \mathbf{1}-2MK_3 \epsilon(t) \mathbf{1}-\beta(t+\theta)^{-1}\epsilon(t)\mathbf{u^*}-(\xi_* +\beta\theta^{-1})K_4\epsilon(t)V^*\\
	=&\epsilon(t)\bigg[C_l\mathbf{1}-K_4(C_f+C_d)\mathbf{1}-2MK_3\mathbf{1} -\beta(t+\theta)^{-1}\mathbf{u^*}-(\xi_* +\beta\theta^{-1})K_4V^*\bigg]\\
	\succeq &\epsilon(t)\bigg[C_l\mathbf{1}-K_4(C_f+C_d)\mathbf{1}-2MK_3\mathbf{1}-\xi_*K_4V^*  -\beta\theta^{-1}\big(\mathbf{u^*}+K_4V^*\big)\bigg]\\
	\succeq &{\bf 0}
	\end{align*}
	provided that we first choose $K_3$ and $K_4$ small such that
	\[
	C_l\mathbf{1}-K_4(C_f+C_d)\mathbf{1}-2MK_3\mathbf{1}-\xi_*K_4V^*  \ggs {\bf 0}
	\]
	while keeping both \eqref{7.27} and \eqref{7.27-1} hold, and then choose $\theta>0$ sufficiently large.

	Therefore,  \eqref{7.28} holds when $K_3, K_4$ and $\theta$ are chosen as above. 
	The proof of the lemma is now complete.
\end{proof}

\subsection{The case $\td\lambda_1\geq 0$}
\begin{lemma}\label{lemma5.3}
Suppose that the assumptions in Theorem \ref{prop5.1} are satisfied. 	If $\td\lambda_1\geq 0$, then
\eqref{7.3} still holds.
\end{lemma}
\begin{proof} This is a modification of the proof of Lemma \ref{lemma5.2}. We will use similar notations.
 Let $\beta={\gamma}-2\in (0,1]$, and $(c_0,\Phi)$  be the solution of  \eqref{2.1a}-\eqref{2.2a}. For fixed $\tilde{\epsilon}>0$,  let $\xi\in C^2(\R)$ satisfy
	\begin{align*}
	0\leq \xi(x)\leq 1, \ \ \ \xi(x)=1\ {\rm for}\ |x|<\tilde{\epsilon},\ \xi(x)=0\ {\rm for}\ |x|>2\tilde{\epsilon}.
	\end{align*}
	Define
	\[
	\begin{cases}
	\bar h(t):=c_0 t+\delta(t),  &\ \ \ t\geq 0,\\
	\ol U(t,x):=(1+\epsilon(t)) \Phi\big(x-\bar h(t)-\lambda(t)\big)-\rho(t,x), &\ \ \ t\geq 0,\  x\leq \bar h(t),
	\end{cases}
	\]
	where 
	\begin{align*}
	&\epsilon(t):=K_1(t+\theta)^{-\beta}, \ 	\delta(t):=K_2-K_3\int_{0}^{t}\epsilon(\tau)\rd \tau,
	\\
	&\rho(t,x):=K_4\xi(x-\bar h(t))\epsilon(t)V^*,\ \lambda(t):=K_5\epsilon(t),
	\end{align*}
	and the positive constants $\theta$, $\tilde\epsilon$ and $K_1, K_2, K_3, K_4, K_5$ are to be  determined.

	Let \[
	C_{\tilde\epsilon}:=\dd\min_{1\leq i\leq m}\min_{x\in [-2\tilde{\epsilon},0]}|\phi_i'(x)|.
	\]
	 Then for $x\in [\bar h(t)-2\tilde{\epsilon},\bar h(t)]$ and $i\in \{1,..., m\}$,
	\begin{align*}
	\bar u_i(t,x)\geq &\phi_i\big(-\lambda(t)\big)-\rho_i(t,x)\geq C_{\tilde\epsilon} \lambda(t)-K_4\epsilon(t)v_i^*\\
	\geq& \epsilon(t)(C_{\tilde\epsilon} K_5-K_4v_i^*)> 0
	\end{align*}
	if 
	\begin{equation}\label{K4}
	K_4= C_{\td\epsilon} K_5/(2\max_{1\leq i\leq m} v_i^*),
	\end{equation}
	 which combined with $\xi(x)=0$ for $|x|\geq 2\tilde{\epsilon}$ implies
	\begin{align}\label{7.19a}
\ol U(t,x)\succeq  {\bf 0} \mbox{ for } \ t\geq 0,\  x\leq \bar h(t).
	\end{align} 
	Let $t_0=t_0(\theta)$ and $K_2=K_2(\theta)$  be given by Step 1 in the proof of Lemma \ref{lemma5.2}. 	Then $[g(t_0),h(t_0)]\subset (-\yy,K_2/2)$, and due to $\rho(0,x)=0$ for $x\leq h(t_0)<K_2/2<K_2=\bar h(0)$, we have
	\begin{equation}\label{initial}\begin{aligned}
\ol U(0,x)=&(1+\epsilon(0)) \Phi(x-K_2-\lambda)\succeq  (1+\epsilon(0)) \Phi(-K_2/2)\\
	\succeq & (1+\epsilon(0)/2)\mathbf{u}^*\succeq   U(t_0,x) \ \mbox{ for } \ x\in [g(t_0),h(t_0)].
	\end{aligned}
	\end{equation}

	{\bf Step 1.} We verify that by choosing $K_1, K_3$ and $K_5$ suitably small,
	\begin{equation}\label{h'}
	\bar h'(t)\geq\sum_{i=1}^{m_0} \mu_i\int_{g(t+t_0)}^{\bar h(t)} \int_{\bar h(t)}^{+\yy} J_i(x-y) \bar u_i(t,x)\rd y\rd x \ \mbox{ for all } \ t> 0.
	\end{equation}
	
	By direct  calculations we have 
	\begin{align*}
	&\sum_{i=1}^{m_0} \mu_i\int_{g(t+t_0)}^{\bar h(t)} \int_{\bar h(t)}^{+\yy} J_i(x-y)  \bar u_i(t,x)\rd y\rd x\\
	\leq& \sum_{i=1}^{m_0} \mu_i\int_{g(t+t_0)}^{\bar h(t)} \int_{\bar h(t)}^{+\yy} J_i(x-y) (1+\epsilon)\phi_i(x-\bar h(t)-\lambda(t))\rd y\rd x\\
	=&(1+\epsilon)\sum_{i=1}^{m_0} \mu_i \int_{-\yy}^{0} \int_{0}^{+\yy} J_i(x-y)  \phi_i(x-\lambda(t))\rd y\rd x\\
	&-(1+\epsilon)\sum_{i=1}^{m_0} \mu_i \int_{-\yy}^{g(t+t_0)-\bar h(t)} \int_{0}^{+\yy} J_i(x-y) \phi_i(x-\lambda(t))\rd y\rd x\\
	\leq &(1+\epsilon) c_0 +(1+\epsilon)\sum_{i=1}^{m_0} \mu_i\int_{-\yy}^{0} \int_{0}^{+\yy} J_i(x-y)  [\phi_i(x-\lambda)-\phi_i(x)]\rd y\rd x\\
	&-(1+\epsilon) \sum_{i=1}^{m_0} \mu_i\int_{-\yy}^{g(t+t_0)-\bar h(t)} \int_{0}^{+\yy} J_i(x-y)  \phi_i(x)\rd y\rd x\\
	\end{align*}
	Let $M_1:=\dd\max_{1\leq i\leq m}\sup_{x\leq 0}|\phi_i'(x)|$ and $C_3$ be given by \eqref{7.5}. Then
	\begin{align*}
	(1+\epsilon) \sum_{i=1}^{m_0} \mu_i\int_{-\yy}^{0} \int_{0}^{+\yy} J_i(x-y)  [\phi_i(x-\lambda(t))-\phi_i(x)]\rd y\rd x\leq 2 C_3M_1 \lambda(t).
	\end{align*}
	By \eqref{7.7},
	\begin{align*}
	&\sum_{i=1}^{m_0} \mu_i\int_{-\yy}^{g(t+t_0)-\bar h(t)} \int_{0}^{+\yy} J_i(x-y)  \phi_i(x)\rd y\rd x\\
	\geq &\frac{\phi_*C_4}{\beta (1+\beta)(2c_0+1)^\beta }\lf[t+\frac{(c_0+1)t_0+C_1+K_2}{(2c_0+1)}\rr]^{-\beta}.
	\end{align*}
	Therefore, as in the proof of Lemma \ref{lemma5.2}, for sufficiently large $\theta$ so that
	\begin{align}\label{7.15}
	\theta>\frac{(c_0+1)t_0+C_1+K_2}{(2c_0+1)}
	\end{align}
	holds, we have
	\begin{align*}
	&\sum_{i=1}^{m_0} \mu_i\int_{g(t+t_0)}^{\bar h(t)} \int_{\bar h(t)}^{+\yy} J_i(x-y) \bar u_i(t,x)\rd y\rd x\\
	\leq&   (1+\epsilon) c_0+2 C_3M_1\lambda(t)-\frac{\phi_*C_4}{\beta (1+\beta)(2c_0+1)^\beta }\lf(t+\theta\rr)^{-\beta}\\
	=&c_0+\epsilon(t)\lf[c_0+2 C_3M_1 K_5-\frac{\phi_*C_4}{K_1\beta (1+\beta)(2c_0+1)^\beta }\rr]\\
	\leq& c_0-K_3\epsilon(t)=\bar h'(t)
	\end{align*}
	provided that $K_1, K_3$ and $K_5$ are suitably small so that
	\begin{align}\label{7.16}
	K_1(c_0+2 C_3 M_1 K_5+K_3)\leq \frac{\phi_* C_4}{\beta (1+\beta)(2c_0+1)^\beta}.
	\end{align}

	{\bf Step 2.} We show that by choosing $K_3, K_5$ suitably small and $\theta$ sufficiently large, 
	 for $t>0$, $x\in [g(t+t_0),\bar h(t)]$, 
	\begin{align}\label{7.17}
	\ol U_t(t,x)\succeq  &D\circ  \int_{g(t+t_0)}^{\bar h(t)}  {\bf J}(x-y)\circ \ol U(t,y)\rd y -\ol U(t,x)+F(\ol U(t,x)).
	\end{align}

	Using the definition of $\ol U$, we have
	\begin{align*}
	\ol U_t(t,x)=&-(1+\epsilon)(\bar h'+\lambda')\Phi'(x-\bar h-\lambda)+\epsilon'\Phi(x-\bar h-\lambda)-\rho_t\\
	= &-(1+\epsilon)[c_0+\delta'+\lambda']\Phi'(x-\bar h-\lambda)+\epsilon'\Phi(x-\bar h-\lambda)-\rho_t
	\end{align*}
	and from \eqref{2.1a}, we obtain
	\begin{align*}
	& -(1+\epsilon)c_0\Phi'(x-\bar h-\lambda)\\ 
	= & (1+\epsilon) \lf[D\circ  \int_{- \yy}^{\bar h+\lambda}  {\bf J}(x-y)\circ  \Phi(y-\bar h-\lambda)\rd y 
	       -D\circ \Phi(x-\bar h-\lambda)+F(\Phi(x-\bar h-\lambda))\rr]\\ 
	\succeq & (1+\epsilon) \lf[D\circ  \int_{- \yy}^{\bar h}  {\bf J}(x-y)\circ  \Phi(y-\bar h-\lambda)\rd y 
	    -D\circ \Phi(x-\bar h-\lambda)+F(\Phi(x-\bar h-\lambda))\rr]\\ 
	= &D\circ  \int_{-\yy}^{\bar h}  {\bf J}(x-y)\circ [\ol U(t,y)+\rho]\rd y -D\circ   [\ol U(t,x)+\rho]+(1+\epsilon)F(\Phi(x-\bar h-\lambda))\\
	=&D\circ  \int_{-\yy}^{\bar h(t)}  {\bf J}(x-y)\circ \ol U(t,y)\rd y -D\circ  \ol U(t,x)\\
	&-D\circ \lf[\rho(t,x)-\int_{-\yy}^{\bar h(t)}{\bf J}(x-y)\circ\rho(t,y) \rd y\rr] +(1+\epsilon)F(\Phi(x-\bar h-\lambda))\\
	\succeq &D\circ  \int_{g(t+t_0)}^{\bar h(t)}  {\bf J}(x-y)\circ \ol U(t,y)\rd y -D\circ  \ol U(t,x)+F(\ol U(t,x))\\
	&-D\circ \lf[\rho(t,x)-\int_{-\yy}^{\bar h(t)}{\bf J}(x-y)\circ\rho(t,y) \rd y\rr]+(1+\epsilon)F(\Phi(x-\bar h-\lambda))-F(\ol U(t,x)).
	\end{align*}
	Hence
	\begin{align*}
	\ol U_t(t,x)\succeq  &D\circ  \int_{g(t+t_0)}^{\bar h(t)}  {\bf J}(x-y)\circ \ol U(t,y)\rd y -D\circ  \ol U(t,x)+F(\ol U(t,x))\\
	&+B(t,x)
	\end{align*}
with
	\begin{align*}
	B(t,x):=	&-D\circ \lf[\rho(t,x)-\int_{-\yy}^{\bar h}{\bf J}(x-y)\circ\rho(t,y) \rd y\rr]+(1+\epsilon)F(\Phi(x-\bar h-\lambda))-F(\ol U)\\
	&-(1+\epsilon)(\delta'+\lambda')\Phi'(x-\bar h-\lambda)+\epsilon'\Phi(x-\underline h-\lambda)-\rho_t.
	\end{align*}
	To show \eqref{7.17},	it remains to choose suitable $K_3, K_5$ and $\theta$ such that 
	$B(t,x)\succeq  {\bf 0}$ for $t>0$ and $x\in [g(t+t_0),\bar h(t)]$.
	\smallskip
	
	{\bf Claim:}	There exist small $\td\epsilon_0\in (0,\tilde{\epsilon}/2)$ and some $\td J_0>0$ depending on $\td\epsilon$ but independent of $\td\epsilon_0$,
	such that  
	\begin{equation}\label{7.19}\begin{aligned}
	&-D\circ \lf[\rho(t,x)-\int_{-\yy}^{\bar h}{\bf J}(x-y)\circ\rho(t,y) \rd y\rr]+(1+\epsilon)F(\Phi(x-\bar h-\lambda))-F(\ol U(t,x))\\
	&\succeq  \; \td J_0\, \rho(t,x) \ \mbox{ for } \ x\in [\bar h(t)-\tilde\epsilon_0,\bar h(t)].
	\end{aligned}
	\end{equation}

	  Indeed, for $x\in [\bar h(t)-\tilde\epsilon_0,\bar h(t)]$,
	\begin{align*}
	&D\circ \lf[\rho(t,x)-\int_{-\yy}^{\bar h(t)}{\bf J}(x-y)\circ\rho(t,y) \rd y\rr]\\
	=&K_2\epsilon(t) \lf[ D\circ V^*-D\circ \int_{-\yy}^{\bar h(t)}{\bf J}(x-y)\circ \xi(y-\bar h(t))V^*\rd y\rr]\\
	\preceq& K_2\epsilon(t) \lf[ D\circ V^*- D\circ \int_{\bar h(t)-\tilde{\epsilon}}^{\bar h(t)}{\bf J}(x-y)\circ V^*\rd y\rr]\\
	=&K_2\epsilon(t) \lf[ D\circ V^*- D\circ \int_{\bar h(t)-\tilde{\epsilon}-x}^{\bar h(t)-x} {\bf J}(x-y)\circ V^*\rd y\rr]\\
	\preceq&D\circ \rho \lf[1-\int_{-\tilde{\epsilon}+\tilde\epsilon_0}^{0}{J}(y)\rd y\rr]\preceq D\circ  \rho \lf[1-\int_{-\tilde{\epsilon}/2}^{0}{J}(y)\rd y\rr].
	\end{align*}
	On the other hand,   for $x\in [\bar h(t)-\tilde\epsilon_0,\bar h(t)]$,  we have
	\begin{align*}
	&(1+\epsilon)F(\Phi(x-\bar h-\lambda)-F(\ol U)\\
	\succeq & F((1+\epsilon)\Phi(x-\bar h-\lambda))-F(\ol U)\\
	=& F(\ol U+\rho)-F(\ol U)=\rho\Big([\nabla F(\ol U)]^T+o(1){\bf I}_m\Big)\\
	=  & K_4\epsilon(t) V^*\Big([\nabla F({\bf 0})]^T+o(1){\bf I}_m\Big)\\
	=& K_4\epsilon(t)[V^* \tilde D	+\tilde\lambda_1 V^*+o(1) V^*]\\
	=& K_4\epsilon(t)[ D\circ V^*	+\tilde\lambda_1 V^*+o(1) V^*]\\
	= &D\circ \rho +\tilde\lambda_1\rho+o(1)\rho.
	\end{align*}
	since  both $\overline U(t,x)$ and $\rho(t,x)$ are close to ${\bf 0}$  for $x\in [\bar h(t)-\tilde\epsilon_0,\bar h(t)]$ with $\tilde\epsilon_0$   small.  
	
	Hence,  for such $x$ and $\tilde\epsilon_0$, since $\td\lambda_1\geq 0$,
		\begin{align*}
	&-D\circ\lf[\rho(t,x)-\int_{-\yy}^{\bar h(t)}{J}(x-y)\rho(t,y) \rd y\rr]+(1+\epsilon)F(\Phi(x-\bar h(t)))-F(\ol U(t,x))\\
	\succeq & D \circ\rho \lf[-1+\int_{-\tilde{\epsilon}/2}^{0}{J}(y)\rd y\rr]
	+ D\circ \rho +\tilde\lambda_1\rho+o(1)\rho
	\\
	 \succeq & \tilde J_0\,\rho(t,x), \ \ \ \ \  \mbox{ with }\ \  \tilde J_0:=\dd\frac12 \min_{1\leq i\leq m}{d_i} \int_{-\tilde{\epsilon}/2}^{0}{J}(y)\rd y \mbox{ if }m_0=m.
	\end{align*}
	 This proves	\eqref{7.19} when $m_0=m$. 
	
	If $m_0<m$, we need to modify  $V^*$ in the definition of $\rho$ slightly. In this case, for $\tilde\delta>0$ small we define
	\[
	\tilde V^*:=V^*+\td\delta D=(v_i^*+\td\delta d_i).
	\]
	Since $d_i=0$ for $i=m_0+1,..., m$ and $d_i>0$ for $i=1,..., m_0$, by $\mathbf{(f_1)}$ (iv) we see that
	\[
	W=(w_i):=D[\nabla F({\bf 0})]^T
	\]
	satisfies $w_i>0$ for $i=m_0+1,..., m$. Let us write 
	\[
	\mbox{$W=W^1+W^2=(w_i^1)+(w_i^2)$ with }\begin{cases} w_i^1=0 \mbox{ for } i=m_0+1,..., m,\\ 
	w_i^2=0 \mbox{  for } i=1,..., m_0.
	\end{cases}
	\]
	Then 
	\[
	\tilde V^*\Big([\nabla F({\bf 0})]^T-\tilde D\Big)=\tilde\lambda_1 V^*+\td\delta \wtd W^1+\td\delta W^2\  \mbox{ with }
	\wtd W^1:=W^1-D\tilde D.
	\]
	It is important to observe that the vector $\wtd W^1=(\td w_i^1)$ has its last $m-m_0$ components 0, namely
	$\tilde w^1_i=0$ for $i=m_0+1,..., m$.
	
	Replacing $V^*$ by $\td V^*$ in the definition of $\rho$, we see that the analysis above is not affected, except that,
	 for $\td\epsilon_0>0$ small and $x\in [\bar h(t)-\tilde\epsilon_0,\bar h(t)]$,  
	\begin{align*}
	&(1+\epsilon)F(\Phi(x-\bar h-\lambda)-F(\ol U)\\
	\succeq   & K_4\epsilon(t) \td V^*\Big([\nabla F({\bf 0})]^T+o(1){\bf I}_m\Big)\\
	=& K_4\epsilon(t)\Big([\tilde V^* \tilde D	+\tilde\lambda_1 V^*+o(1) V^*]+\td\delta \wtd W^1+\td \delta W^2\Big)\\
	=& K_4\epsilon(t)\Big( D\circ \td V^*	+\tilde\lambda_1 V^*+o(1) V^*+\td\delta \wtd W^1+\td \delta W^2\Big) \\
	\succeq &D\circ \rho +K_4\epsilon(t)\Big( o(1) V^*+\td\delta \wtd W^1+\td \delta W^2\Big).
	\end{align*}
Hence,  for such $x$ and $\tilde\epsilon_0$, we now have
		\begin{align*}
	&-D\circ\lf[\rho(t,x)-\int_{-\yy}^{\bar h(t)}{J}(x-y)\rho(t,y) \rd y\rr]+(1+\epsilon)F(\Phi(x-\bar h(t)))-F(\ol U(t,x))\\
	\succeq & D \circ\rho \lf[-1+\int_{-\tilde{\epsilon}/2}^{0}{J}(y)\rd y\rr]
	+ D\circ \rho +K_4\epsilon(t)\Big( o(1) V^*+\td\delta \wtd W^1+\td \delta W^2\Big)	\\
	\succeq &K_4\epsilon(t)\Big( \min_{1\leq i\leq m_0} v_i^*\int_{-\tilde{\epsilon}/2}^{0}{J}(y)\rd y D+ o(1) V^*+\td\delta \wtd W^1+\td \delta W^2\Big).	
	\end{align*}
We now fix $\td\delta>0$ small enough such that
\[
\td\delta \wtd W^1\preceq \frac 12 \min_{1\leq i\leq m_0} v_i^*\int_{-\tilde{\epsilon}/2}^{0}{J}(y)\rd y D,
\]
and notice that
\[
\widehat W:=\frac 12 \min_{1\leq i\leq m_0} v_i^*\int_{-\tilde{\epsilon}/2}^{0}{J}(y)\rd y D+\td \delta W^2\ggs{\bf 0}.
\]
Therefore there exists $\td J_0>0$ such that
\[
\frac 12 \widehat W\succeq \td J_0 \tilde V^*.
\]
Then
\begin{align*}
	 &K_4\epsilon(t)\Big( \min_{1\leq i\leq m_0} v_i^*\int_{-\tilde{\epsilon}/2}^{0}{J}(y)\rd y D+ o(1) V^*+\td\delta \wtd W^1+\td \delta W^2\Big)
	 \\
\succeq &K_4\epsilon(t)\Big( \widehat W+o(1)V^*\Big)\succeq K_4\epsilon(t)\frac 12 \widehat W
 \succeq K_4\epsilon(t)\td J_0 \tilde V^*	 =\tilde J_0\rho,
	\end{align*}
	provided that $\td\epsilon_0>0$ is chosen sufficiently small. 
	
	Therefore
	 for $\td\epsilon_0>0$ small and $x\in [\bar h(t)-\tilde\epsilon_0,\bar h(t)]$,  
  we finally have
		\begin{align*}
	&-D\circ\lf[\rho(t,x)-\int_{-\yy}^{\bar h(t)}{J}(x-y)\rho(t,y) \rd y\rr]+(1+\epsilon)F(\Phi(x-\bar h(t)))-F(\ol U(t,x))\\
	&\succeq 	\tilde J_0\,\rho(t,x), \ \mbox{ as desired.}
\end{align*}

With $\tilde\delta>0$ chosen as above, we will from now on denote
\[
\hat V^*:=\begin{cases} V^* &\mbox{ if } m_0=m,\\
\td V^* & \mbox{ if } m_0<m,
\end{cases}
\]
but keep the notation for $\rho$ unchanged.
	
	Clearly
		\begin{align*}
	-\rho_t(t,x)=\beta K_4K_1(t+\theta)^{-\beta-1}\hat V^*\succeq  {\bf 0}.
	\end{align*}
	Recalling
	 $M_1:=\dd\max_{1\leq i\leq m}\sup_{x\leq 0}|\phi_i'(x)|$,
we obtain, for $x\in [\bar h(t)-\tilde\epsilon_0,\bar h(t)]$ and small $\tilde\epsilon_0$, 	
\begin{align*}
	B(t,x)\succeq  &\ \td  J_0K_2\epsilon(t)\hat V^*+2(\delta'(t)+\lambda'(t))M_1\mathbf{1}+\epsilon'(t)\mathbf{u}^*\\
	= &\ \td  J_0K_2\epsilon(t)\hat V^*+2\epsilon(t)(-K_3-K_5\beta(t+\theta)^{-1})M_1\mathbf{1}
	-\beta (t+\theta)^{-1}\epsilon(t)\mathbf{u^*}\\
	\succeq  &\ \epsilon(t)\bigg[\td  J_0K_2\hat V^*-2(K_3+K_5\beta\theta^{-1})M_1{\bf 1} -\beta\theta^{-1}\mathbf{u^*}\bigg]\\
	=&\  \epsilon(t)\bigg[\td  J_0K_2\hat V^*-2K_3M_1{\bf 1}-\theta^{-1}\Big(K_5\beta M_1{\bf 1} +\beta\mathbf{u^*}\Big)\bigg]\\
	\succeq&\  {\bf 0}
	\end{align*}
	provided that $K_3$ is chosen small so that \eqref{7.16} holds,
	\begin{equation}\label{K3}
	\td  J_0K_2\hat V^*-2K_3M_1{\bf 1}\ggs {\bf 0},
	\end{equation}
	and $\theta$ is chosen sufficiently large.\footnote{In fact, by the choice of $K_2=K_2(\theta)$ in \eqref{K2},  for fixed $K_3$, \eqref{K3} always holds for large enough $\theta$.}
	\medskip
	
	We next estimate $B(t,x)$	for $x\in [g(t+t_0), \bar h(t)-\tilde\epsilon_0]$. From Claim 2 in the proof of Lemma \ref{lemma5.2},   
	 and the Lipschitz continuity of $F$, there exist positive constants $C_l=C_l(\td\epsilon_0)$ and $C_f$ such that,   
for $v=\Phi(x-\bar h(t-\lambda(t)))\in [\Phi(-\tilde{\epsilon}_0),\mathbf{u}^*]$, 
	\begin{align*}
	&(1+\epsilon)F(v)-F((1+\epsilon)v-\rho)\\
	=&(1+\epsilon)F(v)-F((1+\epsilon)v)+F((1+\epsilon)v)-F((1+\epsilon)v-\rho)\\
	\succeq  &C_l\epsilon \mathbf{1}-C_f\rho\succeq C_l\epsilon \mathbf{1}-C_fK_4\epsilon \hat V^*
	\end{align*}
	when $\epsilon=\epsilon(t)$ is small. Hence
	\begin{align*}
	&(1+\epsilon)F(\Phi(x-\bar h-\lambda))-F(\bar U)\\
	\succeq  &C_l\epsilon\mathbf{1} -C_fK_4\epsilon \hat V^*  \ \mbox{ for } \ \ \ x\in [g(t+t_0), \bar h(t)-\tilde\epsilon_0], \ 0<\td\epsilon_0\ll 1.
	\end{align*}
	Clearly, 
	\begin{align*}
	-D\circ\lf[\rho(t,x)-\int_{-\yy}^{\bar h(t)}\mathbf{J}(x-y)\circ\rho(t,x) \rd y\rr]
	\succeq-K_4\epsilon(t)D\circ \hat V^*,
	\end{align*}
	 and 
	\begin{align*}
	\rho_t(t,x)=-K_4\xi'\bar h'(t)\epsilon(t)\hat V^*+ K_4\xi \epsilon'(t)\hat V^*\preceq \xi_*K_4\epsilon(t)\hat V^*
	\end{align*}
	with $\xi_*:=c_0\max_{x\in \R}|\xi'(x)|$.  
	
	We thus obtain, for $x\in [g(t+t_0), \bar h(t)-\tilde\epsilon_0]$ and $0<\tilde\epsilon_0\ll 1$,
	\begin{align*}
	B(t,x)\succeq  	&-K_4\epsilon(t)D\circ \hat V^*+(1+\epsilon)F(\phi(x-\bar h))-F(\ol U)+2M_1 (\delta'+\lambda')\mathbf{1} +\epsilon'\mathbf{u}^*-\rho_t\\
	\succeq  	&C_l\epsilon (t){\bf 1}-K_4\epsilon(t) (D\circ \hat V^* +C_f\hat V^*+\xi_*\hat V^*)+2M_1 (-K_3\epsilon(t)+K_5\epsilon'(t))\mathbf{1} +\epsilon'(t)\mathbf{u^*}\\
	\succeq &\epsilon(t)\bigg[C_l{\bf 1}-K_4(D\circ \hat V^*+C_f\hat V^*+\xi_*\hat V^*)-2M_1(K_3+K_5\beta(t+\theta)^{-1}){\bf 1} -\beta(t+\theta)^{-1}\mathbf{u^*}\bigg]\\
	\succeq &\epsilon(t)\bigg[C_l{\bf 1}-K_4\Big(D\circ \hat V^*+C_f\hat V^*+\xi_*\hat V^*\Big)-2M_1K_3{\bf 1}
	-\theta^{-1}\beta \Big(2M_1K_5{\bf 1}+\mathbf{u^*}\Big)\bigg]\\
	\succeq & {\bf 0}
		\end{align*}
	if we choose $K_3$ and $K_5$ small so that \eqref{7.16} and \eqref{K3} hold and at the same time, due to \eqref{K4}
	\[
	C_l{\bf 1}-K_4\Big(D\circ \hat V^*+C_f\hat V^*+\xi_*\hat V^*\Big)-2M_1K_3{\bf 1}\ggs {\bf 0},
	\]
	and then choose $\theta$ sufficiently large.
	 	Hence,  \eqref{7.17} is satisfied  if $K_3$ and $K_5$ are chosen small as above, and $\theta$ is sufficiently large.

	From \eqref{7.19a}, we have
	\begin{align*}
	\ol U(t,g(t+t_0))\succeq {\bf 0},\ \  \ol U(t,\bar h(t))\succeq {\bf 0} \ \mbox{ for } \  t\geq 0.
	\end{align*}
	Together with \eqref{initial}, \eqref{h'} and \eqref{7.17}, this enables us to use the comparison principle  to conclude that
	\begin{align*}
	h(t+t_0)\leq  \bar h(t),\ U(t+t_0,x)\preceq \ol U(t,x) \ \mbox{ for }\  t\geq 0,\ x\in [ g(t+t_0),\underline h(t)],
	\end{align*}
	which implies \eqref{7.3}. The proof of the lemma is now complete.
\end{proof}

\subsection{Proof of Theorem \ref{prop5.1}}

	Since $\mathbf{(J^1)}$ holds, by Lemma \ref{lemma4.3} and then by \eqref{5.1}, there exists $C_0>0$ such that
	\begin{align*}
		h(t)-c_0t\geq &-C\left[1+\int_0^t(1+x)^{-1}dx+ \int_0^{\frac{c_0}2t} x^2\hat J(x)dx +t \int_{\frac{c_0}2t}^{\yy}x \hat J(x) dx\right]\\
		\geq&- C\left[1+\int_0^1\hat J(x)dx+\ln (t+1)+ C_0\int_1^{\frac{c_0}2t} x^{2-\gamma}dx +C_0t \int_{\frac{c_0}2t}^{\yy}x^{1- \gamma}dx\right].
	\end{align*}
	Therefore when  ${\gamma}\in (2,3)$ we have, for  $t\geq 1$,
	\begin{align*}
		h(t)-c_0t\geq &- C\left[\tilde C+\ln (t+1)+ \tilde C_1 t^{3-\gamma}\right]\geq -\hat C_1 t^{3-\gamma}
	\end{align*}
	for some $\hat C_1, \td C, \td C_1>0$, and when ${\gamma}=3$, for $t\geq 1$, 
	\begin{align*}
	h(t)-c_0t\geq  & -\hat C_2 \ln t
	\end{align*}
	for  some $\hat C_2>0$.	This combined with  Lemmas \ref{lemma5.2} and \ref{lemma5.3} 
	gives the desired conclusion of Theorem \ref{prop5.1}. The proof is completed.
\hfill $\Box$

\section{The growth orders of infinite spreading speed}\label{section6}

 Let $(U, g, h)$ be the unique positive solution of \eqref{1.1}, and assume that spreading happens.
Under the assumptions of Theorem \ref{theorem1.3}, we have
\[
-\lim_{t\to\yy}\frac{g(t)}{t}=\lim_{t\to\yy} \frac{h(t)}{t}=\yy. 
\]

Suppose $\mathbf{(\hat J^\gamma)}$ holds for some $\gamma\in (1, 2]$, namely,  for $|x|\gg 1$ we have
\begin{equation}\label{J-6.1}
\dd J_i(x)\approx  |x|^{-\gamma} \mbox{ for  $ i\in\{1,..., m_0\}$ and some $\gamma\in (1, 2]$}.
\end{equation}
 Then 
\begin{align*}
 \int_{\R}J_i(x) \rd x<\yy,\ \  \int_{\R}|x|J_i(x) \rd x=\yy \mbox{ for } i\in \{1,..., m_0\}.
\end{align*}
So \textbf{(J1)} is not satisfied.

\medskip

The purpose of this section is to prove Theorem \ref{th1.5a}, which we restate as
\begin{theorem}\label{prop6.1} 	Assume that $\mathbf{(J)}$ and  $\mathbf{(f_1)-(f_4)}$  are satisfied. If spreading happens, and additionally \eqref{J-6.1} holds,  then  for large $t>0$,
	\[\begin{cases}
 -g(t),\ h(t)\approx\  t^{1/({\gamma}-1)} &  {\rm if}\ {\gamma}\in (1,2),\\
 -g(t),\ h(t)\approx\  t\ln t &  {\rm if}\ {\gamma}=2.
\end{cases}\]
\end{theorem}

 We will only prove the estimate for $h(t)$, since that for $g(t)$ follows by the change of variable $x\to-x$. Theorem \ref{prop6.1} will follow directly from the lemmas in Subsections 6.1 and 6.2 below.

\subsection{Upper bound} To prove the upper bound a slightly weaker condition than \eqref{J-6.1} is enough.
 We assume that there exist positive constants $C_1$ and $C_2$ such that
\begin{align}\label{8.1}
\frac{C_1}{|x|^{{\gamma}}+1} \leq \sum_{i=1}^{m_0} \mu_i J_i(x)\leq \frac{C_2}{|x|^{{\gamma}}+1}\  \mbox{ for } \ x\in \R
\ \mbox{ and some } {\gamma}\in (1,2].
\end{align}
 Obviously,  \eqref{8.1} has no restriction for the kernel function $J_{i_0}$ whenever $\mu_{i_0}=0$,
and \eqref{J-6.1} implies \eqref{8.1} for the same $\gamma$.

\begin{lemma}
Assume that $\mathbf{(J)}$ and  $\mathbf{(f_1)-(f_4)}$  hold. If spreading happens, and \eqref{8.1} is satisfied, then there exits $C=C({\gamma})>0$ such that 
	\begin{equation}\label{8.2a}
	\begin{cases}
	h(t)\leq Ct^{1/({\gamma}-1)}&{\rm if}\ {\gamma}\in (1,2),\\
	h(t)\leq Ct\ln t&{\rm if}\ {\gamma}=2.
	\end{cases}
	\end{equation}
\end{lemma}
\begin{proof}
	Define, for $t\geq 0$,
	\begin{equation*}
	\bar h(t):=\begin{cases}
	(Kt+\theta)^{1/({\gamma}-1)}&{\rm if}\ {\gamma}\in (1,2],\\
	(Kt+\theta)\ln(Kt+\theta)&{\rm if}\ {\gamma}=2,\\
	\end{cases}
	\end{equation*}
	and
	\begin{align*}
	\ol U(t,x):=\bar u {\bf 1},\; \bar u:=\max_{1\leq i\leq m}\lf\{\|u_{i0}\|_\infty, u_i^*\rr\},\ \ \ x\in [-\bar h(t), \bar h(t)],
	\end{align*}
	with positive constants $\theta$ and $K$ to be determined.

	We start by showing
	\begin{align}\label{8.2}
	&\bar h'(t)\geq \sum_{i=1}^{m_0}\mu_i \int_{-\bar h(t)}^{\bar h(t)} \int_{\bar h(t)}^{+\yy} J_i(x-y) \bar u_i(t,x)\rd y\rd x \ \ \mbox{ for }\ \ t> 0,
	\end{align}
	and 
	\begin{align*}
	&-\bar h'(t)\leq - \sum_{i=1}^{m_0}\mu_i \int_{-\bar h(t)}^{\bar h(t)} \int_{-\yy}^{-\bar h(t)} J_i(x-y) \bar u_i(t,x)\rd y\rd x  \ \mbox{ for } \ t> 0.
	\end{align*}
	Since $\ol U(t,x)=\ol U(t,-x)$ and $J_i(x)=J_i(-x)$, it suffices to prove \eqref{8.2}.

	By simple calculations and \eqref{8.1}, for any $k>1$,
	\begin{align*}
	& \sum_{i=1}^{m_0}\mu_i \int_{-k}^{0}\int_{0}^\yy J_i(x-y)\rd y\rd x= \sum_{i=1}^{m_0}\mu_i \int_{0}^{k}\int_{x}^{\yy} J_i(y)\rd y\rd x\\
	=& \sum_{i=1}^{m_0}\mu_i \int_{0}^kJ_i(y)y\rd y+ \sum_{i=1}^{m_0}\mu_i k\int_{k}^\yy J_i(y)\rd y\\
	\leq&\int_{0}^k\frac{C_2y}{y^{{\gamma}}+1}\rd y+k\int_{k}^\yy \frac{C_2}{y^{{\gamma}}+1}\rd y\leq \int_{0}^1C_2\rd y+\int_{1}^k\frac{C_2y}{y^{{\gamma}}}\rd y+k\int_{k}^\yy \frac{C_2}{y^{{\gamma}}}\rd y,
	\end{align*}
	and so
	\begin{equation}\label{8.3}
	\begin{cases}
	\dd \sum_{i=1}^{m_0}\mu_i \int_{-k}^{0}\int_{0}^\yy J_i(x-y)\rd y\rd x\leq C_2+\frac{C_2}{2-{\gamma}}(k^{2-{\gamma}}-1)+\frac{C_2k^{2-{\gamma}}}{{\gamma}-1}&{\rm if}\ {\gamma}\in (1,2),\\[5mm]
	\dd\sum_{i=1}^{m_0}\mu_i \int_{-k}^{0}\int_{0}^\yy J_i(x-y)\rd y\rd x\leq 2C_2+C_2\ln k& {\rm if}\ {\gamma}=2.
	\end{cases}
	\end{equation}
	A direct calculation gives 
	\begin{align*}
	\int_{-\bar h(t)}^{\bar h(t)} \int_{\bar h(t)}^{+\yy} J_i(x-y) \bar u_i(t,x)\rd y\rd x= \bar u \int_{-2\bar h(t)}^{0} \int_{0}^{+\yy} J_i(x-y) \rd y\rd x.
	\end{align*}
	Hence for $1<{\gamma}<2$,  by \eqref{8.3},
	\begin{align*}
	&\sum_{i=1}^{m_0}\mu_i \int_{-\bar h(t)}^{\bar h(t)} \int_{\bar h(t)}^{+\yy} J_i(x-y) \bar u_i(t,x)\rd y\rd x\\
	\leq& \bar u\lf[C_2+2^{2-{\gamma}}\lf(\frac{C_2}{2-{\gamma}}+\frac{C_2}{{\gamma}-1}\rr)(Kt+\theta)^{(2-{\gamma})/({\gamma}-1)}\rr]\\
	\leq& \frac{K}{{\gamma}-1} (Kt+\theta)^{(2-{\gamma})/(1-{\gamma})}=\bar h'(t)
	\end{align*}
	provided  that $K>0$ is large enough. And for   ${\gamma}=2$,
	\begin{align*}
	&\sum_{i=1}^{m_0}\mu_i \int_{-\bar h(t)}^{\bar h(t)} \int_{\bar h(t)}^{+\yy} J_i(x-y) \bar u_i(t,x)\rd y\rd x\\
	\leq&  \bar u\big(2C_2+C_2\ln [2(Kt+\theta)\ln(Kt+\theta)]\big)\\
	\leq& \bar u\big(2C_2+C_2\ln 2(Kt+\theta)+C_2\ln [\ln(Kt+\theta)]\big)\\
	\leq& K\ln (Kt+\theta)+K=\bar h'(t)
	\end{align*}
	if $K\gg 1$.	This finishes the proof of \eqref{8.2}.
	
	Since $\ol U\geq \mathbf{u}^*$ is a constant vector,  we have, 
	for $t>0$, $x\in [-\bar h(t),\bar h(t)]$, 
	\begin{align}
	\ol U_t(t,x)\equiv {\bf 0}\succeq   D\circ \int_{-\bar h(t)}^{\bar h(t)}  \mathbf{J}(x-y) \circ\ol U(t,y)\rd y -D\circ\ol U(t,x)+F(\ol U(t,x)).
	\end{align}
	Moreover, $\bar h(0)\geq h_0$ for large $\theta$, and obviously
	\begin{align*}
	&\ol U(t,\pm \bar h(t))\succeq {\bf 0} \mbox{ for }  t\geq 0,\\
	&\ol U(0, x)\succeq  U(0,x) \mbox{ for }  x\in [-h_0,h_0].
	\end{align*}
	 	Hence we can apply the comparison principle (Lemma \ref{lemma3.2})  to conclude that
	\begin{align*}
	& [g(t+t_0),h(t+t_0)]\subset  [-\bar h(t),\bar h(t)], \ &&t\geq 0,\\
	&U(t+t_0,x)\preceq \ol U(t,x),&&t\geq 0,\ x\in [ g(t+t_0), h(t+t_0)].
	\end{align*}
	Thus \eqref{8.2a} holds. 
\end{proof}

\subsection{Lower bound} The lower bound is more difficult to obtain, and we will consider the cases $\gamma\in (1,2)$ and $\gamma=2$ separately.

\subsubsection{The case  ${\gamma}\in (1,2)$}
We start with a result from \cite{dn}.
 \begin{lemma}{\rm \cite[ (2.11)]{dn}}
	If $\td J$ satisfies {\rm \textbf{(J)}}, then for any $\epsilon>0$, there is $L_\epsilon>0$ such that for all $l>L_\epsilon$ and
	$\psi_l(x):=l-|x|$,
	\begin{align}\label{7.9a}
	\int_{-l}^{l} \td J(x-y)\psi_l(y)\rd y\geq (1-\epsilon) \psi_l(x)  \mbox{ in } [-l, l].
	\end{align}
\end{lemma}

\begin{lemma}\label{lemma7.3}
Assume that the conditions in Theorem \ref{prop6.1} are satisfied and ${\gamma}\in (1,2)$. Then there exits $C=C({\gamma})>0$ such that 
	\begin{align}\label{7.9}
	h(t)\geq Ct^{1/({\gamma}-1)} \mbox{ for } t\gg 1.
	\end{align}
\end{lemma}
\begin{proof}
	Define 
	\begin{align*}
	&\underline h(t):=(K_1t+\theta)^{1/({\gamma}-1)},\ \ t\geq 0,\\
	&	\underline U(t,x):=K_2\frac{\underline h(t)-|x|}{\underline h(t)}{\Theta}, \ \ \ t\geq 0,\ x\in [-\underline h(t), \underline h(t)],
	\end{align*}
	with positive  constants $\theta$ and $K_1, K_2$ to be determined, where the vector $\Theta=(\theta_i)$ is given by Lemma \ref{lemma2.1a}.

	{\bf Step 1.}	 We show that, for large $K_1$,
	\begin{align}\label{7.10}
	&\underline h'(t)\leq  \mu_i \int_{-\underline h(t)}^{\underline h(t)} \int_{\underline h(t)}^{+\yy} J(x-y) \underline u_i(t,x)\rd y\rd x \ \mbox{ for } \  t> 0.
	\end{align}
		
	By simple calculations and \eqref{8.1}, we obtain
	\begin{align*}
	&\sum_{i=1}^{m_0}\mu_i \int_{-\underline h(t)}^{\underline h(t)} \int_{\underline h(t)}^{+\yy} J_i(x-y) \underline u_i(t,x)\rd y\rd x\\
	\geq& \sum_{i=1}^{m_0}\mu_i K_2\theta_i \int_{0}^{\underline h(t)} \int_{\underline h(t)}^{+\yy} J_i(x-y) \frac{\underline h(t)-x}{\underline h(t)}\rd y\rd x\\
	=&\sum_{i=1}^{m_0} \frac{\mu_i  K_2\theta_i}{\underline h(t)} \int_{-\underline h(t)}^{0} \int_{0}^{+\yy} J_i(x-y) (-x)\rd y\rd x\\
	=&\sum_{i=1}^{m_0} \frac{\mu_i  K_2\theta_i}{\underline h(t)} \int_0^{\underline h(t)} \int_{x}^{+\yy} J_i(y) x\rd y\rd x\\
	=&\sum_{i=1}^{m_0} \frac{\mu_i  K_2\theta_i}{\underline h(t)} \lf(\int_{0}^{\underline h(t)}\int_{0}^y+\int_{\underline h(t)}^{\yy}\int_{0}^{\underline h}\rr) J_i(y) x\rd x\rd y\\
	\geq &\sum_{i=1}^{m_0}\mu_i\theta_i\frac{  K_2}{2\underline h(t)}  \int_{0}^{\underline h(t)}J_i(y)y^2\rd y\geq \sum_{i=1}^{m_0}\mu_i\theta_i\frac{  K_2C_1}{2\underline h(t)}  \int_{0}^{\underline h(t)}\frac{y^2}{y^{\gamma}+1}\rd y\\
	\geq& \sum_{i=1}^{m_0}\mu_i\theta_i \frac{  K_2C_1}{4\underline h(t)}  \int_{1}^{\underline h(t)}y^{2-\gamma}\rd y\geq 
	  \sum_{i=1}^{m_0}\mu_i\theta_i \frac{  K_2C_1}{4\underline h(t)}\frac{\underline h(t)^{3-\gamma} }{3-\gamma}\\
	=&\hat C_0(K_1t+\theta)^{(2-{\gamma})/({\gamma}-1)}\geq  \frac{K_1}{{\gamma}-1}(K_1t+\theta)^{(2-{\gamma})/({\gamma}-1)}=\underline h'(t)
	\end{align*}
	provided that $K_1\geq \hat C_0(\gamma-1)$. This  finishes the proof of Step 1.

	{\bf Step 2.} We show that	, by choosing $K_1, K_2$ and $\theta$ properly, for $t>0$, $x\in (-\underline h(t),\underline h(t))$, 
	\begin{align}\label{U_t}
	\underline U_t(t,x)\succeq  &D \int_{-\underline h(t)}^{\underline h(t)}  \mathbf{J}(x-y) \circ\underline U(t,y)\rd y -D\circ\underline U(t,x)+F(\underline U(t,x)).
	\end{align}
	
	From the definition of $\underline U$, for $t>0$, $x\in (-\underline h(t),\underline h(t))$, 
	\begin{align*}
	\underline U_t(t,x)=&K_2\Theta\frac{|x|\underline h'(t)}{\underline h^2(t)}\preceq K_2\Theta \frac{\underline h'(t)}{\underline h(t)}=   \frac{K_1K_2\Theta}{{\gamma}-1}\underline h(t)^{1-\gamma}.
	\end{align*}
	
	{\bf Claim 1}. For $x\in [-\underline h(t),\underline h(t)]$, there exists a positive constant $\hat C_1$ depending only on $\gamma$  such that
	\begin{align}\label{7.12}
	\int_{-\underline h(t)}^{\underline h(t)}  \mathbf{J}(x-y)\circ \underline U(t,y)\rd y\succeq \hat C_1 K_2\Theta  \underline h(t)^{1-\gamma}.
	\end{align}
	By \eqref{J-6.1}, there exists $\td C_1>0$ such that
	\begin{equation}\label{J-8.1}
	J_i(x)\geq \frac{\td C_1}{|x|^\gamma+1} \mbox{ for } x\in\R,\; i=1,..., m_0.
	\end{equation}
Hence
	\begin{align*}
	&\int_{-\underline h}^{\underline h}  \mathbf{J}(x-y)\circ \underline U(t,y)\rd y=\int_{-\underline h-x}^{\underline h-x}  \mathbf{J}(y)\circ \underline U(t,y+x)\rd y\\
	\succeq  &K_2\Theta\int_{-\underline h-x}^{\underline h-x}  \frac{\td C_1}{|y|^{\gamma}+1} \frac{\underline h-|y+x|}{\underline h}\rd y.
	\end{align*}
	Thus, for $x\in [\underline h/4,\underline h]$, 
	\begin{align*}
	&\int_{-\underline h}^{\underline h}  \mathbf{J}(x-y)\circ \underline U(t,y)\rd y\succeq  K_2\Theta\int_{-\underline h/4}^{0}
	  \frac{\tilde C_1}{|y|^{\gamma}+1} \frac{\underline h-|y+x|}{\underline h}\rd y\\
	=&K_2\Theta\int_{-\underline h/4}^{0}  \frac{\tilde C_1}{|y|^{\gamma}+1} \frac{h-(y+x)}{h}\rd y\succeq  K_2\Theta\int_{-\underline h/4}^{0}  \frac{\tilde C_1}{|y|^{\gamma}+1} \frac{-y}{h}\rd y\\
	=&\frac{K_2\Theta}{\underline h}\int_{0}^{\underline h/4}  \frac{\tilde C_1y}{y^{\gamma}+1} \rd y\succeq  \frac{\td C_1K_2\Theta}{2\underline h}\int_{1}^{\underline h/4}  y^{1-\gamma} \rd y\\
	\succeq &\frac{\td C_1K_2\Theta}{2(2-{\gamma})\underline h}(\underline h/4)^{2-{\gamma}}=\hat C_1 K_2\Theta  \underline h^{1-\gamma}.
	\end{align*}
	And for $x\in [0,\underline h/4]$, 
	\begin{align*}
	&\int_{-\underline h}^{\underline h}  \mathbf{J}(x-y)\circ \underline U(t,y)\rd y\succeq  K_2\Theta\int_{0}^{\underline h/4} \frac{\td C_1}{|y|^{\gamma}+1} \frac{\underline h-|y+x|}{\underline h}\rd y\\
	\succeq & K_2\Theta\int_{0}^{\underline h/4} \frac{\td C_1}{y^{\gamma}+1} \frac{y}{\underline h}\rd y\succeq \hat C_1 K_2\Theta  \underline h^{1-\gamma}
	\end{align*}
	by repeating the last a few steps in the previous calculations.
	
	This proves \eqref{7.12} for $x\in [0,\underline h]$.  \eqref{7.12}  also holds for $x\in [-\underline h,0]$ since both $J(x)$ and $\underline U(t,x)$ are even in $x$.

	{\bf Claim 2}. We can choose small $K_2$ and large $\theta$ such that, for $x\in [-\underline h(t),\underline h(t)]$ and $t\geq 0$,
	\begin{align*}
	D \circ\int_{-\underline h}^{\underline h}  \mathbf{J}(x-y)\circ \underline U(t,y)\rd y -D\circ \underline U(t,x)+F(\underline U(t,x))\succeq  F_* \int_{-\underline h}^{\underline h}  \mathbf{J}(x-y)\circ \underline U(t,y)\rd y
	\end{align*}
	for some positive  constant $F_*$. It is clear that  $\underline U\leq K_2\Theta$, and thus for small $K_2>0$ from the definition of $\Theta$ in Lemma \ref{lemma2.1a}, 
	\begin{align*}
	F(\underline U(t,x))=K_2 \frac{\underline h(t)-|x|}{\underline h(t)}\Theta \Big([\nabla F({\bf 0})]^T+o(1){\bf I}_m\Big)
	\succeq K_2\frac{\underline h(t)-|x|}{\underline h(t)}\frac{3}{4}\lambda_1\Theta=\frac{3}{4}\lambda_1 \underline U(t,x),
	\end{align*}
	where $\lambda_1>0$ is given in Lemma \ref{lemma2.1a}. 
	Moreover, by \eqref{7.9a}, there is $L_1>0$ such that  for $ \theta^{1/({\gamma}-1)}\geq L_1$,
	\begin{align*}
	D \circ\int_{-\underline h(t)}^{\underline h(t)}  \mathbf{J}(x-y)\circ \underline U(t,y)\rd y +\frac{\lambda_1}{4} \underline  U(t,x)\succeq  D\circ\underline U(t,x) \ \mbox{ for }\ x\in [-\underline h(t),\underline h(t)].
	\end{align*}
	Therefore Claim 2 is valid with $F_*=\lambda_1/2$.
	
	Combining Claim 1 and Claim 2, we obtain
	\begin{align*}
	&D\circ \int_{-\underline h}^{\underline h}  \mathbf{J}(x-y)\circ \underline U(t,y)\rd y -D\circ\underline U(t,x)+F(\underline U(t,x))\\
	&\succeq  F_*\hat C_1 K_2\Theta  \underline h(t)^{1-\gamma}
	\succeq  \frac{K_1K_2\Theta}{{\gamma}-1}\underline h(t)^{1-\gamma}
	\succeq  \underline U_t(t,x)
	\end{align*}
	provided that
	\begin{align*}
	K_1\leq F_*\hat C_1({\gamma}-1).
	\end{align*}
	This proves \eqref{U_t}.
	
	{\bf Step 3.} We prove \eqref{7.9} by the comparison principle.
	
	It is clear that 
	\begin{align*}
	\underline U(t,\pm \underline h(t))=0 \ \mbox{ for } \ t\geq 0. 
	\end{align*}
	Since spreading happens for $(U,g,h)$, for fixed $\theta$ and small $K_1, K_2$ as chosen above,  there exists a large  $t_0>0$ such that 
	\begin{align*}
	&[-\underline h(0),\underline h(0)]\subset [g(t_0)/2,h(t_0)/2],\\
	&U(t_0,x)\succeq   K_2\Theta \succeq  \underline U(0,x) \ \mbox{ for }\ \ x\in [-\underline h(0),\underline h(0)].
	\end{align*}
	Moreover, since $J(x)$ and $\underline U(t,x)$ are both even in $x$, \eqref{7.10} implies
	\begin{align*}
	&-\underline h'(t)\geq  \mu_i \int_{-\underline h(t)}^{\underline h(t)} \int^{-\underline h(t)}_{-\yy} J(x-y) \underline u_i(t,x)\rd y\rd x \ \mbox{ for } \  t> 0.
	\end{align*}
	These combined with the estimates in  Step 1 and Step 2 allow us to apply Lemma \ref{lemma3.2} 	to conclude that 
	\begin{align*}
	&[-\underline h(t),\underline h(t)]\subset [g(t+t_0),h(t+t_0)], \ &&t\geq 0,\\
	&\underline U(t,x)\preceq U(t+t_0,x),&&t\geq   0,\ x\in [-\underline h(t),\underline h(t)].
	\end{align*}
	Hence \eqref{7.9} holds. 
\end{proof}

\subsubsection{The case  ${\gamma}=2$}
 The following simple result will play an important role in our analysis later.
\begin{lemma}\label{lemma7.4}
	Let $l_1$ and $l_2$ with $0<l_1<l_2$ be two constants, and define
	\begin{align*}
	\psi(x)=\psi(x;l_1,l_2):=\min\lf\{1,\frac{l_2-|x|}{l_1}\rr\},\ \ \ \ x\in \R.
	\end{align*}
	If $\td J$ satisfies {\bf (J)}, then for any $\epsilon>0$, there is $L_\epsilon>0$ such that for all $l_1>L_\epsilon$ and $l_2-l_1>L_\epsilon$,
	\begin{align}\label{7.11a}
	\int_{-l_2}^{l_2} \td J(x-y)\psi(y)\rd y\geq (1-\epsilon) \psi(x) \mbox{ in } [-l_2, l_2].
	\end{align}
\end{lemma}
\begin{proof}
	Since $\int_{\mathbb{R}}\td J(x)\rd x=1$, there exits $B>0$ such that
	\begin{align}\label{7.12a}
	\int_{-B}^{B}\td J(x)\rd x>1-\epsilon/2.
	\end{align}
	In the following discussion we always assume that $l_1\gg B$ and $l_2-l_1\gg B$. Clearly, for $x\in [-(l_2-l_1)+B,(l_2-l_1)-B]$, due to
	\begin{align*}
	\psi(x)=1\ \mbox{ in } [-(l_2-l_1), l_2-l_1],
	\end{align*}
	we have 
	\begin{align*}
	&\int_{-l_2}^{l_2} \td J(x-y)\psi(y)\rd y\geq \int_{-(l_2-l_1)}^{l_2-l_1} \td J(x-y)\psi(y)\rd y=\int_{-(l_2-l_1)}^{l_2-l_1} \td J(x-y)\rd y\\
	=&\int_{-(l_2-l_1)-x}^{l_2-l_1-x} \td J(y)\rd y\geq \int_{-B}^{B} \td J(y)\rd y\geq 1-\epsilon/2> (1-\epsilon) \psi(x). 
	\end{align*}
	
	It remain to prove \eqref{7.11a} for $x\in [-l_2,-(l_2-l_1)+B]\cup [(l_2-l_1)-B,l_2]$. By the symmetric property of $\psi(x)$ and $\td J(x)$ with respect to $x$, we just need to verify \eqref{7.11a} for $x\in [(l_2-l_1)-B,l_2]$, which will be carried out according to the following three cases:
	\begin{align*}
	{\rm (i)}\ x\in [l_2-l_1-B, l_2-l_1+B],\ \ {\rm (ii)}\ x\in [ l_2-l_1+B,l_2-B],\ \ {\rm (iii)}\ x\in [l_2-B,l_2].
	\end{align*}
	
	(i) For $x\in [l_2-l_1-B, l_2-l_1+B]$, since $\psi(z)$ is nonincreasing for $z\geq 0$, 
	we have
	\begin{align*}
	&\int_{-l_2}^{l_2} \td J(x-y)\psi(y)\rd y= \int_{-l_2-x}^{l_2-x} \td J(y)\psi(y+x)\rd y\\
	\geq& \int_{-2l_2+l_1+B}^{B} \td J(y)\psi(y+x)\rd y\geq \int_{-B}^{B}  \td J(y)\psi(y+x)\rd y\\
	\geq& \int_{-B}^{B}  \td J(y)\psi(y+l_2-l_1+B)\rd y.
	\end{align*} 
	By the definition of $\psi$, 
	\begin{align*}
	&\psi(y+l_2-l_1+B)=\frac{l_2-(y+l_2-l_1+B)}{l_1}=1-\frac{y+B}{l_1},\ \ \ y\in [-B,B].
	\end{align*}
	Hence,
	\begin{align*}
	&\int_{-B}^{B}  \td J(y)\psi(y+l_2-l_1+B)\rd y=\int_{-B}^{B}  \td J(y)\rd y-\int_{-B}^{B}  \td J(y)\frac{y+B}{l_1}\rd y\\
	\geq& 1-\epsilon/2-\|\td J\|_{L^\yy(\R)}\frac{2B^2}{l_1}\geq 1-\epsilon\geq (1-\epsilon)\psi(x)
	\end{align*}
	provided
	\begin{align*}
	l_1\geq \frac{4\|\td J\|_{L^\yy(\R)}B^2}{\epsilon},
	\end{align*}
	which then gives
	\begin{align*}
	\int_{-l_2}^{l_2} \td J(x-y)\psi(y)\rd y\geq (1-\epsilon)\psi(x)\ \mbox{ for } \ x\in [l_2-l_1-B, l_2-l_1+B]. 
	\end{align*}

	(ii) For $x\in [ l_2-l_1+B,l_2-B]$, 
	\begin{align*}
	&\int_{-l_2}^{l_2} \td J(x-y)\psi(y)\rd y= \int_{-l_2-x}^{l_2-x} \td J(y)\psi(y+x)\rd y\\
	\geq&\int_{-2l_2-B+l_1}^{B} \td J(y)\psi(y+x)\rd y\geq \int_{-B}^{B} \td J(y)\psi(y+x)\rd y.
	\end{align*}
	From the definition of $\psi$, for $x\in [ l_2-l_1+B,l_2-B]$ and $y\in [-B,B]$, 
	\begin{align*}
	\psi(y+x)=\frac{l_2-(y+x)}{l_1}=\frac{l_2-x}{l_1}-\frac{y}{l_1}=\psi(x)-\frac{y}{l_1}.
	\end{align*}
	Thus, by \eqref{7.12a}, 
	\begin{align*}
	&\int_{-l_2}^{l_2} \td J(x-y)\psi(y)\rd y\geq \int_{-B}^{B} \td J(y)\psi(y+x)\rd y\\
	=&\psi(x)\int_{-B}^{B} \td J(y)\rd y-\int_{-B}^{B} \td J(y)\frac{y}{l_1}\rd y=\psi(x)\int_{-B}^{B} \td J(y)\rd y\geq (1-\epsilon) \psi(x). 
	\end{align*}
	
	(iii) For $x\in [l_2-B,l_2]$, 
	\begin{align*}
	&\int_{-l_2}^{l_2} \td J(x-y)\psi(y)\rd y= \int_{-l_2-x}^{l_2-x} \td J(y)\psi(y+x)\rd y\\
	\geq&\int_{-2l_2-B+l_1}^{l_2-x} \td J(y)\psi(y+x)\rd y\geq \int_{-B}^{l_2-x} \td J(y)\psi(y+x)\rd y\\
	=&\int_{-B}^{B} \td J(y)\psi(y+x)\rd y-\int_{l_2-x}^{B} \td J(y)\psi(y+x)\rd y
	\end{align*}
	As in (ii), we see that 
	\begin{align*}
	\int_{-B}^{B} \td J(y)\psi(y+x)\rd y=\psi(x)\int_{-B}^{B} \td J(y)\rd y\geq (1-\epsilon) \psi(x).
	\end{align*}
	By the definition of $\psi$, 
	\begin{align*}
	\psi(y+x)\leq 0\ \mbox{ for } \ x\in [l_2-B,l_2], \ y\in [l_2-x,B],
	\end{align*}
	which indicates
	\begin{align*}
	\int_{-l_2}^{l_2} \td J(x-y)\psi(y)\rd y\geq \int_{-B}^{B} \td J(y)\psi(y+x)\rd y\geq (1-\epsilon) \psi(x).
	\end{align*}
	The proof is now complete.	
\end{proof}

\begin{lemma}\label{lemma7.5}
If the conditions in Theorem \ref{prop6.1} are satisfied and ${\gamma}=2$, then there exits $C>0$ such that 
	\begin{align}\label{7.13a}
	h(t)\geq Ct\ln t \mbox{ for } t\gg1.
	\end{align}
\end{lemma}
\begin{proof}
	For fixed $\beta\in (0,1)$, define 
	\[
	\begin{cases}
	\underline h(t):=K_1(t+\theta)\ln (t+\theta), & t\geq 0,\\
	\dd	\underline U(t,x):=K_2\min\lf\{1, \frac{\underline h(t)-|x|}{(t+\theta)^{\beta}}\rr\}\Theta, & t\geq 0,\ x\in [-\underline h(t), \underline h(t)],
	\end{cases}
	\]
	for   constants $\theta\gg 1$ and  $1\gg K_1>0, 1\gg K_2>0$ to be determined, where $\Theta$ is given in Lemma \ref{lemma2.1a}. 
	Obviously, for any $t> 0$, the function $\partial_t \underline U(t,x)$ exists for $x\in [-\underline h(t), \underline h(t)]$  except 
	when $|x|=\underline h(t)-(t+\theta)^{\beta}$. However, the one-sided partial derivates  $\partial_t\underline U(t\pm 0, x)$ always exist.

	{\bf Step 1.}	 We show that by choosing $\theta$ and $K_1, K_2$ suitably,
	\begin{align}\label{7.11}
	&\underline h'(t)\leq  \sum_{i=1}^{m_0} \mu_i \int_{-\underline h(t)}^{\underline h(t)} \int_{\underline h(t)}^{+\yy} J_i(x-y) \underline u_i(t,x)\rd y\rd x&& \mbox{ for }  t> 0,\\
	&-\underline h'(t)\geq -\sum_{i=1}^{m_0}  \mu_i\int_{-\underline h(t)}^{\underline h(t)} \int_{-\yy}^{-\underline h(t)} J_i(x-y) \underline u_i(t,x)\rd y\rd x&& \mbox{ for }   t> 0.\label{7.11-a}
	\end{align}
	Since $\underline U(t,x)=\underline U(t,-x)$ and ${\bf J}(x)={\bf J}(-x)$, we see that \eqref{7.11-a} follows from \eqref{7.11}.
	
	By elementary calculations and \eqref{8.1}, we have
	\begin{align*}
	&\sum_{i=1}^{m_0}\mu_i \int_{-\underline h(t)}^{\underline h(t)} \int_{\underline h(t)}^{+\yy} J_i(x-y) \underline u_i(t,x)\rd y\rd x\\
	\geq& \sum_{i=1}^{m_0}\mu_i \int_{0}^{\underline h(t)-(t+\theta)^\beta} \int_{\underline h(t)}^{+\yy} J_i(x-y) \underline u_i(t,x)\rd y\rd x\\
	=& \sum_{i=1}^{m_0} \mu_i K_2\theta_i \int_{-\underline h(t)}^{-(t+\theta)^\beta}  \int_{0}^{+\yy} J_i(x-y) \rd y\rd x= \sum_{i=1}^{m_0} \mu_i K_2\theta_i \int_{(t+\theta)^\beta}^{\underline h(t)}  \int_{x}^{+\yy} J_i(y) \rd y\rd x\\
	=&\sum_{i=1}^{m_0} \mu_i K_2\theta_i \lf(\int_{(t+\theta)^\beta}^{\underline h(t)}\int_{(t+\theta)^\beta}^y+\int_{\underline h(t)}^{\yy}\int_{(t+\theta)^\beta}^{\underline h}\rr) J_i(y) \rd x\rd y\\
	\geq& \sum_{i=1}^{m_0} \mu_i K_2\theta_i \int_{(t+\theta)^\beta}^{\underline h(t)}\int_{(t+\theta)^\beta}^yJ_i(y) \rd x\rd y\geq \sum_{i=1}^{m_0} \mu_i C_1 K_2\theta_i \int_{(t+\theta)^\beta}^{\underline h(t)}\frac{y-(t+\theta)^\beta}{y^{2}+1} \rd y\\
	\geq &\sum_{i=1}^{m_0} \mu_i C_1 K_2\theta_i \int_{(t+\theta)^\beta}^{\underline h(t)}\frac{y-(t+\theta)^\beta}{2y^2}\rd y\\
	=&\sum_{i=1}^{m_0} \mu_i C_1 K_2\theta_i\frac 12 \lf(\ln \underline h(t)-\beta\ln (t+\theta)+ \frac{(t+\theta)^\beta}{\underline h(t)}-1\rr)\\
	\geq& \sum_{i=1}^{m_0} \mu_i C_1 K_2\theta_i\frac 12 \lf(\ln \underline h(t)-\beta\ln (t+\theta)-1\rr)\\
	=&\sum_{i=1}^{m_0} \mu_i C_1 K_2\theta_i\frac 12 \lf(\ln K_1+ \ln (t+\theta)+\ln (\ln( t+\theta))-\beta\ln( t+\theta)-1\rr)\\
	\geq&\sum_{i=1}^{m_0}\frac{ \mu_i C_1 K_2\theta_i(1-\beta)}{2} [ \ln (t+\theta)+1]\geq K_1\ln (t+\theta)+K_1=\underline h'(t)
	\end{align*}
	if
	\begin{equation}\label{7.14a}
	\begin{cases}
	\ln (\ln\theta)\geq -\ln K_1+2\\[2mm]
	\dd K_1\leq K_2\sum_{i=1}^{m_0}\frac{\mu_i C_1 \theta_i(1-\beta)}{2},
	\end{cases}
	\end{equation}
	which then finishes the proof of Step 1.
	
	{\bf Step 2.} We show that	 by choosing $K_1, K_2$ and $\theta$ suitably,  for $t>0$ and $x\in (-\underline h(t),\underline h(t))$, 
	\begin{align}\label{7.14}
	\underline U_t(t,x)\preceq  &D \circ\int_{-\underline h(t)}^{\underline h(t)}  {\bf J}(x-y) \circ\underline U(t,y)\rd y -D\circ\underline U(t,x)+F(\underline U(t,x)).
	\end{align}
	
	From the definition of $\underline U$, for $t>0$,
	\begin{align*}
	&\underline U_t(t,x)=K_1K_2\frac{(1-\beta)\ln (t+\theta)+1}{(t+\theta)^{\beta}} \Theta+\frac{K_2\beta |x|}{(t+\theta)^{1+\beta}}\Theta,\ \ \ \underline h(t)-(t+\theta)^\beta<|x|\leq \underline h(t),\\
	&\underline U_t(t,x)={\bf 0}, \ |x|< \underline h(t)-(t+\theta)^\beta.
	\end{align*}
	
	{\bf Claim 1}. For $x\in [-\underline h(t),-\underline h(t)+(t+\theta)^\beta]\cup [\underline h(t)-(t+\theta)^\beta,\underline h(t)]$ and large $\theta$,
	\begin{align}\label{7.15a}
	\int_{-\underline h(t)}^{\underline h(t)}  {\bf J}(x-y)\circ \underline U(t,y)\rd y\succeq \frac{\tilde C_1K_2\beta\ln (t+\theta)}{4(t+\theta)^\beta}\Theta,
	\end{align}
	where $\tilde C_1>0$ is given by \eqref{J-8.1}.  
		
	A simple calculation yields, for $x\in [\underline h(t)-(t+\theta)^\beta,\underline h(t)]$,
	\begin{align*}
	&\int_{-\underline h(t)}^{\underline h(t)}  {\bf J}(x-y) \circ\underline U(t,y)\rd y
	\succeq K_2\Theta\circ \int_{\underline h(t)-(t+\theta)^\beta}^{\underline h(t)}  {\bf J}(x-y) \frac{\underline h-y}{(t+\theta)^\beta}\rd y\\
	=&\frac{K_2\Theta}{(t+\theta)^\beta}\circ \int_{\underline h(t)-(t+\theta)^\beta-x}^{\underline h(t)-x}  {\bf J}(y) [\underline h(t)-(y+x)]\rd y.
	\end{align*}
	Hence, for $x\in [\underline h(t)-\frac 34(t+\theta)^\beta,\underline h(t)]$, by simple calculations and \eqref{J-8.1},
	\begin{align*}
	&\int_{-\underline h(t)}^{\underline h(t)}  {\bf J}(x-y)\circ \underline U(t,y)\rd y
	\succeq \frac{K_2\Theta}{(t+\theta)^\beta}\circ \int_{-(t+\theta)^\beta/4}^{0}  {\bf J}(y) (-y)\rd y\\
	=&\frac{K_2\Theta}{(t+\theta)^\beta}\circ\int_0^{(t+\theta)^\beta/4}  {\bf J}(y) y\rd y
	\succeq \frac{\td C_1K_2\Theta}{(t+\theta)^\beta} \int_0^{(t+\theta)^\beta/4} \frac{y}{y^2+1} \rd y\\
	\succeq&\frac{\td C_1K_2\Theta}{2(t+\theta)^\beta}\int_1^{(t+\theta)^\beta/4} y^{-1} \rd y
	=\frac{\td C_1K_2\Theta}{2(t+\theta)^\beta}[\beta\ln (t+\theta)-\ln 4]\\
	\succeq &  \frac{\td C_1K_2\beta\ln (t+\theta)}{4(t+\theta)^\beta}\Theta
	\end{align*}
	if 
	\begin{align}\label{7.17b}
	\frac\beta 2\ln \theta\geq \ln 4 .
	\end{align}
	
	And for $x\in [\underline h(t)-(t+\theta)^\beta,\underline h(t)-\frac 34(t+\theta)^\beta]$, 
	\begin{align*}
	&\int_{-\underline h(t)}^{\underline h(t)}  {\bf J}(x-y)\circ \underline U(t,y)\rd y
	\succeq \frac{K_2\Theta}{(t+\theta)^\beta}\circ \int_{0}^{3(t+\theta)^\beta/4}  {\bf J}(y) [\underline h(t)-(y+x)]\rd y\\
	\succeq & \frac{K_2\Theta}{(t+\theta)^\beta}\circ \int_{0}^{(t+\theta)^\beta/4}  {\bf J}(y) y\rd y \succeq \frac{\td C_1K_2\beta\ln (t+\theta)}{4(t+\theta)^\beta}\Theta.
	\end{align*}
This proves  \eqref{7.15a} for $x\in [\underline h(t)-(t+\theta)^\beta,\underline h(t)]$. 

For $x\in [-\underline h(t),-\underline h(t)+(t+\theta)^\beta]$,  \eqref{7.12} also holds since both ${\bf J}(x)$ and $\underline U(t,x)$ are
even in $x$. Claim 1 is thus proved.

\smallskip

	{\bf Claim 2}. We can choose small $K_2$ and large $\theta$ such that, for $x\in [-\underline h(t),\underline h(t)]$,
	\begin{align}\label{7.20a}
	D\circ\!\! \int_{-\underline h(t)}^{\underline h(t)}\!\!  \mathbf{J}(x-y)\!\circ\! \underline U(t,y)\rd y\! -\!D\!\circ\! \underline U(t,x)\!+\!F(\underline U(t,x))\!\succeq\! F_* \int_{-\underline h(t)}^{\underline h(t)}\!  \mathbf{J}(x-y) \circ\underline U(t,y)\rd y
	\end{align}
	for some $F_*>0$. 
	
	For small $K_2>0$,  from ${\bf 0}\preceq \underline U\preceq K_2\Theta$ and the definition of $\Theta$ in Lemma \ref{lemma2.1a}, we have
	\begin{align*}
	F(\underline U(t,x))=&\underline U(t,x)\Big([\nabla F({\bf 0})]^T+o(1){\bf I}_m\Big)\\
	=&K_2\min\lf\{1, \frac{\underline h(t)-|x|}{(t+\theta)^{\beta}}\rr\}\Theta \Big([\nabla F(0)]^T+o(1){\bf I}_m\Big)
	\\
	\succeq & K_2\min\lf\{1, \frac{\underline h(t)-|x|}{(t+\theta)^{\beta}}\rr\}\frac{3}{4}\lambda_1\Theta =\frac{3}{4}\lambda_1 \underline U(t,x),
	\end{align*}
	where $\lambda_1>0$ is given by Lemma \ref{lemma2.1a}. 
	 
	 For large $\theta$ and $t\geq 0$, we have
	\begin{align}\label{7.21a}
	\underline h(t)-(t+\theta)^{\beta}\geq \theta^\beta(K_1\theta^{1-\beta} \ln \theta-1)\geq \theta^\beta,\ \  (t+\theta)^{\beta}\geq \theta^{\beta}.
	\end{align}
	Hence, by  \eqref{7.11a}, there is large $L_1>0$ such that, for  $\theta^\beta>L_1$ we have
	\begin{align*}
	D \circ\int_{-\underline h(t)}^{\underline h(t)}  \mathbf{J}(x-y) \circ\underline U(t,y)\rd y +\frac{\lambda_1}{4} \underline  U(t,x)\succeq D\circ\underline U(t,x)\ \mbox{ for } \ x\in [-\underline h(t),\underline h(t)].
	\end{align*}
Therefore  \eqref{7.20a} holds with $F_*=\lambda_1/2$.

	Applying \eqref{7.15a} and \eqref{7.20a}, we have, for $x\in [-\underline h(t),-\underline h(t)+(t+\theta)^\beta]\cup [\underline h(t)-(t+\theta)^\beta,\underline h(t)]$,
	\begin{align*}
	&D \int_{-\underline h(t)}^{\underline h(t)}  \mathbf{J}(x-y)\circ \underline U(t,y)\rd y -\underline U(t,x)+F(\underline U(t,x))\\
	\succeq& \frac{F_*\td C_1K_2\beta\ln (t+\theta)}{4(t+\theta)^\beta}\Theta
	\succeq K_1K_2\frac{\ln (t+\theta)+1}{(t+\theta)^{\beta}} 
	\Theta\\
	=&\lf[K_1K_2\frac{(1-\beta)\ln (t+\theta)+1}{(t+\theta)^{\beta}} +\frac{K_2\beta \underline h(t)}{(t+\theta)^{1+\beta}}\right]
	\Theta\\
	\succeq&\lf[\frac{K_1K_2(1-\beta)\ln (t+\theta)+K_1K_2}{(t+\theta)^{\beta}} +\frac{K_2\beta |x|}{(t+\theta)^{1+\beta}}\right]
	\Theta\\
	=&\ \underline U_t(t,x)
	\end{align*}
	if apart from the earlier requirements, we further have
	\begin{align}\label{7.22}
	\ln \theta>2 \mbox{ and } K_1\leq \frac{F_*\td C_1\beta}{2}.
	\end{align}
	For $|x|< \underline h(t)-(t+\theta)^\beta$,
	\begin{align*}
	&D\circ \int_{-\underline h(t)}^{\underline h(t)}  \mathbf{J}(x-y) \circ\underline U(t,y)\rd y -D\circ\underline U(t,x)+F(\underline U(t,x))\\&\succeq F_* \int_{-\underline h(t)}^{\underline h(t)}  \mathbf{J}(x-y) \circ\underline U(t,y)\rd y
	\succeq {\bf 0}=\underline U_t(t,x).
	\end{align*}
	Thus \eqref{7.14} holds. 
	(Let us stress that it is possible to find $K_1$, $K_2$ and large $\theta$ such that  \eqref{7.14a}, \eqref{7.17b}, \eqref{7.21a} and \eqref{7.22} hold simultaneously.) 
	
	{\bf Step 3.} We finally prove \eqref{7.13a}.
	
	Clearly,  $\underline U(t,\pm \underline h(t))=0$ for $t\geq 0$.  Since spreading happens for $(U,g,h)$ and $K_2>0$ is small, there is a large constant $t_0>0$ such that 
	\begin{align*}
	&[-\underline h(0),\underline h(0)]\subset [g(t_0)/2,h(t_0)/2],\\
	&\underline U(0,x)\preceq K_2\Theta\preceq U(t_0,x)\ \mbox{ for } \ x\in [-\underline h(0),\underline h(0)].
	\end{align*}
	By Remark \ref{cp-rmk} and  Lemma \ref{lemma3.2}, we obtain
	\begin{align*}
	&[-\underline h(t),\underline h(t)]\subset [g(t+t_0),h(t+t_0)], \ &&t\geq 0,\\
	& \underline U(t,x)\preceq U(t+t_0,x),&&t\geq 0,\ x\in [-\underline h(t),\underline h(t)].
	\end{align*}
	Thus \eqref{7.13a} holds. 
	This completes the proof of the lemma.	
\end{proof}

\end{document}